\documentclass[review]{elsarticle}
\usepackage{srcltx}
\usepackage{eurosym}
\usepackage{mathtools}
\usepackage{amsmath}
\usepackage{amsfonts}
\usepackage{amssymb}
\usepackage{amsthm}
\usepackage{graphicx}
\usepackage{mathrsfs}
\usepackage{xcolor}
\usepackage{exscale}
\usepackage{latexsym}	
\usepackage[plainpages=true,pdfpagelabels,hypertexnames=true,colorlinks=true,pdfstartview=FitV,linkcolor=blue,citecolor=blue,urlcolor=blue,unicode]{hyperref}
\PassOptionsToPackage{unicode}{hyperref}
\PassOptionsToPackage{naturalnames}{hyperref}
\usepackage{enumerate}
\usepackage[shortlabels]{enumitem}
\usepackage{bookmark}
\usepackage{wasysym}
\usepackage{esint}
\def\Yint#1{\mathchoice
    {\YYint\displaystyle\textstyle{#1}}%
    {\YYint\textstyle\scriptstyle{#1}}%
    {\YYint\scriptstyle\scriptscriptstyle{#1}}%
    {\YYint\scriptscriptstyle\scriptscriptstyle{#1}}%
      \!\iint}
\def\YYint#1#2#3{{\setbox0=\hbox{$#1{#2#3}{\iint}$}
    \vcenter{\hbox{$#2#3$}}\kern-.50\wd0}}
\def\longdash{-\mkern-9.5mu-} 
\def\tiltlongdash{\rotatebox[origin=c]{18}{$\longdash$}}
\def\fiint{\Yint\tiltlongdash}

\def\Xint#1{\mathchoice
    {\XXint\displaystyle\textstyle{#1}}%
    {\XXint\textstyle\scriptstyle{#1}}%
    {\XXint\scriptstyle\scriptscriptstyle{#1}}%
    {\XXint\scriptscriptstyle\scriptscriptstyle{#1}}%
      \!\int}
\def\XXint#1#2#3{{\setbox0=\hbox{$#1{#2#3}{\int}$}
    \vcenter{\hbox{$#2#3$}}\kern-.50\wd0}}
\def\hlongdash{-\mkern-13.5mu-}
\def\tilthlongdash{\rotatebox[origin=c]{18}{$\hlongdash$}}
\def\hint{\Xint\tilthlongdash}
\makeatletter
\def\namedlabel#1#2{\begingroup
   \def\@currentlabel{#2}%
   \label{#1}\endgroup
}
\makeatother
\makeatletter
\newcommand{\rmh}[1]{\mathpalette{\raisem@th{#1}}}
\newcommand{\raisem@th}[3]{\hspace*{-1pt}\raisebox{#1}{$#2#3$}}
\makeatother

\newcommand{\lsb}[2]{#1_{\rmh{-3pt}{#2}}}
\newcommand{\lsbo}[2]{#1_{\rmh{-1pt}{#2}}}
\newcommand{\lsbt}[2]{#1_{\rmh{-2pt}{#2}}}

\newcounter{desccount}
\newcommand{\descitem}[2]{%
  \item[(#1)] \refstepcounter{desccount}\label{#2}
}
\newcommand{\descref}[1]{\hyperref[#1]{\textcolor{black}{(}\textcolor{blue}{\bf #1}\textcolor{black}{)}}}

\usepackage[ddmmyyyy]{datetime}

\makeatletter
\newcommand{\vo}{\vec{o}\@ifnextchar{^}{\,}{}}
\makeatother

\numberwithin{equation}{section}
\everymath{\displaystyle}

\newtheorem{theorem}{Theorem}[section]
\newtheorem{lemma}[theorem]{Lemma}
\newtheorem{corollary}[theorem]{Corollary}

\newtheorem{definition}[theorem]{Definition}
\newtheorem{remark}[theorem]{Remark}        
\numberwithin{equation}{section}

\def\al{\alpha}

\def\be{\beta}
\def\ga{\gamma}
\def\de{\delta}
\def\ve{\varepsilon}
\def\vt{\vartheta}

\def\th{\theta}
\def\la{\lambda}

\def\ka{\kappa}

\def\sig{\sigma}

\def\om{\omega}
\def\pa{\partial}

\def\Th{\Theta}

\def\Om{\Omega}
\def\La{\Lambda}
\def\aa{\mathcal{A}}

\newcommand{\tO}{\tilde{\Om}}
\newcommand{\tQ}{\tilde{Q}}
\newcommand{\tB}{\tilde{B}}

\newcommand{\tm}{I}

\newcommand{\mm}{\mathcal{M}}

\newcommand{\RR}{\mathbb{R}}

\newcommand{\NN}{\mathbb{N}}

\newcommand{\redref}[2]{\texorpdfstring{\protect\hyperlink{#1}{\textcolor{black}{(}\textcolor{red}{#2}\textcolor{black}{)}}}{}}
\newcommand{\redlabel}[2]{\hypertarget{#1}{\textcolor{black}{(}\textcolor{red}{#2}\textcolor{black}{)}}}

\newcommand{\iprod}[2]{\langle #1 ,  #2\rangle}

\newcommand{\lbr}[1][(]{\left#1}
\newcommand{\rbr}[1][)]{\right#1}

\newcommand{\cpt}[1][p]{\capacity_{1,#1}}

\newcommand{\txt}[1]{\qquad \text{#1} \quad}

\newcommand{\tal}{\tilde{\al}}

\DeclareMathOperator{\dv}{div}
\DeclareMathOperator{\capacity}{cap}

\DeclareMathOperator{\spt}{spt}

\newcommand{\cac}{C}

\def\tal{\tilde{\alpha}}

\def\tga{\tilde{\gamma}}

\newcommand{\tq}{\tilde{q}}
\newcommand{\tp}{\tilde{p}}
\newcommand{\tr}{\tilde{r}}

\newcommand{\tu}{\tilde{u}}
\newcommand{\tuh}{\tilde{u}_h}

\newcommand{\tk}{\tilde{k}}
\newcommand{\tz}{\tilde{z}}
\newcommand{\mfz}{\mathfrak{z}}

\newcommand{\mch}{\mathcal{H}}

\newcommand{\mcc}{\mathcal{C}}

\newcommand{\mcZ}{\mathcal{Z}}

\newcommand{\avgs}[2]{\lsbo{\lbr #2 \rbr}{#1}}
\newcommand{\avgsnoleft}[2]{\lsbo{( #2 )}{#1}}

\newcommand{\ddt}[1]{\frac{d#1}{dt}}
\newcommand{\dds}[1]{\frac{d#1}{ds}}

\newcommand{\vlh}{\lsbt{v}{\la,h}}
\newcommand{\vl}{\lsbt{v}{\la}}

\newcommand{\elam}{\lsbo{E}{\lambda}}

\newcommand{\vwspace}{L^{p-\be}(-T,T; W_0^{1,p-\be}(\Om))}

\newcommand{\pard}[1][\la]{d_{#1}}

\newcommand{\lamot}{\La_0,\La_1}

\newcommand{\tTh}{{\Upsilon}}
\newcounter{whitney}
\refstepcounter{whitney}

\newcounter{ineqcounter}
\refstepcounter{ineqcounter}

\begin{document}

\begin{frontmatter}

\title{Boundary higher integrability for very weak solutions of quasilinear parabolic equations.}

\author[myaddress]{Karthik Adimurthi\corref{mycorrespondingauthor}\tnoteref{thanksfirstauthor}}
\cortext[mycorrespondingauthor]{Corresponding author}
\ead{karthikaditi@gmail.com and kadimurthi@snu.ac.kr}
\tnotetext[thanksfirstauthor]{Supported by the National Research Foundation of Korea grant NRF-2015R1A2A1A15053024.}

\author[myaddress,myaddresstwo]{Sun-Sig Byun\tnoteref{thankssecondauthor}}
\ead{byun@snu.ac.kr}
\tnotetext[thankssecondauthor]{Supported by the National Research Foundation of Korea grant  NRF-2015R1A4A1041675. }

\address[myaddress]{Department of Mathematical Sciences, Seoul National University, Seoul 08826, Korea.}
\address[myaddresstwo]{Research Institute of Mathematics, Seoul National University, Seoul 08826, Korea.}

\begin{abstract}
We prove boundary higher integrability for the (spatial) gradient of \emph{very weak} solutions of quasilinear parabolic equations of the form $$u_t - \text{div}\,\mathcal{A}(x,t, \nabla u)=0 \quad \text{on} \ \Omega \times \RR,$$ where the non-linear structure $\text{div}\,\mathcal{A}(x, t,\nabla u)$ is modelled after the  $p$-Laplace operator.    To this end, we prove that the gradients satisfy a reverse H\"older inequality near the boundary. In order to do this, we construct a suitable test function which is Lipschitz continuous and  preserves the boundary values. \emph{These results are new even for linear parabolic equations on domains with smooth boundary and make no assumptions on the smoothness of $\mathcal{A}(x,t,\nabla u)$}. These results are also applicable for systems as well as higher order parabolic  equations.

\vspace*{0.5cm}
\noindent {\bf R\'esum\'e}\\
Nous montrons l'int\'egrabilit\'e des limites pour le gradient (spatial) de solutions tr\`es faibles d'\'equations paraboliques quasi-lin\'eaires de la forme
$$
u_t-\text{div}\mathcal{A}(x,t,\nabla u)=0 \quad \text{ sur } \Omega\times \mathbb{R},
$$
o\`u la structure non lin\'eaire $\text{div}\mathcal{A}(x,t,\nabla u)$ est model\'ee d'apr\`es l'op\'erateur
$p$-Laplace. Nous prouvons que les gradients satisfont l'in\'egalit\'e de H\"older inverse pr\`es du bord. Pour ce faire, nous construisons une fonction test appropri\'ee qui est lipschitzienne et pr\'eserve les valeurs limites. Ces r\'esultats sont nouveaux, m\^eme pour les \'equations  paraboliques lin\'eaires \`a fronti\`ere lisse et ne font aucune hypoth\`ese sur la r\'egularit\'e de $\text{div}\mathcal{A}(x,t,\nabla u)$ et sont \'egalement applicables aux syst\`emes ainsi qu'aux \'equations paraboliques d'ordre sup\'erieur. 
\end{abstract}

\begin{keyword}
quasilinear parabolic equations, boundary higher integrability, very weak solutions
\MSC[2010] 35K10\sep  99-00\sep 35K59\sep 35K65\sep 35K67.
\end{keyword}

\end{frontmatter}

\section{Introduction}
In this paper, we consider the boundary regularity  of quasilinear parabolic partial differential equations of the form 
\begin{equation}
\label{main}
 \left\{
 \begin{array}{ll}
u_t - \dv  \aa(x,t,\nabla u) = 0 & \text{on} \ \Om \times (-T,T) \\
u = 0 & \text{on} \ \partial \Om \times (-T,T),
 \end{array}
\right.
\end{equation}
modelled after the well studied $p$-Laplacian operator in a bounded domain $\Om\subset \RR^n$,  potentially with non-smooth boundary $\pa \Om$.

Weak solutions to \eqref{main} are in the space $u \in L^2(-T,T; L^2(\Om)) \cap L^p(-T,T; W_0^{1,p}(\Om))$ which allows one to use $u$ as a test function. But from the definition of weak solution, we see that the expression (see Definition \ref{very_weak_solution}) makes sense if we only assume $u \in L^2(-T,T; L^2(\Om)) \cap L^{s}(-T,T; W_0^{1,s}(\Om))$ for some $s > p-1$. But under this milder notion of solution called \emph{very weak solution}, we lose the ability to use $u$ as a test function. This difficulty was overcome in \cite{KL} in the interior case by constructing a suitable test function after modifying $u$ on a \emph{bad} set. The result was further extended to higher order systems in \cite{Bog} {\color{black}(see also \cite[Chapter 3]{bogelein2013regularity} where a more general form of Lipschitz truncation method was developed)}. In this paper, we show the boundary higher integrability for the gradient of very weak solutions to \eqref{main} in the sense of \cite{KL}. \emph{The results presented here are new even for linear parabolic equations on smooth domains.}

Regarding the assumption on the structure of the boundary of the domain, we assume that the complement of the domain $\Om^c$ satisfies a uniform capacity density condition (see Definition \ref{p_thick_domain}) . This condition is essentially optimal for our main results, as this condition is needed to ensure that Poincar\'e's inequality is applicable (see \cite[Chapter 11]{AH} for example).  The proofs are based on the Caccioppoli and Sobolev-Poincar\'e type inequalities as well as on the careful analysis of the associated strong Maximal function. As with regularity results concerning quasilinear parabolic equations, we make use of intrinsic scaling and covering techniques.

The first nonlinear parabolic higher integrability results apparently date back to a 1982 paper \cite{Gstr} in which the author studied the local higher integrability for systems of parabolic equations with quadratic growth conditions. However for more general quasilinear systems, the problem remained open until the papers \cite{KL, KL1} in which interior higher integrability was shown for both \emph{weak} and \emph{very weak} solutions to equations of the form considered in \eqref{main}.  For higher order quasilinear parabolic systems, interior higher integrability for very weak solutions was proved in \cite{Bog}. The boundary higher integrability for \emph{weak} solutions under the condition that $\Om^c$ satisfies a uniform capacity density condition was done in \cite{Par, Mik1}.  

We shall make precise all the notations in the following section, but let us first state the main theorem:
\begin{theorem}
 \label{main_theorem}
 Let $n\geq 2$ and $\frac{2n}{n+2} < p<\infty$ and $\Om$ be a bounded domain whose complement is uniformly $p$-thick with constants $r_0,b_0$, then there exists $\be_0 = \be_0(n,p,\lamot,b_0) \in (0,1)$ such that the following holds: For any $\be \in (0,\be_0]$ and any very weak solution $u \in L^2(-T,T; L^2(\Om)) \cap L^{p-\be}(-T,T; W_0^{1,p-\be}(\Om))$ solving \eqref{main} in $\Om \times (-T,T)$,  we have the following improved integrability $u \in L^2(-T,T; L^2(\Om)) \cap L^{p+\be}(-T,T; W_0^{1,p+\be}(\Om))$ along with the following qualitative estimates: for any $Q_{z_0}(\rho,s)$ with $z_0 \in (\overline{\Om} \times (-T,T))$, there holds
\begin{equation}
\label{main_est}
  \fiint_{Q_{z_0}(\rho,s)} |\nabla u|^{p} \ dz \apprle \lbr \fiint_{Q_{z_0}(2\rho, 2^2s)} \lbr |\nabla u| + \th\rbr^{p-\be} \ dz\rbr^{1+\frac{\be}{d}}  + \fiint_{Q_{z_0}(2\rho, 2^2 s)} (1 + \th^p) \ dz. 
 \end{equation}
where $\th$ is from \eqref{bound_b}, $C = C(n,p,b_0,\lamot)$  and 
\begin{equation}
 \label{de_d}
 d:= \left\{ \begin{array}{ll}
              2-\be & \text{if} \ p \geq 2 \\
              p-\be - \frac{(2-p)n}{2} & \text{if} \ p<2
             \end{array}\right.
\end{equation}

\end{theorem}

\begin{remark}
     In Theorem \ref{main_theorem}, we claim $u \in L^2(-T,T; L^2(\Om)) \cap L^{p+\be}(-T,T; W_0^{1,p+\be}(\Om))$ whereas the estimate in \eqref{main_est} only gives $u \in L^2(-T,T; L^2(\Om)) \cap L^{p}(-T,T; W_0^{1,p}(\Om))$. The current paper only proves higher integrability till the natural exponent as in \eqref{main_est}, but to reach $p+\be$, we can now make use of the results from \cite{Par,Mik1}. 
\end{remark}

The plan of the paper is as follows: In Section \ref{Preliminaries}, we shall collect all the assumptions and some well known, but useful lemmas. In Section \ref{Construction_of_test_function}, we shall give the construction of the test function using Lipschitz truncation and collect its properties. In Section \ref{Caccioppoli_inequality}, we will use results from Section \ref{Construction_of_test_function} to obtain Caccioppoli inequality and Reverse H\"older inequality and finally in Section \ref{main_estimates}, we will provide the proof of Theorem \ref{main_theorem}.

\subsection*{Acknowledgement} 
The authors would like to thank the anonymous referee for careful reading of the earlier version and giving valuable comments and suggestions.

\section{Preliminaries}\label{Preliminaries}

In this section, we shall collect some well known theories and estimates that will be used in the latter parts of the paper. 

\subsection{Variational \texorpdfstring{$p$}{pc}-Capacity}
Let $1<p<\infty$, then the \emph{variational $p$-capacity} of a compact set $K \Subset \Om$ is defined to be 
\begin{equation*}
 \label{var_cap_def}
 \cpt(K, \RR^n) = \inf \left\{ \int_{\RR^n} |\nabla \phi|^p \ dx \ : \ \phi \in C_c^{\infty}(\RR^n),\ \lsb{\chi}{K}(x) \leq \phi(x) \leq 1 \right\},
\end{equation*}
where $\lsb{\chi}{K} (x) = 1 $ for $x \in K$ and $\lsb{\chi}{K}(x) = 0$ for $x \notin K$. To define the variational $p$-capacity of an open set $O \subset \RR^n$, we take the supremum over the capacities of the compact sets belonging to $O$. The variational $p$-capacity of an arbitrary set $E \subset \RR^n$ is defined by taking the infimum over the capacities of the open sets containing $E$. For further details, see \cite{AH,HKM}.


Let us now introduce the capacity density conditions which we later impose on the complement of the domain. For higher integrability results, this condition is essentially sharp (see  \cite[Remark 3.3]{KK94} for the linear elliptic case and \cite{AP1} for the quasilinear elliptic case). 
\begin{definition}[Uniform $p$-thickness] \label{p_thick}
Let $\tO\subset\RR^n$ be a bounded domain. We say that the complement $\tO^c:=\RR^n\setminus \tO$ is 
uniformly $p$-thick for some $1< p \leq n$ with constants  $ \tilde{b}_0,\tilde{r}_0>0$, if  the inequality  
$$ {\rm cap}_{p} ( \overline{B_r(x)} \cap \tO^c, B_{2r}(x)) \geq  \tilde{b}_0\, {\rm cap}_{p} (\overline{B_{r}(x)}, B_{2r}(x))$$ 
holds for any $x \in \partial \tO$ and $r\in(0, \tr_0]$. 
\end{definition}

It is well-known that the class of  domains with uniform $p$-thick complements
is very large. They include all domains with Lipschitz boundaries or even those that satisfy a uniform exterior corkscrew condition, where the latter means that there exist constants $c_0, r_0>0$ such that for all $0<t\leq r_0$ and all $x\in \RR^n\setminus \Om$, there is $y\in B_t(x)$
such that $B_{t/c_0}(y)\subset \RR^n\setminus \Om$.

%
%

If we replace the capacity with the Lebesgue measure in Definition \ref{p_thick}, then we obtain a measure density condition. A set $E$ satisfying the measure density condition is uniformly $p$-thick for all $p>1$. If $p>n$, then every non-empty set is uniformly $p$-thick. The following lemma from \cite[Lemma 3.8]{Par} extends the capacity estimate in Definition \ref{p_thick} to make precise the notion of being \emph{uniformly $p$-thick}:
\begin{lemma}[\cite{Par}]
 Let $\tO$ be a bounded open set, and suppose that $\RR^n \setminus \tO$ is uniformly $p$-thick with constant $\tilde{b}_0,\tr_0$. Choose $y \in \tO$ such that $B_{\frac34 r}(y)\setminus \tO \neq \emptyset$, then there exists a constant $\tilde{b}_1 = \tilde{b}_1(b_0,r_0,n,p)>0$ such that 
 \[
  \cpt(\overline{B_{2r}(y)}\setminus \tO, B_{4r}(y) ) \geq \tilde{b}_1 \cpt( \overline{B_{2r}(y)}, B_{4r}(y) ). 
 \]
\end{lemma}

Following the definition of $p$-thickness, a simple consequence of H\"older's and Young's inequality gives the following result (for example, see \cite[Lemma 3.13]{Par} for the proof):
\begin{lemma}
 Let $1 < p \leq n$ be given and suppose  a set $E \subset \RR^n$ is uniformly $p$-thick with constants $\tilde{b}_0,\tr_0$. Then $E$ is uniformly $q$-thick for all $q \geq p$ with constants $\tilde{b}_1, \tilde{r}_1$.
\end{lemma}

We remark that the 
notion of $p$-thickness is self improving (see \cite{Le88} or  \cite{Anc,Mik} for the details):
\begin{theorem}[\cite{Le88}]
\label{self_improv_cap}
 Let $1 < p \leq n$ be given and suppose a set $E \subset \RR^n$ is uniformly $p$-thick with constants $\tilde{b}_0,\tr_0$. Then there exists an exponent $q = q(n,p,b_0)$ with $1<q<p$ for which $E$ is uniformly $q$-thick with constants $\tilde{b}_1, \tr_1$. 
\end{theorem}
We next state a generalized Sobolev-Poincar\'e's inequality which was originally obtained by V. Maz'ya \cite[Sec. 10.1.2]{Maz} (see also    \cite[Sec. 3.1]{KK94} and \cite[Corollary 8.2.7]{AH}). 
\begin{theorem} \label{sobolev-poincare} Let $B$ be a ball and $\phi \in W^{1,p}(B)$ be $p$-quasicontinuous, with $p>1$. Let $\kappa=n/(n-p)$ if $1<p<n$ and $\kappa=2$ if 
$p=n$. Then there exists a constant $c = c(n, p) > 0$ such that 
\begin{equation}\label{sob_poin_est}
 \left( \hint_B |\phi|^{\kappa p} \, dx\right)^{\frac{1}{\kappa p}} \leq c \left( 
\frac{1}{{\rm cap}_{1,p} (\overline{N(\phi)}, 2B ) } \int_B |\nabla \phi|^p \, dx\right)^{\frac{1}{p}},
\end{equation}
where $N(\phi)=\{x \in B: \phi (x) = 0 \}$.
\end{theorem}

%


\begin{lemma}
\label{interpolation_capacity_version_main}
Let $B_{\rho}$ be a ball and let   $1 < \sig < \infty$ ,  $1< \vartheta \leq n$  be any two exponents. Let $g $ be any function such that the following assumption holds: 
$$\cpt[\vartheta] (\overline{N(g)},2B) \geq c_0 \cpt[\vartheta] (\overline{B},2B) = c_0 \rho^{n-p}, \quad \text{where} \ N(g) = \{ x \in B : g(x) = 0\}. $$
Let $\tal \in (0,1)$ and $1<\tq,\tr<\infty$  be given such that the exponents are all related by the relations:
\begin{gather}
 \tal \sig < \tq \leq \vartheta^* = \frac{n \vartheta}{n-\vartheta}   \qquad \text{and} \qquad \tr:={\frac{(1-\tal) \sig \tq}{ \tq - \tal \sig}} \Longleftrightarrow \tq:= \frac{\tal \sig \tr}{\tr - (1-\tal)\sig}. \label{lemma1.7.1}
\end{gather}
which is equivalent to the relation
\begin{gather*}
     - \frac{n}{\sig} \leq \tal\lbr 1 - \frac{n}{\vartheta} \rbr  - (1-\tal) \frac{n}{\tr}. \nonumber 
\end{gather*}
Then the following interpolation type estimate holds:
\[
 \hint_B \lbr \frac{|g|}{\rho}\rbr^{\sig} \ dx  \leq C_{(n,\sig,\vartheta,\tr,c_0)} \lbr \hint_B|\nabla g|^{ \vartheta}  \ dx \rbr^{\frac{\tal \sig}{\vartheta}} \lbr \hint_B \lbr \frac{|g|}{\rho}\rbr^{\tr} \ dx \rbr^{\frac{(1-\tal)\sig}{\tr}}.
\]
\end{lemma}
\begin{proof}
Let us first consider the case $\vartheta<n$. The proof follows by applying Theorem \ref{sobolev-poincare} as follows: let $\tal \sig < \tq \leq \vartheta^* = \frac{n \vartheta}{n-\vartheta}$, then making use of \eqref{lemma1.7.1},  we obtain
\begin{equation*}
 \begin{array}{ll}
  \hint_B \lbr \frac{|g|}{\rho}\rbr^{\sig} \ dx 
  & \apprle \lbr \hint_B \lbr \frac{|g|}{\rho} \rbr^{ \vartheta^*} \ dx \rbr^{\frac{\tal \sig}{\vartheta^*}} \lbr \hint_B \lbr \frac{|g|}{\rho}\rbr^{\tr} \ dx \rbr^{\frac{(1-\tal)\sig}{\tr}} \\
  & \apprle \lbr \hint_B |\nabla g|^{ \vartheta} \ dx \rbr^{\frac{\tal \sig}{\vartheta}} \lbr \hint_B \lbr \frac{|g|}{\rho}\rbr^{\tr} \ dx \rbr^{\frac{(1-\tal)\sig}{\tr}}.
 \end{array}
\end{equation*}

%
%
%
 In the case $\vartheta =n$ (and also $\vartheta >n$), we observe that $W^{1,\vartheta}$ embeds into $L^t$ for all $t< \infty$ and the proof becomes simpler.
\end{proof}

\begin{definition}
 \label{p_thick_domain}
 In this paper, we shall assume that the domain $\Om$ is bounded and that it's complement $\Om^c$ is uniformly $p$-thick with constants $b_0,r_0$ in the sense of Definition of \ref{p_thick}.
\end{definition}

\subsection{Assumptions on the Nonlinear structure}

We shall now collect the assumptions on the nonlinear structure in \eqref{main}. 
We assume that $\aa(x,t,\nabla u)$ is an Carath\'eodory function, i.e., we have $(x,t) \mapsto \aa(x,t,\zeta)$  is measurable for every $\zeta \in \RR^n$ and 
$\zeta \mapsto \aa(x,t,\zeta)$ is continuous for almost every  $(x,t) \in \Om \times (-T,T)$.

We further assume that for a.e. $(x,t) \in \Om \times (-T,T)$ and for any $\zeta \in \RR^n$, there exists two given positive constants $\lamot$ such that  the following bounds are satisfied   by the nonlinear structures :
\begin{gather}
 \iprod{\aa(x,t,\zeta)}{\zeta} \geq \La_0 |\zeta |^p  - h_1  \txt{and} |\aa(x,t,\zeta)| \leq \La_1  |\zeta|^{p-1} + h_2, \label{abounded}
\end{gather}
where, the functions $h_1, h_2: \Om \times (-T,T) \mapsto \RR$ are assumed to be  measurable  with bounded norm 
\begin{equation}
 \label{bound_b}
 \th^p:=|h_1| + |h_2|^{\frac{p}{p-1}} \txt{and} \| \th\|_{L^{\hat{q}}(\Om \times (-T,T))} < \infty \quad \text{ for some  } \hat{q} \geq p+\be_0,
\end{equation}
where $\be_0$ is from Theorem \ref{main_theorem} (for example, $\be_0 = 1$ is admissible in  \eqref{bound_b}).  \emph{An important aspect of the estimates obtained in this paper is that we do not make any  assumptions regarding the smoothness of $\aa(x,t,\zeta)$ with respect to $x,t,\zeta$}.

As the basic sets for our estimates, we will use parabolic cylinders where the radii in space and time are coupled. This is due to the fact that in the case that $p \neq 2$, the size of the cylinders implicitly depend on the solution itself and cannot be taken independently. This difficulty extends to the problems dealing with \emph{very weak} solutions also. 

In what follows, we will always assume 
\begin{equation}
 \label{restriction_p}
 \frac{2n}{n+2} < p < \infty.
\end{equation}

\begin{remark}The restriction in \eqref{restriction_p} is necessary when dealing with parabolic problems because of the embedding $W^{1,\frac{2n}{n+2}} \hookrightarrow L^2$ due to which,  we will invariably have to deal with the $L^2$-norm of the solution which comes from the time-derivative. On the other hand, if we assume $u \in L^{r}(\Om \times (-T,T))$ for some $r\geq 1$ such that $n(p-2) + rp >0$ (see \cite[Chapter 5]{DiB1} for more on this), then we can obtain analogous result as to Theorem \ref{main_theorem}. This extension of Theorem \ref{main_theorem} to the case $1< p \leq \frac{2n}{n+2}$ is beyond the scope of this paper and will be pursued elsewhere. \end{remark}

\subsection{Function Spaces}\label{function_spaces}

Let $1\leq \vt < \infty$, then $W_0^{1,\vt}(\Om)$ denotes the standard Sobolev space which is the completion of $C_c^{\infty}(\Om)$ under the $\|\cdot\|_{W^{1,\vt}}$ norm. 

The parabolic space $L^{\vt}(-T,T; W^{1,\vt}(\Om))$ for any $\vt \in (1,\infty)$ is the collection of measurable functions $f(x,t)$ such that for almost every $t \in (-T,T)$, the function $x \mapsto f(x,t)$ belongs to $W^{1,\vt}(\Om)$ with the following norm being finite:
\[
 \| f\|_{L^{\vt}(-T,T; W^{1,\vt}(\Om)} := \lbr \int_{-T}^T \| u(\cdot, t) \|_{W^{1,\vt}(\Om)}^{\vt} \ dt \rbr^{\frac{1}{\vt}} < \infty.
\]

Analogously, the parabolic space $L^{\vt}(-T,T; W_0^{1,\vt}(\Om))$ is the collection of measurable functions $f(x,t)$ such that for almost every $t \in (-T,T)$, the function $x \mapsto f(x,t)$ belongs to $W_0^{1,\vt}(\Om)$.

\subsection{Notations}

We shall clarify the notation that will be used throughout the paper: 
\begin{enumerate}[(i)]
    
 \item\label{not1} We shall use $\nabla$ to denote derivatives with respect the space variable $x$.
\item\label{not2} We shall sometimes alternate between using $\ddt{f}$, $\pa_t f$ and $f'$ to denote the time derivative of a function $f$.
 \item\label{not3} We shall use $D$ to denote the derivative with respect to both the space variable $x$ and time variable $t$ in $\RR^{n+1}$. 
 \item\label{not4}  Let $z_0 = (x_0,t_0) \in \RR^{n+1}$ be a point and $\rho, s >0$ be two given parameters and let $\al \in (0,\infty)$. We shall use the following symbols to denote the following regions:
 \begin{equation}\label{notation_space_time}
\def\arraystretch{1.5}
 \begin{array}{ll}
 \tm_s(t_0) := (t_0 - s, t_0+s) \subset \RR,
& \qquad Q_{\rho,s}(z_0) := B_{\rho}(x_0) \times \tm_{s}(t_0) \subset \RR^{n+1},\\ 
 \al Q_{\rho,s}(z_0) := B_{\al \rho}(x_0) \times \tm_{\al^2s}(t_0)  \subset \RR^{n+1},
 &\qquad \mch_s(t_0) := \RR^n \times \tm_s(t_0) \subset \RR^{n+1},\\ 
 \alpha \mch_s(t_0) := \RR^n \times \tm_{\alpha^2 s}(t_0) \subset \RR^{n+1},
 & \qquad\mcc_{\rho}(x_0) := \Om \cap B_{\rho}(x_0) \times \RR \subset \RR^{n+1},\\
\Om_{\rho,s}(z_0) := \Om \cap B_{\rho}(x_0) \times \tm_s(t_0) \subset \RR^{n+1}, 
& \qquad\Om_{\rho}(x_0) := \Om \cap B_{\rho}(x_0) \subset \RR^{n}.
 \end{array}
\end{equation}

\item\label{not5} We shall use $\int$ to denote the integral with respect to either space variable or time variable and use $\iint$ to denote the integral with respect to both space and time variables simultaneously. 

Analogously, we will use $\fint$ and $\fiint$ to denote the average integrals as defined below: for any set $A \times B \subset \RR^n \times \RR$, we define
\begin{gather*}
\avgs{A}{f}:= \fint_A f(x) \ dx = \frac{1}{|A|} \int_A f(x) \ dx,\\
\avgs{A\times B}{f}:=\fiint_{A\times B} f(x,t) \ dx \ dt = \frac{1}{|A\times B|} \iint_{A\times B} f(x,t) \ dx \ dt.
\end{gather*}

\item\label{not6} Given any positive function $\mu$, we shall denote $\avgs{\mu}{f} := \int f\frac{\mu}{\|\mu\|_{L^1}}dm$ where the domain of integration is the domain of definition of $\mu$ and $dm$ denotes the associated measure. 

\item\label{not7} Given any $\la> 0$, we can convert $\RR^{n+1}$ into a metric space where the parabolic cylinders correspond to \emph{balls} under the parabolic metric given by:
\begin{equation}
 \label{parabolic_metric}
d_{\la} (z_1,z_2) := \max \left\{ |x_2-x_1|, \sqrt{\la^{p-2}|t_2-t_1|} \right\}.
\end{equation}

\item\label{not8} In what follows, $r_0$ and $b_0$ will denote the constants arising from the assumption that  $\Om^c$ is uniformly $p$-thick (see Definition \ref{p_thick_domain}).
\end{enumerate}

\subsection{Very weak solution}

There is a well known difficulty in defining the notion of solution for \eqref{main} due to a lack of time derivative of $u$. To overcome this, one can either use Steklov average or convolution in time. In this paper, we shall use the former approach (see also \cite[Page 20, Eqn (2.5)]{DiB1} for further details).

Let us first define Steklov average as follows: let $h \in (0, 2T)$ be any positive number, then we define
\begin{equation}\label{stek1}
  u_{h}(\cdot,t) := \left\{ \begin{array}{ll}
                              \hint_t^{t+h} u(\cdot, \tau) \ d\tau \quad & t\in (-T,T-h), \\
                              0 & \text{else}.
                             \end{array}\right.
 \end{equation}

\begin{definition}[Very weak solution] 
\label{very_weak_solution}
 
 Let $ \be \in (0,1)$ and $h \in (0,2T)$ be given and suppose $p-\be > 1$. We  then say $u \in L^2(-T,T; L^2(\Om)) \cap L^{p-\be}(-T,T; W_0^{1,p-\be}(\Om))$ is a very weak solution of \eqref{main} if for any $\phi \in W_0^{1,\frac{p-\be}{1-\be}}(\Om) \cap L^{\infty}(\Om)$, the following holds:
 \begin{equation}
 \label{def_weak_solution}
  \int_{\Om \times \{t\}} \frac{d [u]_{h}}{dt} \phi + \iprod{[\aa(x,t,\nabla u)]_{h}}{\nabla \phi} \ dx = 0 \txt{for any}-T < t < T-h.
 \end{equation}

\end{definition}

%

 \subsection{Maximal Function}

For any $f \in L^1(\RR^{n+1})$, let us now define the strong maximal function in $\RR^{n+1}$ as follows:
\begin{equation*}
 \label{par_max}
 \mm(|f|)(x,t) := \sup_{\tQ \ni(x,t)} \fiint_{\tQ} |f(y,s)| \ dy \ ds
\end{equation*}
where the supremum is  taken over all parabolic cylinders $\tQ_{a,b}$ with $a,b \in \RR^+$ such that $(x,t)\in \tQ_{a,b}$. An application of the Hardy-Littlewood maximal theorem in $x-$ and $t-$ directions shows that the Hardy-Littlewood maximal theorem still holds for this type of maximal function (see \cite[Lemma 7.9]{Gary} for details):
\begin{lemma}
\label{max_bound}
 If $f \in L^1(\RR^{n+1})$, then for any $\al >0 $, there holds
 \[
  |\{ z \in \RR^{n+1} : \mm(|f|)(z) > \al\}| \leq \frac{5^{n+2}}{\al} \|f\|_{L^1(\RR^{n+1})}.
 \]

 and if $f \in L^{\vartheta}(\RR^{n+1})$ for some $1 < \vartheta \leq \infty$, then there holds
 \[
  \| \mm(|f|) \|_{L^{\vartheta}(\RR^{n+1})} \leq C_{(n,\vartheta)} \| f \|_{L^{\vartheta}(\RR^{n+1})}.
 \]

\end{lemma}

 \subsection{A few well known lemmas}\label{few_well_known_lemmas}

 We shall recall the following well known lemmas that will be used throughout the paper. 
 The first one is a standard lemma regarding integral averages (for a proof in this setting, see for example \cite[Chapter 8.2]{bogelein2007thesis} for the details).
\begin{lemma}
\label{time_average}
Let $\ga >0$ be any fixed number and suppose  $[f]_h(x,t) : = \hint_{t-\ga h^2}^{t+\ga h^2} f(x,\tau) \ d\tau$. Then we have the following properties:
\begin{enumerate}[(i)]
 \item $[f]_h \rightarrow f$ a.e $(x,t) \in \RR^{n+1}$ as $h \searrow 0$, $[f]_h$ is continuous and  bounded  in time for a.e. $x \in \RR^n$.
 \item For any cylinder $Q_{r, \ga r^2} \subset \RR^{n+1}$ with $r >0$, there holds
\[
  \fiint_{Q_{r,\ga r^2}} [f]_h(x,t) \ dx \ dt \apprle_n \fiint_{Q_{r,\ga(r+h)^2}}f(x,t) \ dx \ dt. 
 \]
 \item The function $[f]_h(x,t)$ is differentiable with respect to $t \in \RR$, moreover $[f]_h(\cdot, t) \in C^{1}(\RR)$ for a.e. $x \in \RR^n$. 
 \end{enumerate}
\end{lemma}

The next Lemma we recall is a standard iteration lemma (see for example  \cite[Lemma 5.1]{Gia}):
\begin{lemma}
\label{iteration_lemma}	
 Let $0< \theta < 1$, $B \geq 0$, $A \geq 0$, $\tal >0$ and $0 < r < \rho < \infty$ and let $f \geq 0$ be a bounded measurable function satisfying 
 \[
  f(r_1) \leq \theta f(r_2) + A(r_2-r_1)^{-\tal} + B \qquad \forall \ r < r_1 < r_2 < \rho,
 \]
then there exists a constant $C=C(\tal,A,B)$ such that the following holds:
\[
 f(r) \leq C (A (\rho - r)^{-\tal} + B), \qquad \text{ for every }\ r < \rho.
\]

\end{lemma}

 We could not find the proof of the following version of the  Parabolic Poincar\'e Lemma in literature, hence for the sake of completeness, we also present the proof.
\begin{lemma}
 \label{lemma_crucial_1}
 Let $f \in L^{\vt} (-T,T; W^{1,\vt}(\Om))$ with $\vt \in (1,\infty)$  and suppose that $\mathcal{B}_{r} \Subset \Om$ be compactly contained ball of radius $r>0$. Let $\tm \subset (-T,T)$ be a time interval  and $\rho(x,t) \in L^1(\mathcal{B}_r \times \tm)$ be any positive function such that $\|\rho\|_{L^{\infty}(\mathcal{B}_r\times \tm)} \apprle_n \frac{\|\rho\|_{L^1(\mathcal{B}_r\times \tm)}}{|\mathcal{B}_r\times \tm|} $ and $\mu(x) \in C_c^{\infty}(\mathcal{B}_r)$ such that $\int_{\mathcal{B}_r} \mu(x) \ dx = 1$ with $|\mu| \leq \frac{C{(n)}}{r^n}$ and $|\nabla \mu| \leq  \frac{C(n)}{r^{n+1}}$. Then there holds
 \begin{equation*}
 \begin{array}{ll}
  \fiint_{\mathcal{B}_r \times \tm} \left|\frac{f - \avgs{\rho}{f}}{r}\right|^{\vt} \ dz & \apprle_{n,\vt} \fiint_{\mathcal{B}_r \times \tm} |\nabla f|^{\vt} \ dz + \sup_{t_1,t_2 \in \tm} \left| \frac{\avgs{\mu}{f}(t_2) - \avgs{\mu}{f}(t_1)}{r} \right|^{\vt},
  \end{array}
 \end{equation*}
where $\avgs{\rho}{f}:= \int_{\mathcal{B}_r\times \tm} f(z) \frac{\rho(z)}{\|\rho\|_{L^1(\mathcal{B}_r\times \tm)}} \ dz $ and $\avgs{\mu}{f}(t_i) := \int_{\mathcal{B}_r} f(x,t_i) \mu(x) \ dx$ for $i = 1,2$. 
\end{lemma}
\begin{proof}
 Let us first consider the case of $\rho(x,t) = \mu(x)\lsb{\chi}{\tm}(t)$. In this case, we get
 \begin{equation}
 \label{lem5.1.2}
  \begin{array}{rcl}
  \fiint_{\mathcal{B}_r\times \tm} \left| \frac{f(z)  - \avgs{\mu \times \lsb{\chi}{\tm}}{f}}{r} \right|^{\vt} \ dz 
 & \overset{\redlabel{lemma2.14.a}{a}}{\apprle} &  \fiint_{\mathcal{B}_r\times \tm} |\nabla f|^{\vt} \ dz + \sup_{t_1,t_2 \in {\tm}} \left|\frac{\avgs{\mu}{f}(t_2) - \avgs{\mu}{f}(t_1)}{r}\right|^{\vt}.
  \end{array}
 \end{equation}
Here to obtain $\redref{lemma2.14.a}{a}$ above, we made of the standard Poincar\'e's inequality.

For the general case, we observe that 
\begin{equation}
  \label{lem5.1.3}
  \fiint_{\mathcal{B}_r\times \tm} \left|\frac{f - \avgs{\rho}{f}}{r}\right|^{\vt} \ dz \apprle \fiint_{\mathcal{B}_r\times \tm} \left|\frac{f - \avgs{\mu \times \lsb{\chi}{\tm}}{f}}{r}\right|^{\vt} \ dz  + \fiint_{\mathcal{B}_r\times \tm} \left|\frac{\avgs{\rho}{f} - \avgs{\mu \times \lsb{\chi}{\tm}}{f}}{r}\right|^{\vt} \ dz.
 \end{equation}
The first term of \eqref{lem5.1.3} can be controlled as in \eqref{lem5.1.2} and  to control the second term, we observe that
\begin{equation*}
 \label{lem5.1.4}
 \begin{array}{ll}
 |\avgs{\rho}{f} - \avgs{\mu \times \lsb{\chi}{\tm}}{f}| & \leq \frac{\|\rho\|_{L^\infty(\mathcal{B}_r\times \tm)}}{\|\rho\|_{L^1(\mathcal{B}_r\times \tm)}} \iint_{\mathcal{B}_r\times \tm} |f - \avgs{\mu \times \lsb{\chi}{\tm}}{f}|\ dz  \apprle \fiint_{\mathcal{B}_r\times \tm} |f - \avgs{\mu \times \lsb{\chi}{\tm}}{f}|\ dz.\\
 \end{array}
\end{equation*}
This completes the proof of the Lemma. 
\end{proof}

We will use the following result which can be found in \cite[Theorem 3.1]{PG} (see also \cite{Prato}):
\begin{lemma}
\label{metric_lipschitz}
 Let $\la >0$ and  $\mcZ \subset \RR^{n+1}$ be given. For any $z \in \mcZ$ and $r>0$, let $Q_{r,\ga r^2}(z)$ be the parabolic cylinder centred  at $z$ with radius $r$.  Suppose  there exists a constant $C>0$ independent of $z$ and $r$ such that the following bound holds:
 \[
  \frac{1}{|Q_{r,\ga r^2}(z)\cap \mcZ|} \iint_{Q_{r,\ga r^2}(z)\cap \mcZ} \left|\frac{f(x,t) - \avgs{Q_{r,\ga r^2}(z)\cap \mcZ}{f}}{r}\right| \ dx \ dt \leq C \qquad \forall\  z \in \mcZ \text{ and } r>0.
 \]
 Then $f \in C^{0,1} (\mcZ)$. 
\end{lemma}

Let us also recall a parabolic version of Gehring's lemma (for example, see \cite[Lemma 6.4]{Bog}).
\textcolor{black}{
\begin{lemma}
\label{gehring_lemma}
 Let $\al_0 \geq 1$, $\ka \geq 1$, $\ve_0>0$, $p>1$ and $\be_0>0$ be given. Let $q$ be given such that  $1<p-\varepsilon_0 \leq q < p-2\be<p-\be$ for some $\be \in (0,\be_0)$.  Furthermore, for a cylinder $Q_2 = Q_{2,2^2}$, let  $f \in L^{p-\be}(Q_2)$ and $g \in L^{\tp}(Q_2)$ for some $\tp \geq p$ be given. Suppose for each $\la \geq \al_0$ and almost every $z \in Q_2$ with $f(z) >\la$, there exists a parabolic cylinder $Q = Q_{\rho,s}(z)\subset Q_2$ such that 
 \[
  \frac{\la^{p-\be}}{\ka} \leq \fiint_Q f^{p-\be}(x,t) \ dx \ dt \leq \ka \lbr \fiint_Q f^q(x,t) \ dx\ dt \rbr^{\frac{p-\be}{q}} + \ka \fiint_Q g^{p-\be}(x,t) \ dx \ dt \leq \ka^2 \la^{p-\be},
 \]
then there exists $\de_0 = \de_0(\ka,p,\be,q,\ve_0)$ and $C = C(\ka,p,\be,q,\ve_0)$, such that $f \in L^{p-\be + \de_1}(Q_2)$ with $\de_1 = \min \{ \de_0, \tp-p+\be\}$. This improved higher integrability comes with the following bound:
\[
 \iint_{Q_2} f^{p-\be + \de}(x,t) \ dx \ dt\apprle \al_0^{\de} \iint_{Q_2} f^{p-\be}(x,t) \ dx\ dt + \iint_{Q_2} g^{p-\be+\de}(x,t) \ dx\ dt \quad \text{for all} \ \de \in (0,\de_1]. 
\]
\end{lemma}}

\section{Construction of test function.}\label{Construction_of_test_function}

Since we are interested in estimates near the boundary, for the subsequent calculations, we will fix a parabolic cylinder of the form $Q:=Q_{\rho, s} (0,0) = B \times I$, where $(0,0) \in \pa \Om \times (-T,T)$. We will further assume 
\begin{equation}
 \label{def_s}
 s = \al_0^{2-p} \rho^2 \txt{for some} \al_0 >0 \ \text{to be chosen}.
\end{equation}
\emph{Since $(0,0)$ and $(\rho,s)$ are fixed, we shall henceforth simplify the notation in  \eqref{notation_space_time} by dropping obvious references to  $\rho$ and $s$.}

\begin{definition}
 \label{exponent_choice} 
 Since we assume that domain $\Om^c$ is uniformly $p$-thick with constants $r_0,b_0$ (see Definition \ref{p_thick_domain}), by the self improving property (see Theorem \ref{self_improv_cap}), there exists an $\ve_0 > 0$ such that $\Om^c$ is uniformly  $p-\varepsilon_0$ thick with constants $\tilde{b}_0,\tilde{r}_0$.  Let us now take $\be$ and $q$ such that the following holds:
 \begin{equation}
  \label{def_q}
  p-\varepsilon_0 < q \leq p-2\be < p-\be < p. 
 \end{equation}
Note that we will eventually obtain $\be_0 = \be_0(n,p,b_0,\lamot)\in (0,\ve_0/2)$  such that all the calculations will hold for any $\be \in (0,\be_0]$. 
\end{definition}

Consider the following cut off functions:
\begin{equation}
\label{cut_off_function}\begin{array}{c}
 \eta(x) \in C_c^{\infty} (8B) , \ \eta(x) \equiv 1 \ \text{for} \ x \in 4B,  \qquad  \zeta(t) \in C_c^{\infty} (8\tm) , \ \zeta(t) \equiv 1 \ \text{for} \ t \in 4\tm , \\
 |\eta(x)| \leq 1, \qquad |\eta(x)| \leq \frac{C(n)}{\rho}, \txt{and} |\zeta(t)| \leq 1, \qquad | \zeta'(t)| \leq \frac{C}{s}.
 \end{array}
\end{equation}

\begin{definition}
 Since we assume that there exists a very weak solution $u \in L^2(-T,T; L^2(\Om)) \cap L^{p-\be} (-T,T; W_0^{1,p-\be}(\Om))$, we can now multiply the solution by the cut off functions to get:
 \begin{equation}
  \label{def_u_tilde}
  \tuh(x,t) := [u]_h(x,t) \eta(x) \zeta(t).
 \end{equation}
Here $[u]_h$ denotes the usual Steklov average defined in \eqref{stek1}. 
\end{definition}

Let us define the following function with $\th$ as in \eqref{bound_b}:
\begin{equation}
 \label{def_g}
 g(x,t) := \max \left\{ \mm(|\nabla \tu|^q)^{\frac{1}{q}}(x,t), \mm((|\nabla u|+|\th|)^q\lsb{\chi}{8Q})^{\frac{1}{q}}(x,t)\right\} \quad \text{for any} \ (x,t) \in \RR^{n+1}.
\end{equation}
We have the following bound for $g(x,t)$:
\begin{lemma}
 \label{bound_g_x_t}
 Under the assumption that $\Om$ satisfies Definition \ref{p_thick_domain}, the following bound holds:
 \begin{equation*}
  \label{bound_g_x_t_1}
  \| g\|_{L^{p-\be} (\RR^{n+1})} \leq C_{(n,p,b_0)} \| (|\nabla u| + |\th|) \lsb{\chi}{8Q} \|_{L^{p-\be}(\RR^{n+1})}.
 \end{equation*}
\end{lemma}
\begin{proof}
 Making use of  Lemma \ref{max_bound}, we get
 \begin{equation*}
  \begin{array}{ll}
   \iint_{\RR^{n+1}} |g(x,t)|^{p-\be} \ dz & \apprle \iint_{\RR^{n+1} }\lbr \mm(|\nabla \tu|^q)^{\frac{p-\be}{q}}(z)+ \mm((|\nabla u|+|\th|)^q\lsb{\chi}{8Q})^{\frac{p-\be}{q}}(z)\rbr \ dz \\
   & \apprle \iint_{\RR^{n+1} }\lbr |\nabla \tu|^{p-\be}+ \lbr |\nabla u|+|\th|\rbr^{p-\be} \lsb{\chi}{8Q} \rbr \ dz.
  \end{array}
 \end{equation*}
 Since $\tu(x,t) = u(x,t) \eta(x) \zeta(t)$ and $\spt(\tu) \subset 8Q$, we use Theorem \ref{sobolev-poincare} to get
 \[
   \iint_{\RR^{n+1} } |\nabla \tu|^{p-\be}\ dz  \apprle \iint_{8Q}\lbr |\nabla u|^{p-\be}  + \left|\frac{u}{\rho}\right|^{p-\be} \rbr \ dz 
    \apprle \iint_{8Q} |\nabla u|^{p-\be}\ dz.
 \]
This completes the proof of the Lemma. 
\end{proof}

We now define the good set
\begin{equation*}
 \label{def_elam}
\elam:=  \left\{ (x,t) \in \RR^{n+1} : g(x,t) \leq \la \right\} \quad \text{for} \ \la > c_e\al_0. 
 \end{equation*}
 \emph{Note that $\al_0>0$  and $c_e$ are constants that will be quantified explicitly in Section \ref{Caccioppoli_inequality}. }

We now have the following parabolic Whitney type decomposition of $\elam^c$ which is taken from  \cite[Section 2.3]{Seb} {\color{black}(see also \cite[Chapter 3]{bogelein2013regularity} where this was proved in the stated form or \cite[Lemma 3.1]{diening2010existence} for details):}
\begin{lemma}
\label{whitney_decomposition}

Let $\ga := \la^{2-p}$, then there exists a $\ga$-parabolic Whitney covering  $\{Q_i(z_i)\}$ of $\elam^c$ in the following sense:
 \begin{description}
 \setcounter{whitney}{1}
  \descitem{W\thewhitney}{W7} $Q_j(z_j) = B_j(x_j) \times I_j(t_j)$ where $B_j(x_j) = B_{r_j}(x_j)$ and $I_j(t_j) = (t_j - \ga r_j^2, t_j + \ga r_j^2)$. 
  \stepcounter{whitney}\descitem{W\thewhitney}{W2} Since $\ga = \la^{2-p}$, recalling the notation from \eqref{parabolic_metric}, we have $d_{\la}(z_j,\elam) = 16r_j$.
  \stepcounter{whitney}\descitem{W\thewhitney}{W1} $\bigcup_j \frac12 Q_j(z_j) = \elam^c$.
  \stepcounter{whitney}\descitem{W\thewhitney}{W3} for all $j \in \NN$, we have $8Q_j \subset \elam^c$ and $16Q_j \cap \elam \neq \emptyset$.
  \stepcounter{whitney}\descitem{W\thewhitney}{W4} if $Q_j \cap Q_k \neq \emptyset$, then $\frac12 r_k \leq r_j \leq 2r_k$.
  \stepcounter{whitney}\descitem{W\thewhitney}{W5} $\frac14 Q_j \cap \frac14Q_k = \emptyset$ for all $j \neq k$.
  \stepcounter{whitney}\descitem{W\thewhitney}{W6} $\sum_j \lsb{\chi}{4Q_j}(z) \leq c(n)$ for all $z \in \elam^c$.
  \end{description}
  Subject to this Whitney covering, we have an associated partition of unity denoted by $\{ \om_j\}_{j\in\NN} \in C_c^{\infty}(\RR^{n+1})$ such that the following holds:
  \begin{description}
  \stepcounter{whitney}\descitem{W\thewhitney}{W8} $\lsb{\chi}{\frac12Q_j} \leq \om_j \leq \lsb{\chi}{\frac34Q_j}$.
  \stepcounter{whitney}\descitem{W\thewhitney}{W9} $\|\om_j\|_{\infty} + r_j \| \nabla \om_j\|_{\infty} + r_j^2 \| \nabla^2 \om_j\|_{\infty} + \la r_j^2 \| \pa_t \om_j\|_{\infty} \leq C(n)$.
  \end{description}
  For a fixed $k \in \NN$, let us define 
  \begin{equation*}\label{Ak}A_k := \left\{ j \in \NN: \frac34Q_k \cap \frac34Q_j \neq \emptyset\right\},\end{equation*} then we have
  \begin{description}
  \stepcounter{whitney}\descitem{W\thewhitney}{W10} Let $i \in \NN$ be given, then $\sum_{j \in A_i} \om_j(z) = 1$  for all $z \in \frac34Q_i$.
  \stepcounter{whitney}\descitem{W\thewhitney}{W11} Let $i \in \NN$ be given and let  $j \in A_i$, then $\max \{ |Q_j|, |Q_i|\} \leq C(n) |Q_j \cap Q_i|.$
  \stepcounter{whitney}\descitem{W\thewhitney}{W12}  Let $i \in \NN$ be given and let  $j \in A_i$, then $ \max \{ |Q_j|, |Q_i|\} \leq \left|\frac34Q_j \cap \frac34Q_i\right|.$
  \stepcounter{whitney}\descitem{W\thewhitney}{W13} For any $i \in \NN$, we have $\# A_i \leq c(n)$.
  \stepcounter{whitney}\descitem{W\thewhitney}{W14} Let $i \in \NN$ be given, then for any $j \in A_i$, we have $\frac34Q_j \subset 4Q_i$.
 \end{description}
\end{lemma}
%
%

Once we have the Whitney-type parabolic decomposition as in Lemma \ref{whitney_decomposition}, we can now define the following extension:
\begin{equation}
 \label{lipschitz_extension}
\vlh(z) := \tuh(z) - \sum_i \om_i(z) \lbr \tuh(z) - \tuh^i \rbr. 
\end{equation}
where 
\begin{equation}
\label{def_tuh_i}
 \tuh^i := \left\{ \begin{array}{ll}
                   \frac1{\|\om_i\|_{L^1(\frac34Q_i)}}\iint_{\frac34Q_i} \tuh(z) \om_i (z) \ dz & \text{if} \ \frac34Q_i \subset \Om_{8\rho}\times \RR = \mcc_{8\rho}, \\
                   0 & \text{else}.
                  \end{array}\right. 
\end{equation}

For each $\la > c_e \al_0$, we will show that $\vlh$ constructed in \eqref{lipschitz_extension} is Lipschitz continuous with respect to the parabolic metric $d_{\la}$ on $2\mch = \RR^n \times (-4s,4s)$. In order to do this, we shall make use of Lemma \ref{metric_lipschitz}. In order to apply Lemma \ref{metric_lipschitz} with $Y = 2\mch$, we will consider the following two cases, one where the cylinders are contained inside $4\mch = \RR^n \times (-16s,16s)$ and the other where the cylinders are not contained inside $4\mch$. 
In this regard, let us define the following:
\begin{gather}
 \Theta := \left\{ i \in \NN : \frac34Q_i \cap 2\mch \neq \emptyset\right\}, \label{theta} \\
 \Theta_1 := \left\{ i \in \Theta : 8Q_i \subset 4\mch \right\}, \label{theta_1} \\
 \Theta_2 := \left\{ i \in \Theta : 8Q_i \cap (4\mch)^c \neq \emptyset\right\}  = \Theta \setminus \Theta_1.\label{theta_2} 
\end{gather}


We will use the following important lemma throughout  the paper:
\begin{lemma}
\label{lemma_crucial_2}
 Let $\ga >0$ be given and suppose that $u \in L^2(-T,T;L^2(\Om)) \cap \vwspace$ with $0 \leq \be \leq \min\{1,p-1\}$ be a very weak solution of \eqref{main}. Let $\mathcal{B} \subset \Om$ be a compactly contained region and $(t_1,t_2) \subset (-T, T-\ga h^2)$ for some $h \in (0,T)$ be a time interval. Let $\phi(x) \in C_c^{\infty}(\mathcal{B})$, $\varphi(t) \in C_c^{\infty}(t_1,t_2)$ be two non-negative functions and $[u]_h$ be the Steklov average as defined in \eqref{stek1}. Then   the following estimate holds:
 \begin{equation*}
  \label{lemma_crucial_2_est}
  \begin{array}{ll}
  |\avgs{\phi}{{[u]_h\varphi}} (t_2) - \avgs{\phi}{{[u]_h\varphi}}(t_1)| & \leq \La_1 \|\nabla \phi\|_{L^{\infty}{(\mathcal{B})}} \|\varphi\|_{L^{\infty}(t_1,t_2)} \iint_{\mathcal{B} \times (t_1,t_2)} {[|\nabla u|^{p-1} + |\th|^{p-1}]_h} \ dz + \\
   & \qquad +  \|\phi\|_{L^{\infty}{(\mathcal{B})}} \|\varphi'\|_{L^{\infty}(t_1,t_2)} \iint_{\mathcal{B} \times (t_1,t_2)} |[u]_h| \ dz.
  \end{array}
 \end{equation*}
\end{lemma}
\begin{proof}
 Let us use $\phi(x)\varphi(t)$ as a test function in \eqref{def_weak_solution} to get 
 \[
  \int_{\Om\times\{t\}} \ddt{[u]_h} (x,t) \phi(x)\varphi(t) \ dx +  \iprod{[\aa(x,t,\nabla u)]_h}{\nabla \phi}(x,t)\varphi(t) \ dx =0
 \]
 Using the Fundamental theorem of calculus, we get
 \begin{equation*}
  \begin{array}{ll}
   |\avgs{\phi}{{[u]_h\varphi}} (t_2) - \avgs{\phi}{{[u]_h\varphi}}(t_1)|
   &= \left|\iint_{\mathcal{B} \times (t_1,t_2)} \frac{d}{dt} \lbr[(] [u]_h(x,t) \phi(x)\varphi(t) \rbr[)] \ dz \right|\\
   & \leq \left|\iint_{\mathcal{B} \times (t_1,t_2)} \iprod{[\aa(x,t,\nabla u)]_h}{\nabla \phi}(x,t) \varphi(t)\ dz \right| + \\
   & \qquad + \left| \iint_{\mathcal{B} \times (t_1,t_2)} [u]_h(x,t) \phi(x) \ddt{\varphi(t)} \ dz \right|  \\
   & \overset{\redlabel{lemma3.4.1}{a}}{\leq} C(\La_1,p)\|\nabla \phi\|_{L^{\infty}{(\mathcal{B})}} \|\varphi\|_{L^{\infty}(t_1,t_2)} \iint_{\mathcal{B} \times (t_1,t_2)} {[|\nabla u|^{p-1} + |\th|^{p-1}]_h} \ dz \\
   & \qquad + \|\phi\|_{L^{\infty}{(\mathcal{B})}} \|\varphi'\|_{L^{\infty}(t_1,t_2)} \iint_{\mathcal{B} \times (t_1,t_2)} |[u]_h| \ dz.
  \end{array}
 \end{equation*}
 To obtain \redref{lemma3.4.1}{a} above, we made use of \eqref{abounded} and  \eqref{bound_b} which completes the proof.
\end{proof}

\subsection{Properties of the test function}
In this subsection, we shall prove all the properties that the constructed function in \eqref{lipschitz_extension} satisfies. Recall that $\Om^c$ is uniformly $p$-thick with constants $b_0,r_0$. 
The first lemma gives a point-wise bound for $\vlh$ on $\elam^c$. 

\begin{lemma}
\label{lemma3.7}
 For any $z \in \elam^c$, we have
 \begin{equation*}
  \label{bound_v_l_h}
  |\vlh(z)| \leq C_{(n,p,q,\lamot,b_0)} \rho \la.
 \end{equation*}
\end{lemma}

\begin{proof}
From \eqref{lipschitz_extension},  for any  $z \in \elam^c$, we see that $\vlh(z) = \sum_j \om_j(z) \tuh^j$ with $\tuh^j\neq 0$ if and only if $\frac34Q_j \subset \mcc_{8\rho}$, which automatically implies $\vlh(z) = 0$ for all $z \in \mcc_{8\rho}^c$. 

In order to prove the Lemma, making use of \descref{W8}, we see that \eqref{bound_v_l_h} follows if the following holds:
\begin{equation}
\label{claim_bound}
|\tuh^j| \leq C_{(n,p,q,\lamot,b_0)}\  \rho \la,  \txt{for all}\ j \in \NN.
\end{equation}
%


We shall now proceed with proving \eqref{claim_bound}.    Since we only have to consider the case  $\frac34Q_j \subset \mcc_{8\rho}$, we have the following two sub-cases:
 \begin{description}[leftmargin=*]
  \item[Case $r_j \geq \rho$:] In this case, we observe that $8B \subset 32B_j$ which gives the following sequence of estimates:
  \begin{equation}
  \label{est_1}
   \begin{array}{ll}
    |\tuh^j| 
    & \overset{\redlabel{3.13.a}{a}}{\apprle} \rho \hint_{32\tm_j} \lbr \hint_{32B_j} \left| \frac{\tuh(x,t)}{r_j} \right|^q \ dx\rbr^{\frac1{q}} \ dt 
     \overset{\redlabel{3.13.b}{b}}{\apprle} \rho \lbr \fiint_{32Q_j}  \left| \nabla \tuh(x,t) \right|^q \ dx\ dt\rbr^{\frac1{q}}  
     \overset{\redlabel{3.13.c}{c}}{\apprle} \rho \la. 
   \end{array}
  \end{equation}

  In order to obtain \redref{3.13.a}{a}, we have used the fact that $r_j \leq \frac{32}{3} \rho$ since $\frac34Q_j \subset \mcc_{8\rho}$.  To get \redref{3.13.b}{b}, we  used Theorem \ref{sobolev-poincare} since $32B_j \cap \Om^c \neq \emptyset$. Finally to obtain \redref{3.13.c}{c}, we made use of \descref{W4}.
  
  \item[Case $r_{j} \leq \rho$:] Let us define the following constant $k_0 := \min\{ \tk_1,\tk_2\}$, where $\tk_1$ and $\tk_2$ satisfy
   \begin{equation}\label{def_k_0}
2^{\tk_1 - 1} r_j < \rho \leq 2^{\tk_1} r_j \txt{and} 2^{\tk_2 -1} Q_j \cap (\Om \times \RR)^c = \emptyset \ \text{  but  } \ 2^{\tk_2}Q_j \cap (\Om \times \RR)^c \neq \emptyset.    
   \end{equation}

Using \eqref{cut_off_function} and  triangle inequality, we get
  \begin{equation}
   \label{est_2} 
   \begin{array}{ll}
    |\tuh^j| 
    & \overset{\text{}}{\apprle} \sum_{m=0}^{k_0 -2} \lbr \avgs{2^m Q_j}{{[u]_h}} - \avgs{2^{m+1}Q_j}{{[u]_h}}\rbr + \avgs{2^{k_0-2}Q_j}{{[u]_h}} := \sum_{m=0}^{k_0-2}S_1^m + S_2.
   \end{array}
  \end{equation}
  \begin{description}[leftmargin=*]
   \item[Estimate for $S_1^m$ for $m \leq {k_0-2}$:]  In this case,  we see that  $2^{m+1}Q_j \subset \mcc_{8\rho}$. Thus applying Lemma \ref{lemma_crucial_1} for any $\mu \in C_c^{\infty}(B(2^{m+1}r_j, x_j))$ satisfying $| \mu(x)| \leq \frac{C(n)}{(2^{m+1}r_j)^{n}}$ and $|\nabla \mu(x)| \leq \frac{C(n)}{(2^{m+1}r_j)^{n+1}}$, we get
   \begin{equation}
   \label{S_1_1}
    \begin{array}{ll}
     S_1^m
     & \apprle  (2^{m+1}r_j) \lbr \fiint_{2^{m+1}Q_j} |[\nabla u]_h|^q \ dz + \sup_{t_1,t_2 \in {2^{m+1}I_j}} \left|\frac{\avgs{\mu}{{[u]_h}}(t_2)-\avgs{\mu}{{[u]_h}}(t_1)}{2^{m+1}r_j} \right|^q\rbr^{\frac1{q}}.
    \end{array}
   \end{equation}

Since $B(x_j,2^{m+1}r_j) \subset \Om$, we can apply Lemma \ref{lemma_crucial_2} with the test function $\phi(x)=\mu(x)$ and $\varphi(t) \equiv 1$, which  for any $t_1,t_2 \in \frac34 I_j$, gives
\begin{equation}
\label{S_1_2}
 \begin{array}{ll}
  |\avgs{\mu}{{[u]_h}}(t_2)-\avgs{\mu}{{[u]_h}}(t_1)| & \apprle  \|\nabla \mu \|_{L^{\infty}} \iint_{2^{m+1}Q_j} [|\nabla u|^{p-1}+|\th|^{p-1}]_h \ dz\\
  & \apprle  2^{m+1} r_j \ga \fiint_{2^{m+1}Q_j} [|\nabla u|^{p-1}+|\th|^{p-1}]_h \ dz.
 \end{array}
\end{equation}

If $m \leq 3$, then we rescale to reach $16Q_j$ and use \descref{W4}, and  if $m \geq 4$, then trivially we can use \descref{W4}. Thus combining \eqref{S_1_1} and \eqref{S_1_2} along with \descref{W4}, we get
\begin{equation}
 \label{S_1_3}
 S_1^m \apprle 2^{m+1} r_j \lbr \la^q + (\la^{p-1}\ga)^q +  \rbr^{\frac{1}{q}} \apprle 2^{m+1}r_j \la ,
\end{equation}
where we have used $\ga = \la^{2-p}$.

   \item[Estimate for $S_2$:] This term can be easily estimated as in the case $r_j \geq \rho$ to obtain
   \begin{equation}
   \label{S_2_1}
    \begin{array}{ll}
     \avgs{2^{k_0-2}Q_j}{{[u]_h}} 
     & \apprle \rho \la. 
    \end{array}
   \end{equation}
   
  \end{description}
Substituting \eqref{S_1_3} and \eqref{S_2_1} into \eqref{est_2} and making use of \eqref{def_k_0}, we get
\begin{equation}
 \label{east_2}
 \begin{array}{ll}
  |\tuh^j|  \apprle\la  \sum_{m=0}^{k_0-2} 2^{m+1}r_j + \la \rho 
 \apprle \la  \sum_{m=0}^{k_0-2} 2^{m+1}2^{1-k_0}\rho + \la \rho 
  \leq C_{(n,p,q,\lamot,b_0)} \rho \la. 
 \end{array}
\end{equation}
\end{description}
Thus \eqref{est_1} and \eqref{east_2} imply  \eqref{claim_bound} holds, which proves the lemma.
\end{proof}
Now we prove a sharper estimate over $\Th_1$ from \eqref{theta_1}.
\begin{lemma}
\label{lemma3.8}
 For any $i \in \Th_1$ and any $j \in A_i$, there holds
 \begin{equation*}
  \label{3.26}
  |\tuh^i - \tuh^j| \leq C_{(n,p,q,\lamot,b_0)} \min\{\rho, r_i\} \la.
 \end{equation*}
\end{lemma}

\begin{proof}
We only have to consider the case $r_i \leq \rho$, because if $\rho \leq r_i$, we can directly use Lemma \ref{lemma3.7} to get the desired conclusion.

If either $\tuh^i = 0$ or $\tuh^j = 0$, then using \descref{W4}, \descref{W14} and Theorem \ref{sobolev-poincare}, we get
%
%
\begin{equation}
\label{lemma3.8.1}
 \begin{array}{ll}
  |\tuh^j| & \apprle r_j \lbr \fiint_{16Q_j} \left| \frac{\tuh}{r_j} \right|^q \ dz \rbr^{\frac{1}{q}} \apprle r_j \lbr \fiint_{16Q_j} \left| {\nabla \tuh} \right|^q \ dz \rbr^{\frac{1}{q}} \apprle r_j \la. 
 \end{array}
\end{equation}
Combining \eqref{lemma3.8.1} and \eqref{claim_bound}, the lemma follows in this case.


Now let us consider the case of $\tuh^i \neq 0$ and $\tuh^j \neq 0$. This implies $\frac34Q_i \subset \mcc_{8\rho}$ and $\frac34Q_j \subset \mcc_{8\rho}$ along with $j \in \Th_1$. Using triangle inequality and \descref{W12}, we get
\begin{equation}
 \label{3.29}
 \begin{array}{ll}
 |\tuh^i - \tuh^j| 
 & \apprle \frac{|\frac34Q_i|}{|\frac34Q_i \cap \frac34Q_j|}\fiint_{\frac34Q_i} |\tuh(z) - \tuh^i| \ dz + \frac{|\frac34Q_j|}{|\frac34Q_i \cap \frac34Q_j|}\fiint_{ \frac34Q_j} |\tuh(z) - \tuh^j| \ dz \\
 & \apprle \fiint_{\frac34Q_i} |\tuh(z) - \tuh^i| \ dz + \fiint_{ \frac34Q_j} |\tuh(z) - \tuh^j| \ dz.
 \end{array}
\end{equation}

Let us now estimate each of the terms in \eqref{3.29} as follows: we apply H\"older's inequality followed by  Lemma \ref{lemma_crucial_1} with $\mu \in C_c^{\infty}(\frac34B_i)$ satisfying $|\mu(x)| \apprle \frac{1}{r_i^n}$ and $|\nabla \mu(x)| \apprle \frac{1}{r_i^{n+1}}$,  to get
\begin{equation}
 \label{3.30}
 \begin{array}{ll}
 \fiint_{\frac34Q_i} | \tuh(z) - \tuh^i| \ dz 
 & \apprle r_i \lbr \fiint_{\frac34Q_i} |\nabla \tuh|^q \ dz + \sup_{t_1,t_2 \in \frac34I_i} \left|\frac{\avgs{\mu}{\tuh}(t_2) - \avgs{\mu}{\tuh}(t_1)}{r_i} \right|^q \rbr^{\frac{1}{q}}.
 \end{array}
\end{equation}

The first term on the right of \eqref{3.30} can be controlled using \descref{W4}.  To control the second term of \eqref{3.30}, recalling  \eqref{def_u_tilde}, we apply Lemma \ref{lemma_crucial_2} with $\phi(x) = \eta(x) \mu(x)$ and $\varphi(t) = \zeta(t) \equiv 1$ on $\frac34I_i$ as a test function. Further making  use of \eqref{bound_b}, we get
\begin{equation}
 \label{3.31}
 \begin{array}{ll}
  |\avgs{\mu}{\tuh}(t_2) - \avgs{\mu}{\tuh}(t_1)| & \apprle \| \nabla (\eta \mu) \|_{\infty} \iint_{\frac34Q_i} [|\nabla u|^{p-1}+ |\th|^{p-1}]_h \ dz \\
  & \overset{\redlabel{3.24.a}{a}}{\apprle} \lbr \frac{1}{\rho r_i^n} + \frac{1}{r_i^{n+1}}  \rbr |16Q_i| \la^{p-1} 
   \overset{\redlabel{3.24.b}{b}}{\apprle}   \la r_i.
 \end{array}
\end{equation}
To obtain \redref{3.24.a}{a}, we again made use of \eqref{bound_b}, \descref{W4},  \eqref{cut_off_function} and the structure of $\mu$. To obtain \redref{3.24.b}{b}, we used $r_i \leq \rho$ and $\ga = \la^{2-p}$.

Combining \eqref{3.30} and \eqref{3.31}, we get in the case $i \in \Th_1$, the bound
\begin{equation}
 \label{bound_when_i_1}
 \fiint_{\frac34Q_i} | \tuh(z) - \tuh^i| \ dz  \apprle r_i \la.
\end{equation}
%
%
%
%
%
%
This completes the proof of the lemma. 
\end{proof}

\begin{lemma}
\label{lemma3.9}
 Given any  $z \in \elam^c$, we have $z \in \frac34Q_i$ for some $i \in \NN$. If either  $i \in \Th_1$ or $i \in \Th_2$ with $\rho \leq r_i$, then there holds
 \begin{equation}
  \label{3.34}
  |\nabla \vlh(z)| \leq C_{(n,p,q,\lamot,b_0)} \la.
 \end{equation}

\end{lemma}

\begin{proof}
 Let us prove each of the cases separately:
 \begin{description}[leftmargin=*]
  \item[Case $i\in \Th_1$:] In this case, we observe that $\sum_{j} \om_j(z) = \sum_{j\in A_i} \om_j(z)=1$ for any $z \in \elam^c$, which implies
  \begin{equation}
   \label{3.35}
  \sum_j \nabla \om_j(z) =  \nabla \lbr \sum_j \om_j(z) \rbr = 0 \qquad \forall \ z \in \elam^c.
  \end{equation}

  Thus using \eqref{lipschitz_extension} and  \eqref{3.35}, we get
  \begin{equation*}
   \label{3.36}
   \begin{array}{ll}
   \nabla \vlh(z) = \sum_j \nabla \om_j(z) \tuh^j 
   & = \sum_{ j\in A_i} \nabla \om_j(z) \lbr \tuh^j - \tuh^i\rbr. 
   \end{array} 
  \end{equation*}

  Now making use of \descref{W9} and  \descref{W13} along with \eqref{theta} and Lemma \ref{lemma3.8}, we get
  \begin{equation*}
  \label{3.37}
   |\nabla \vlh(z)| \apprle \sum_{ j\in A_i} \frac{1}{r_j} \min\{r_i, \rho\} \la \apprle \la.
  \end{equation*}

  \item[Case $i \in \Th_2$ and $\rho \leq r_i$:] This is a simpler case, as we can directly use \eqref{lipschitz_extension} and  \eqref{claim_bound} along with \descref{W13} to get
  \begin{equation*}
   \label{3.38}
    |\nabla \vlh(z)|  \apprle \sum_{ j\in A_i} |\nabla \om_j(z)| |\tuh^j| \apprle \sum_{ j\in A_i} \frac{1}{r_j} \rho \la  \apprle \la.
  \end{equation*}
 \end{description}
 This completes the proof of the Lemma. 
\end{proof}

\subsection{Estimates on the spatial derivative of \texorpdfstring{$\vlh$}.}
In this subsection, we shall prove a few useful estimates on the spatial derivative of $\vlh(z)$ for $z \in \elam^c$. 

\begin{lemma}
 \label{lemma3.10.1}
 Let $z \in \elam^c$ and $\ve >0$ be any number, then there exists a constant $C(n,p,q,\lamot,b_0)$ such that the following holds:
 \begin{align}
  |\vlh(z)| & \leq C \fiint_{4Q_i}|\tuh(y,s)| \ dy\ ds \leq  \frac{Cr_i\la}{\varepsilon} + \frac{C\varepsilon}{\la r_i} \fiint_{4Q_i}|\tuh(y,s)|^2 \ dy\ ds, \label{lemma3.10_bound1}\\
  |\nabla \vlh(z)| &\leq C \frac{1}{r_i} \fiint_{4Q_i}|\tuh(y,s)| \ dy\ ds \leq  \frac{C \la}{\varepsilon} + \frac{C\varepsilon}{\la r_i^2} \fiint_{4Q_i}|\tuh(y,s)|^2 \ dy\ ds\label{lemma3.10_bound2}.
 \end{align}
\end{lemma}
\begin{proof} Let us prove each of the estimates as follows:
 \begin{description}[leftmargin=*]
  \item[Proof of \eqref{lemma3.10_bound1}:] Making use of \descref{W9} along with \eqref{lipschitz_extension}, \descref{W13} and \descref{W14}, we get
  \begin{equation*}
    |\vlh(z)| 
     \apprle \sum_{j \in A_i} |\om_j(z)| \fiint_{\frac34Q_j} |\tuh(y,s)|  \ dy \ ds 
     \apprle \fiint_{4Q_i} |\tuh(y,s)|  \ dy \ ds.
  \end{equation*}
  To get the second inequality in \eqref{lemma3.10_bound1}, we use Young's inequality to get
  \begin{equation*}
   \fiint_{4Q_i} |\tuh(y,s)|  \sqrt{\frac{\la r_i}{\ve}} \sqrt{\frac{\ve}{\la r_i}}\ dy \ ds \apprle \frac{r_i\la}{\varepsilon} + \frac{\varepsilon}{\la r_i} \fiint_{4Q_i}|\tuh(y,s)|^2 \ dy\ ds.
  \end{equation*}
\item[Proof of \eqref{lemma3.10_bound2}:] This again follows exactly as the bound for \eqref{lemma3.10_bound1}, but in this case we additionally make use of the fact that $|\nabla \om_j| \apprle \frac{1}{r_j}$ along with \descref{W5} to get the desired conclusion. 
 \end{description}
 This proves the lemma.
\end{proof}

\begin{lemma}
 \label{lemma3.10.2}
 Let $z \in \elam^c$ and $\ve \in (0,1)$ be given.  From Lemma \ref{whitney_decomposition}, we have $z \in \frac34Q_i$ for some $i \in \NN$. Suppose $i \in \Th_1$, then there holds:
 \begin{align}
  |\vlh(z)|& \leq C_{(n,p,q,\lamot,b_0)} \lbr \min\{ \rho, r_i\} \la + |\tuh^i| \rbr \leq  C_{(n,p,q,\lamot,b_0)} \lbr \frac{ r_i\la}{\ve} + \frac{\ve}{r_i \la} |\tuh^i|^2 \rbr  \label{lemma3.10_bound3}\\
  |\nabla \vlh(z)| &\leq C_{(n,p,q,\lamot,b_0)} \frac{\la}{\ve}.\label{lemma3.10_bound4}
 \end{align}
\end{lemma}
\begin{proof}
 Let us prove each of the estimates as follows:
 \begin{description}[leftmargin=*]
  \item[Proof of \eqref{lemma3.10_bound3}:] Using triangle inequality, \descref{W9}, \descref{W13} and Lemma \ref{lemma3.8}, we get
  \begin{equation*}
    |\vlh(z)| 
     \apprle  \sum_{j : j \in A_i} \om_j(z)|\tuh^j -\tuh^i| + |\tuh^i| 
     \apprle \min\{ \rho, r_i\} \la + |\tuh^i| \overset{\redlabel{3.33.a}{a}}{\apprle}\frac{r_i\la}{\ve} + \frac{\ve}{r_i \la} |\tuh^i|^2.
  \end{equation*}
  To obtain \redref{3.33.a}{a}, we made use of Young's inequality and this proves the estimate. 
  \item[Proof of \eqref{lemma3.10_bound4}:] This follows exactly as \eqref{lemma3.10_bound3}, but in this case we additionally make use of the fact that $|\nabla \om_j| \apprle \frac{1}{r_j}$ along with \descref{W5} to get the desired conclusion. 
 \end{description}
 This proves the lemma.
\end{proof}

\begin{lemma}
 \label{lemma3.10.3}
 Let $z \in \elam^c$, then from Lemma \ref{whitney_decomposition}, we have $z \in \frac34Q_i$ for some $i \in \NN$. Suppose $i \in \Th_2$, then there holds:
  \begin{align}  
  | \vlh(z)| &\leq C_{(n,p,q,\lamot,b_0)}\lbr  r_i \la + \frac{\la^{1-p}r_i}{s} \fiint_{4Q_i}|\tuh(y,s)|^2 \ dy\ ds \rbr,\label{lemma3.10_bound5}\\
  |\nabla \vlh(z)| &\leq C_{(n,p,q,\lamot,b_0)}\lbr  \la + \frac{\la^{1-p}}{s} \fiint_{4Q_i}|\tuh(y,s)|^2 \ dy\ ds\rbr.\label{lemma3.10_bound6} 
 \end{align}
\end{lemma}
\begin{proof}
 Since $i \in \Th_2$, we must have from the definition of $\Th_2$ in \eqref{theta_2} and \descref{W5} that $\ga r_j^2 \apprge s$ for all $j \in A_i$,  where $\ga = \la^{2-p}$. Setting $\ve = 1$ in Lemma \ref{lemma3.10.1} and making use of the bound $\la^{2-p} r_i^2 \apprge s$ (which uses \descref{W5}) along with \descref{W13}, the proof of the lemma follows.
\end{proof}

\subsection{Estimates on the time derivative of \texorpdfstring{$\vlh$}.}
\begin{lemma}
\label{lemma3.11}
Let $z \in \elam^c$, then from Lemma \ref{whitney_decomposition}, there exists an $i \in \NN$ such that $z \in \frac34Q_i$. Then the following estimates for the time derivative of $\vlh$ holds:
 \begin{align}
 |\pa_t \vlh(z)| & \leq C_{(n,p,q,\lamot,b_0)} \frac{1}{\ga r_i^2} \fiint_{4Q_i} |\tuh(y,s)| \ dy\ ds. \label{lemma3.11.bound1} 
 \end{align}
 If $i \in \Th_1$, then we have
 \begin{align}
   |\pa_t \vlh(z)| & \leq C_{(n,p,q,\lamot,b_0)} \frac{1}{\ga r_i^2} \min\{r_i,\rho\} \la. \label{lemma3.11.bound2}
  \end{align}
  If $i \in \Th_2$, then there holds
  \begin{align}
   |\pa_t \vlh(z)| & \leq C_{(n,p,q,\lamot,b_0)} \frac{\rho \la}{s}   \label{lemma3.11.bound3}.
 \end{align}

\end{lemma}
\begin{proof}
 We shall prove each of the estimates as follows:
 \begin{description}[leftmargin=*]
  \item[Estimate \eqref{lemma3.11.bound1}:] Using \descref{W9} and \descref{W13}, we can proceed analogous to how  \eqref{lemma3.10_bound1} was obtained to get 
  \begin{equation*}
  \begin{array}{ll}
   |\pa_t \vlh(z)| &\apprle \sum_{j:j \in A_i} |\pa_t \om_j(z)| |\tuh^j| 
    \apprle \frac{1}{\ga r_i^2} \fiint_{4Q_i} |\tuh(y,s)| \ dy\ ds .
   \end{array}
  \end{equation*}
  \item[Estimate \eqref{lemma3.11.bound2}:] Since $\sum_{j} \om_j(\tilde{z})\equiv 1$ for all $\tilde{z} \in \elam^c$, we must have $\sum_j \pa_t \om_j(\tilde{z}) = 0$, which gives
\begin{equation*}
\label{3.54}
 \begin{array}{ll}
  |\pa_t \vlh(z)| &= \left|\sum_{ j\in A_i} \pa_t\om_j(z) \tuh^j\right| \apprle \sum_{ j\in A_i} |\pa_t\om_j(z)| |\tuh^j-\tuh^i|
   \overset{\redlabel{3.54.a}{a}}{\apprle} \frac{1}{\ga r_i^2} \min\{ r_i,\rho\} \la.
 \end{array}
\end{equation*}
To get \redref{3.54.a}{a}, we have made use of Lemma \ref{lemma3.8} (which is applicable since $i \in \Th_1$) along with \descref{W5} and \descref{W13}. 
  \item[Estimate \eqref{lemma3.11.bound3}] Since $i \in \Th_2$, we automatically have $\ga r_i^2 \apprge s$. Now making use of  \eqref{claim_bound}, \descref{W5} and \descref{W13}, we get
  \begin{equation*}
   \begin{array}{ll}
    |\pa_t \vlh(z)| & \apprle \sum_{j\in A_i} |\pa_t \om_j(z)| |\tuh^j|  \apprle \sum_{j\in A_i} \frac{1}{\ga r_j^2} \rho \la 
     \apprle \frac{\rho \la}{s}.
   \end{array}
  \end{equation*}
 \end{description}
 This proves the lemma.
\end{proof}

\subsection{Main estimates and Lipschitz continuity of \texorpdfstring{$\vlh$}. } 
\begin{lemma}
\label{lemma3.12}
For any $\vartheta \geq 1$,  we have the following bound:
 \begin{equation*}
  \label{3.56}
  \iint_{8Q\setminus\elam } |\vlh(z)|^{\vartheta} \ dz \leq C_{(n,p,q,\lamot,b_0)} \iint_{8Q\setminus\elam } |\tuh(z)|^{\vartheta} \ dz.
 \end{equation*}

\end{lemma}

\begin{proof}
 Since $\elam^c$ is covered by Whitney cylinders (see Lemma \ref{whitney_decomposition}), let us pick some $i \in \NN$ and consider the corresponding parabolic Whitney cylinder. Using the construction from \eqref{lipschitz_extension} along with \descref{W5}, \descref{W9} and \descref{W13}, we get
 \begin{equation}
 \label{3.57}
   \iint_{\frac34Q_i} |\vlh(z)|^{\vartheta} \ dz  \apprle \sum_{j:j \in A_i} \iint_{\frac34Q_i} \om_j(z)^{\vartheta} |\tuh^j|^{\vartheta} 
    \apprle \iint_{4Q_i} |\tuh(z)|^{\vartheta} \ dz. 
 \end{equation}
 Summing \eqref{3.57} over all $i \in \NN$ and using $\spt(\tuh)\subset 8Q$ along with \descref{W4} and \descref{W7}, we get
\begin{equation}
 \iint_{8Q\setminus \elam} |\vlh(z)|^{\vartheta} \ dz \leq \sum_i\iint_{\frac34Q_i} |\vlh(z)|^{\vartheta} \ dz \apprle \sum_i \iint_{4Q_i} |\tuh(z)|^{\vartheta} \ dz \apprle \iint_{8Q\setminus \elam} |\tuh(z)|^{\vartheta} \ dz.
\end{equation}
This proves the Lemma.
\end{proof}

\begin{lemma}
\label{lemma3.13}
 For any  $i \in \Th_1$,  we have the following estimate: 
 \begin{equation}
 \label{3.59}
  \fiint_{\frac34Q_i} \left|\frac{\tuh(z) - \vlh(z)}{r_i}\right|^q \ dz \leq C_{(n,p,q,\lamot,b_0)} \la^q.
 \end{equation}

\end{lemma}

\begin{proof}
 For any  $z \in \frac34Q_i$, using \eqref{lipschitz_extension} along with triangle inequality and  \descref{W9}, we get
 \begin{equation}
 \label{3.60}
    \fiint_{\frac34Q_i}|\tuh(z) - \vlh(z)|^q \ dz 
   \leq \fiint_{\frac34Q_i}|\tuh(z) - \tuh^i|^q \ dz +  \sum_{j:j \in A_i}\fiint_{\frac34Q_i} \left|\tuh^j - \tuh^i \right|^q\  dz := J_1 + J_2.
 \end{equation}

 We shall estimate each of the terms of \eqref{3.60} as follows (note that $i \in \Th_1$): 
 \begin{description}[leftmargin=*]
  \item[Estimate for $J_1$:] If $\tuh^i \neq 0$, then  $J_1$ is exactly as in \eqref{3.30}, which implies 
  \begin{equation}
  \label{3.61}
   J_1 \apprle (r_i \la)^q. 
  \end{equation}
   If $\tuh^i = 0$, then we can proceed as \eqref{lemma3.8.1} to again bound $J_1$ by \eqref{3.61}.

  \item[Estimate for $J_2$:] To bound this term, we can directly use Lemma \ref{lemma3.8} to get 
  \begin{equation}
  \label{3.62}
   J_2 \apprle (r_i \la)^q. 
  \end{equation}
 \end{description}
 Substituting \eqref{3.61} and \eqref{3.62} into \eqref{3.60} and making use of  \descref{W13}, the lemma follows.
\end{proof}


\begin{lemma}
\label{lemma2.5}
Given any  $i \in \Th_2$ and $\al_0$ as in \eqref{def_s}, there holds:
 \[
   \fiint_{\frac34Q_i} \left| \frac{\tuh - \tuh^i}{r_i} \right|^q \ dz \leq C_{(n,p,q,\lamot,b_0)} \la^q \lbr 1  + \lbr \frac{\la^{2-p}}{\al_0^{2-p}}\rbr^{\frac{n+1}{2}}\rbr^q.
 \]
\end{lemma}

\begin{proof}
 Since $i \in \Th_2$, we have $\ga r_i^2 \apprge s$. Let us split into cases $\tuh^i \neq 0$ and $\tuh^i = 0$:
 \begin{description}[leftmargin=*]
  \item[Case $\tuh^i \neq 0$:] From the construction in \eqref{lipschitz_extension}, we have  $\frac34Q_i \subset \mcc_{8\rho}$, which also implies the bound $r_i \leq 8\rho$. Now applying Lemma  \ref{lemma_crucial_1} with $\mu(x) \in C_c^{\infty} (\frac34B_i)$ satisfying $|\mu| \leq \frac{C(n)}{r_i^n}$ and $|\nabla \mu| \leq \frac{C(n)}{r_i^{n+1}}$,  we get
 \begin{equation}\label{lemma2.5.8.1}
  \begin{array}{rcl}
   \fiint_{\frac34Q_i} \left| \frac{\tuh(z) - \tuh^i}{r_i} \right|^q \ dz 
   & \apprle& \fiint_{\frac34Q_i} |\nabla \tuh|^q \ dz + \sup_{t_1,t_2 \in \frac34I_i} \left| \frac{\avgs{\mu}{\tuh}(t_1)-\avgs{\mu}{\tuh}(t_2)}{r_i} \right|^q .
  \end{array}
 \end{equation}
 The first term of \eqref{lemma2.5.8.1} can be estimated using \descref{W4}. To estimate the second term of \eqref{lemma2.5.8.1}, let us apply Lemma \ref{lemma_crucial_2} with $\phi(x) = \mu(x) \eta(x)$ and $\varphi(t) = \zeta(t)$, which combined with \eqref{cut_off_function}, gives for any $t_1,t_2 \in \frac34I_i$, the estimate
  \begin{equation}
  \label{lemma2.5.8}
  \begin{array}{ll}
  |\avgs{\mu}{\tuh}(t_2)- \avgs{\mu}{\tuh}(t_1)| 
   & \apprle \|\nabla (\eta \mu) \|_{L^{\infty}(\frac34B_i)}   \iint_{\frac34Q_i} [|\nabla u|^{p-1}+|\th|^{p-1}]_h\ dz  \\
   & \qquad +\  \|\mu\|_{L^{\infty}(\frac34B_i)}   \iint_{\frac34Q_i\cap 8Q} \frac{[u]_h}{s} \ dz \\
    &:= J_1 + J_2.
  \end{array}
 \end{equation}

  We can easily control $J_1$ of \eqref{lemma2.5.8} using \descref{W4} and  $r_i \leq 8\rho$ to get
  \begin{equation}
   \label{lemma2.5.9}
  J_1  \apprle \lbr \frac{1}{\rho r_i^n} + \frac{1}{ r_i^{n+1}}\rbr|Q_i| \la^{p-1}  \apprle r_i \la.
  \end{equation}

To control $J_2$ of  \eqref{lemma2.5.8}, we proceed as follows:
\begin{equation}
\label{lemma2.5.9.1}
  J_2 \overset{\text{\redlabel{3.50.a}{a}}}{\apprle}  \frac{\rho}{r_i^n s} |8Q|  \lbr \fiint_{8Q} \lbr \frac{[u]_h}{\rho}\rbr^q \ dz\rbr^{\frac1{q}} \overset{\text{\redlabel{3.50.b}{b}}}{\apprle} \frac{\rho^{n+1}}{r_i^n}   \lbr \fiint_{8Q} |\nabla [u]_h|^q \ dz\rbr^{\frac1{q}} \overset{\text{\redlabel{3.50.c}{c}}}{\apprle} 
   \lbr \frac{\ga}{\al_0^{2-p}}\rbr^{\frac{n+1}{2}} r_i \la.
\end{equation}
To obtain \redref{3.50.a}{a}, we made use of H\"older's inequality. To get \redref{3.50.b}{b}, we made use of Theorem \ref{sobolev-poincare}. Finally  to get \redref{3.50.c}{c}, we used \descref{W4}, $\ga r_i^2 \apprge s$ and \eqref{def_s}.

  Thus substituting \eqref{lemma2.5.8},  \eqref{lemma2.5.9} and \eqref{lemma2.5.9.1} into \eqref{lemma2.5.8.1}, the Lemma follows for case $\tuh^i \neq 0$.
  \item[Case $\tuh^i = 0$:]   In this case, we can directly apply Theorem \ref{sobolev-poincare} and make use of \descref{W4} to get 
 \begin{equation*}\label{lemma2.5.6}
  \begin{array}{ll}
   \fiint_{\frac34Q_i} \left| \frac{\tuh}{r_i} \right|^q \ dz \apprle \fiint_{\frac34Q_i} |\nabla \tuh|^q \apprle \la^q.
  \end{array}
 \end{equation*}
 \end{description}
 This completes the proof of the Lemma. 
 \end{proof}

\begin{lemma}
\label{lemma3.14}
 For any $0< \vartheta \leq q$, there exists a positive constant $C(n,p,q,\lamot,b_0, \vartheta)>0$ such that the following holds:
 \begin{equation*}
  \iint_{8B \times 2\tm \setminus \elam} |\pa_t \vlh(z)  (\vlh(z) - \tuh(z))|^{\vartheta} \ dz \leq C \la^{\vartheta p} |\RR^{n+1} \setminus \elam| + \frac{C}{s^{\vartheta}} \iint_{8Q} |\tuh(z)|^{2\vartheta} \ dz. 
 \end{equation*}

\end{lemma}

\begin{proof}
From \descref{W3} and \eqref{theta}, we see that 
 $8B \times 2\tm \setminus \elam\subset \sum_{i \in \Th} 4Q_i$. Thus to prove the lemma, it will be sufficient to  consider the two sub-cases, one where $i \in \Th_1$ and the other where $i \in \Th_2$. Thus, for a given $i \in \Th$, let us define the following: 
 \begin{equation*}
  J_i:= \iint_{\frac34Q_i} |\pa_t \vlh(z)  (\vlh(z) - \tuh(z))|^{\vartheta} \ dz .
 \end{equation*}

 \begin{description}[leftmargin=*]
  \item[Case $i\in \Th_1$:] Making use of  \eqref{lemma3.11.bound2} and H\"older's inequality applied to \eqref{3.59} (recall $\ga = \la^{2-p}$), we get
  \begin{equation}
  \label{3.66}
J_i  \apprle \lbr \frac{1}{\ga r_i^2} r_i \la \rbr^{\vartheta} \iint_{\frac34Q_i} |\vlh(z) - \tuh(z)|^{\vartheta} \ dz 
 \apprle \la^{\vartheta p} \ |\frac34Q_i|.
  \end{equation}
  \item[Case $i\in \Th_2$:] In this case, we obtain the following sequence of estimates: 
  \begin{equation}
  \label{3.67}
   \begin{array}{ll}
    J_i & \overset{\redlabel{3.52.a}{a}}{\apprle} \lbr \frac{1}{\ga r_i^2} \rbr^{\vartheta} \lbr \fiint_{4Q_i} |\tuh(z)| \ dz \rbr^{\vartheta} \lbr \iint_{\frac34Q_i} |\vlh(z)| + |\tuh(z)| \ dz \rbr^{\vartheta} \\
    & \overset{{\redlabel{3.52.b}{b}}}{\apprle} \lbr \frac{1}{\ga r_i^2} \rbr^{\vartheta} \lbr \fiint_{4Q_i} |\tuh(z)| \ dz \rbr^{\vartheta} \lbr |\frac34Q_i|\fiint_{4Q_i} |\tuh(z)| \ dz + \iint_{4Q_i}|\tuh(z)| \ dz \rbr^{\vartheta} \\
    & \overset{{\redlabel{3.52.c}{c}}}{\apprle} \frac{1}{s^{\vartheta}}  \lbr \iint_{4Q_i} |\tuh(z)|^{2\vartheta} \ dz \rbr.
   \end{array}
  \end{equation}
  To get \redref{3.52.a}{a}, we made use of \eqref{lemma3.11.bound1}. To obtain \redref{3.52.b}{b} above, we  made use of \eqref{lemma3.10_bound1}. Finally to get \redref{3.52.c}{c}, we have made use of the fact that $i\in\Th_2$, which implies $\ga r_i^2 \apprge s$. 

 \end{description}

 Combining \eqref{3.66} and \eqref{3.67} followed by summing over all $i \in \Th$ and  using  \descref{W7}, we get 
 \begin{equation*}
  \begin{array}{ll}
   \iint_{8B \times 2\tm \setminus \elam} |\pa_t \vlh(z)  (\vlh(z) - \tuh(z))|^{\vartheta} \ dz 
   & \apprle \sum_{i \in \Th_1} J_i + \sum_{i \in \Th_2} J_i \\
   & \overset{}{\apprle} \la^{\vartheta p} |\RR^{n+1} \setminus \elam| + \frac{1}{s^{\vartheta}} \iint_{8Q} |\tuh(z)|^{2\vartheta} \ dz. 
  \end{array}
 \end{equation*}
This completes the proof of the lemma. 
\end{proof}

We shall now prove the Lipschitz continuity of $\vlh$ on $2\mch = \RR^n \times (-4s, 4s)$. {\color{black}The method of making use of Lemma \ref{metric_lipschitz} to prove Lemma \ref{lemma3.15} is due to \cite[Chapter 3]{bogelein2013regularity}.}	
\begin{lemma}
\label{lemma3.15}
The extension $\vlh$ from  \eqref{lipschitz_extension} is $C^{0,1}(2\mch)$ with respect to the parabolic metric $d_{\la}$. 
\end{lemma}

\begin{proof} 
Let us consider a parabolic cylinder $\tQ_{r}(z) := \tQ_{r, \ga r^2} (z) := \tQ$ for some $z \in 2\mch$ and $r>0$. To prove the Lemma, we make use of Lemma \ref{metric_lipschitz} to bound
 \begin{equation}
 \label{bound_I_r}
I_r(z) :=   \fiint_{\tQ \cap 2\mch} \left|\frac{\vlh(\tz) - \avgsnoleft{\tQ\cap2\mch}{\vlh}}{r}\right| \ d\tz \leq {o}(1),
 \end{equation}
where ${o}(1)$ denotes a bound independent of $z \in 2\mch$ and $r>0$ only. 
%
%
%
%
In order to do this, we split the proof into three cases given by
%
\begin{gather*}
 (i)\ 2\tQ \subset 4\mch \setminus \elam, \quad (ii) \ 2\tQ \subset 4\mch \quad  \text{with} \ \ga r^2 \leq 4s, \\
(iii)\ 2\tQ \subset 4\mch \quad \text{with}\  \ga r^2 > 4s  \quad \text{\bf OR}  \quad 2\tQ \cap (4\mch)^c \neq \emptyset. 
\end{gather*}

\begin{description}[leftmargin=*]
 \item[Case $2\tQ \subset 4\mch \setminus \elam$:]  In this case, we observe the following lower bound must hold:
 \begin{equation}
  \label{3.70}
  |\tQ \cap 2\mch| \apprge |\tQ| \approx r^{n+2} \ga. 
 \end{equation}
 From the construction of $\vlh$ in \eqref{lipschitz_extension}, we see that  $\vlh$ is smooth on $\elam^c$.

 If $2\tQ \subset  \mcc_{8\rho}$, then setting $\tQ = \tB_r \times \tilde{I}_{\ga r^2}$ and  using the definition of derivatives, we get
  \begin{equation}
   \label{3.71}
   \def\arraystretch{2.2}
I_r(z)  \apprle   \frac{1}{r} \fiint_{\tQ \cap 2\mch} \fiint_{\tQ \cap 2\mch}  |\vlh(\tz) - \vlh(\tz_1)| \ d\tz \ d\tz_1
 \apprle  \sup_{\tz \in \tQ \cap 2\mch}\lbr  |\nabla\vlh(\tz)| + r |\pa_t \vlh(\tz)|\rbr .
  \end{equation}
  
  On the other hand, if $2\tQ \cap \mcc_{8\rho}^c \neq \emptyset$, then applying Theorem \ref{sobolev-poincare} on each times slice, we obtain
  \begin{equation}
   \label{3.72}
    I_r(z)  \apprle \fiint_{\tQ \cap 2\mch} \left|\frac{\vlh(\tz)}{r}\right| \ d\tz \apprle \fiint_{\tQ \cap 2\mch} \left| \nabla \vlh(\tz)\right| \ d\tz
    \apprle \sup_{\tz \in \tQ\cap 2\mch} |\nabla \vlh(\tz)|.
  \end{equation}

%
%
%

From \eqref{3.71} and \eqref{3.72}, we see that in order to  bound $I_r(z)$ independent of $z$ and $r$, we have to bound the following expression independent of $z$ and $r$:
 \begin{equation}\label{3.72.1}\sup_{\tz \in \tQ \cap 2\mch} \lbr |\nabla\vlh(\tz)| + r |\pa_t \vlh(\tz)|\rbr.\end{equation}
 
To this end, let  $\tz \in \tQ \cap 2\mch$, which implies  $\tz \in \frac34Q_i$ for some $i \in \NN$. Since $\tQ \subset \elam^c$, we also get
\begin{equation}
 \label{3.74}
 r \leq \pard (\tz, \elam) \leq \pard (\tz, z_i) + \pard (z_i , \elam) \leq r_i + 16r_i = 17r_i. 
\end{equation}

In order to bound \eqref{3.72.1}, we consider the following three sub-cases:
\begin{description}
\item[Subcase $i \in \Th_1$:] Using \eqref{3.34} and  \eqref{lemma3.11.bound2} along with  \eqref{3.74}, we get 
\begin{equation*}
\label{3.76}
|\nabla \vlh(\tz)| + r |\pa_t \vlh(\tz)| \apprle \la + r \frac{1}{\ga r_i^2} \min\{ r_i,\rho\} \la \apprle \la + \la^{p-1} \leq o(1). 
\end{equation*}


\item[Subcase $i \in \Th_2$ and $\rho \leq r_i$:] In this case, we make use of \eqref{3.34},  \eqref{lemma3.11.bound1} and \eqref{3.74} to get
\begin{equation*}
\label{3.77}
\begin{array}{ll}
|\nabla \vlh(\tz)| + r |\pa_t \vlh(\tz)| & \apprle \la + r \frac{1}{\ga r_i^2} \frac{1}{|Q_i|} \iint_{4Q_i} |\tuh(y,s)| \ dy\ ds \\
& \overset{}{\apprle } \la +  \frac{1}{\ga^2 \rho^{n+3}} \| \tuh\|_{L^1(\RR^{n+1})} \leq o(1).
\end{array}
\end{equation*}

\item[Subcase $i \in \Th_2$ and $\rho \geq r_i$:]  In this case, we use the fact $\ga r_i^2 \apprge s$ along with \eqref{lemma3.10_bound2} and \eqref{lemma3.11.bound3}  to obtain
\begin{equation*}
 \label{3.79}
 \begin{array}{ll}
  |\nabla \vlh(\tz)| + r |\pa_t \vlh(\tz)| & \apprle \frac{1}{r_i|Q_i|} \iint_{4Q_i} |\tuh(y,s)| \ dy\ ds + r \frac{\rho \la}{s} \\
  & \overset{}{\apprle } \frac{1}{\lbr \frac{s}{\ga}\rbr^{\frac{n+1}{2}} s} \| \tuh\|_{L^1(\RR^{n+1})} + \frac{\rho^2 \la}{s} \leq o(1).
 \end{array}
\end{equation*}
\end{description}

 We have shown that $I_r(z) \leq o(1)$ when $2\tQ \subset 4\mch \setminus \elam$. 

 \item[Case $2\tQ \subset 4\mch$ and $\ga r^2 \leq 4s$:] If $2\tQ \subset \elam^c$, then this is just the previous case.  Hence without loss of generality, we can assume 
 \begin{equation}\label{second_case}
 2\tQ \cap \elam \neq \emptyset. 
  \end{equation}
  Noting that \eqref{3.70} must also hold in this case, we apply 
%
  triangle inequality and estimate $I_r(z)$ by 
 \begin{equation}
  \label{3.81}
  \begin{array}{ll}
   I_r(z) 
   & \leq \fiint_{\tQ \cap 2\mch} \left| \frac{\vlh(\tz) - \tuh(\tz)}{r}\right|+ \left| \frac{\tuh(\tz) - \avgs{\tQ\cap 2\mch}{\tuh}}{r}\right| + \left| \frac{\avgs{\tQ\cap 2\mch}{\tuh} - \avgsnoleft{\tQ \cap 2\mch}{\vlh}}{r}\right| \ d\tz \\
   & \leq 2J_1 + J_2,
  \end{array}
 \end{equation}
 where we have set
 \begin{gather}
  J_1:= \fiint_{\tQ \cap 2\mch} \left| \frac{\vlh(\tz) - \tuh(\tz)}{r}\right| \ dz  \txt{and} J_2 := \fiint_{\tQ \cap 2\mch} \left| \frac{\tuh(\tz) - \avgs{\tQ\cap 2\mch}{\tuh}}{r}\right|\ dz.\label{def_J_1_2}
 \end{gather}
We now estimate each of the terms of \eqref{def_J_1_2} as follows:
%
\begin{description}[leftmargin=*]
\item[Estimate for $J_1$:] If $\tQ \subset \elam$, then $\vlh = \tuh$ which implies $J_1=0$. Hence without loss of generality, along with \eqref{second_case}, we can further assume $\tQ \cap \elam^c \neq \emptyset$. 
Using the construction of \eqref{lipschitz_extension}, we get
\begin{equation}
 \label{3.81.1}
 \begin{array}{ll}
  J_1 
  & \apprle \sum_{i\in \Th} \frac{1}{|\tQ\cap2\mch|} \iint_{\tQ \cap2\mch\cap \frac34Q_i} \left| \frac{\tuh(\tz) - \tuh^i}{r}\right| \ d\tz.
 \end{array}
\end{equation}

Let us fix an $i \in \Th$ and take two points $\tz_1 \in \tQ \cap \frac34Q_i$ and $\tz_2 \in \elam \cap 2\tQ$. Making use of  \descref{W2} along with the trivial bound $\pard (\tz_1, \tz_2) \leq  4r$ and $\pard (z_i, \tz_1) \leq 2r_i$,  we get
\begin{equation}
\label{3.83.1}
16r_i =\pard(z_i,\elam) \leq \pard (z_i, \tz_1) + \pard (\tz_1, \tz_2) \leq 2r_i + 4r  \ \Longrightarrow \ 2r_i \leq r,
\end{equation}
where $z_i$ denotes the centre of $\frac34Q_i$ as in \descref{W1}. 

Note that \eqref{3.70} holds and thus summing over all $i \in \Th$ such that  $\tQ \cap2\mch\cap \frac34Q_i \neq \emptyset$ in \eqref{3.81.1} and making use of  \eqref{3.83.1}, we get
\begin{equation*}
 \label{3.84.1}
 J_1 \apprle \sum_{\substack{i\in\Th \\ \tQ \cap2\mch\cap \frac34Q_i \neq \emptyset }} \frac{|\frac34Q_i|}{|\tQ\cap2\mch|} \fiint_{\frac34Q_i} \left| \frac{\tuh(\tz) - \tuh^i}{r}\right| \ d\tz \apprle \sum_{i\in \Th}  \fiint_{\frac34Q_i} \left| \frac{\tuh(\tz) - \tuh^i}{r_i}\right| \ d\tz.
\end{equation*}
To bound \eqref{3.84.1},  in the case $i \in \Th_2$, we apply H\"older's inequality followed by Lemma \ref{lemma2.5} and in the case $i \in \Th_1$, we again apply H\"older's inequality and use \eqref{3.61}. Thus we get
\begin{equation}
\label{3.84.2}
 J_1 \apprle \la \lbr 1  + \lbr \frac{\la^{2-p}}{\al_0^{2-p}}\rbr^{\frac{n+1}{2}}\rbr.
\end{equation}

 \item[Estimate for $J_2$:] Note that \eqref{3.70} is valid, which we use along with triangle inequality to get
 \begin{equation}
 \label{3.82}
  \begin{array}{ll}
   J_2 
   & \apprle \fiint_{\tQ \cap 2\mch}  \left| \frac{\tuh(\tz) - \avgs{\tQ}{\tuh}}{r}\right| +  \left| \frac{\avgs{\tQ}{\tuh} - \avgs{\tQ\cap 2\mch}{\tuh}}{r}\right| \ d\tz 
    \apprle 2 \fiint_{\tQ}  \left| \frac{\tuh(\tz) - \avgs{\tQ}{\tuh}}{r}\right| \ d\tz.
  \end{array}
 \end{equation}
  If  $\tQ \subset \mcc_{8\rho}$, then we estimate \eqref{3.82} by first applying Lemma \ref{lemma_crucial_1} for some  $\mu \in C_c^{\infty} (\tB)$ with $|\mu| \leq \frac{C(n)}{\tr^{n}}$ and  $|\nabla \mu| \leq \frac{C(n)}{\tr^{n+1}}$ to get
 \begin{equation*}
  \label{3.83}
  \begin{array}{ll}
   J_2 
   & \apprle \lbr \fiint_{\tQ} |\nabla \tuh(\tz)|^q \ d\tz + \sup_{t_1,t_2 \in \tilde{I}}\left|\frac{\avgs{\mu}{\tuh}(t_2)-\avgs{\mu}{\tuh}(t_1)}{r} \right|^q\rbr^{\frac{1}{q}}.
  \end{array}
 \end{equation*}
 
 Since we have $\tQ \subset \mcc_{8\rho}$, we have  $\tr \leq 8\rho$. Furthermore, since  $\tQ \subset 4\mch$, we see that  $\zeta(t) \equiv 1$ on $\tQ$. Thus,  we can apply Lemma \ref{lemma_crucial_2} and make use of the hypothesis \eqref{second_case} and mimic the procedure in the proof of Lemma \ref{lemma3.8} to obtain the bound 
%
%
%
 \begin{equation}
 \label{3.88}
  J_2 \apprle \la.
 \end{equation}

 On the other hand,  if $\tQ \nsubseteq \mcc_{8\rho}$, then we can directly apply Theorem \ref{sobolev-poincare} and use  \eqref{3.82},  to get
 \begin{equation}
 \label{3.89}
  \begin{array}{ll}
   J_2 & \apprle \fiint_{\tQ} \left| \frac{\tuh(\tz)}{r}\right| \ d\tz \apprle \lbr \fiint_{\tQ} |\nabla \tuh(\tz)|^q \ d\tz \rbr^{\frac{1}{q}} \apprle \la. 
  \end{array}
 \end{equation}
 
\end{description}

Thus substituting either \eqref{3.88} or \eqref{3.89} along with  \eqref{3.84.2} into \eqref{3.81}, we get  $I_r(z) \leq o(1)$ in the case that $2\tQ \subset 4\mch$ and $\ga r^2 \leq 4s$. 
 
 \item[Final case $2\tQ \subset 4\mch$ and $\ga r^2 > 4s$ OR $2\tQ \cap (4\mch)^c \neq \emptyset$:] In either case, we have $\ga r^2 \apprge s$, which gives the estimate
 \begin{equation*}
  \label{3.91}
  \begin{array}{ll}
  \fiint_{\tQ \cap 2\mch}  \left| \frac{\vlh(\tz) - \avgsnoleft{\tQ \cap 2\mch}{\vlh}}{r} \right|\ d\tz & \apprle \frac{1}{r^{n+3}} \iint_{\tQ \cap 2\mch} |\vlh(\tz)| \ d\tz \\
  & \apprle \frac{1}{r^{n+3}} \iint_{8Q \cap \elam} |\vlh(\tz)| \ d\tz + \frac{1}{r^{n+3}} \iint_{8Q\setminus \elam} |\vlh(\tz)| \ d\tz.
  \end{array}
 \end{equation*}

 By construction of $\vlh$ in \eqref{lipschitz_extension}, we have $\vlh = \tuh$ on $\elam$.  On $8Q \setminus \elam$, we can apply Lemma \ref{lemma3.12} (with $\vartheta=1$) to obtain the following bound:
 \begin{equation*}
  \begin{array}{ll}
   \fiint_{\tQ \cap 2\mch}  \left| \frac{\vlh(\tz) - \avgsnoleft{\tQ \cap 2\mch}{\vlh}}{r} \right|\ d\tz & \apprle \lbr \frac{\ga}{s} \rbr^{\frac{n+3}{2}} \iint_{8Q \cap \elam} |\tuh(\tz)| \ d\tz +  \lbr \frac{\ga}{s} \rbr^{\frac{n+3}{2}} \iint_{8Q \setminus \elam} |\tuh(\tz)| \ d\tz \\
   & \apprle o(1).
  \end{array}
 \end{equation*}
\end{description}
This completes the proof of the lemma. 
%
%
\end{proof}

\begin{remark}
In the proof of Lemma \ref{lemma3.15}, we can actually show the Lipschitz constant of $\vlh$ on $2\mch$  with respect to the parabolic metric \eqref{parabolic_metric} is $C_{(n,p,q,\lamot,b_0)}\la$. This improved bound will not be needed for the purposes of this paper, but it is worthwhile to mention how to obtain such an improvement. 

 In order to show this improvement, we can either use sharper estimates in the proof of Lemma \ref{lemma3.15} or mimic the proof from \cite{KL}. 
%
%
 
 Since we only take Steklov average in the time variable,  a priori we don't have continuity of $\vlh$ in the space variable to be able to implement the argument of \cite{KL} directly. On the other hand, once we have Lemma \ref{lemma3.15}, we know that $\vlh$ is continuous on $2\mch$, which enables us to implement the arguments of \cite{KL} and hence obtain the improved bound.

\end{remark}

\begin{lemma}
 \label{lemma3.18}
 For any $i \in \Th$ and $k \in \{0,1\}$, there exists a positive  constant $C{(n,p,q,\lamot,b_0)}$ such that there holds:
 \begin{equation}
 \label{lemma3.18.bound_1}
  \iint_{\frac34Q_i} [|\nabla u|^{p-1}+|\th|^{p-1}]_h |\nabla^k \vlh| \ dz \leq C \rho^{1-k} \lbr \la^p |4Q_i| + \frac{\lsb{\chi}{i \in \Th_2}}{s} \iint_{4Q_i} |\tuh|^2 \ dz \rbr. 
 \end{equation}
Here we have used the notation  $\lsb{\chi}{i\in \Th_2} = 1$ if $i \in \Th_2$ and $\lsb{\chi}{i\in \Th_2} = 0$ if $i \in \Th_1$ and $\nabla^0 \vlh := \vlh$. 
\end{lemma}

\begin{proof}
 We shall split the proof into two cases: 
 \begin{description}[leftmargin=*]
  \item[Case $i\in \Th_1$ or $i\in \Th_2$ with $\rho \leq r_i$:] In this case, we make use of Lemma \ref{lemma3.7} when $k=0$ or Lemma \ref{lemma3.9} when $k=1$,  along with \descref{W4}, to get
  \begin{equation*}
   \label{lemma3.18.1}
   \begin{array}{ll}
    \iint_{\frac34Q_i}  [|\nabla u|^{p-1}+|\th|^{p-1}]_h |\nabla^k \vlh| \ dz & \apprle \rho^{1-k} \la \iint_{\frac34Q_i}  [|\nabla u|^{p-1}+|\th|^{p-1}]_h  \ dz\\
    &{\apprle} \rho^{1-k} \la |4Q_i| \la^{p-1} = \rho^{1-k} \la^p |4Q_i|.
   \end{array}
  \end{equation*}

  \item[Case $i \in \Th_2$ with $r_i \leq \rho$:] If $k=0$, we make use of \eqref{lemma3.10_bound6} and in the case $k=1$, we make use of \eqref{lemma3.10_bound5}, along with \descref{W4} to  get
  \begin{equation*}
   \label{lemma3.18.2}
   \begin{array}{ll}
    \iint_{\frac34Q_i}  [|\nabla u|^{p-1}+|\th|^{p-1}]_h |\nabla \vlh| \ dz 
    & \apprle \rho^{1-k}\lbr \la  + \frac{\la^{1-p}}{s} \fiint_{4Q_i} |\tuh|^2 \ dz \rbr |4Q_i| \la^{p-1} \\
    & \apprle \rho^{1-k} \lbr \la^p |4Q_i| + \frac{1}{s} \iint_{4Q_i} |\tuh|^2 \ dz \rbr.
   \end{array}
  \end{equation*}
 \end{description}
This completes the proof of the lemma. 
\end{proof}


\begin{corollary}
 \label{corollary3.20}
There exists a positive constant $C{(n,p,q,\lamot,b_0)}$ such that  the following estimate holds for any $k \in \{0,1\}$:
 \begin{equation*}
  \label{corollary3.20.1}
  \iint_{8B \times 2\tm \setminus \elam} [|\nabla u|^{p-1}+|\th|^{p-1}]_h |\nabla^k \vlh| \ dz \leq C \rho^{1-k} \lbr \la^p |\RR^{n+1} \setminus \elam| + \frac{1}{s} \iint_{8Q} |\tuh|^2 \ dz \rbr. 
 \end{equation*}

\end{corollary}
\begin{proof}
 Since  $8B \times 2\tm \setminus \elam \subset \sum_{i\in \Th} \frac34Q_i$, all we need to do is to sum over $i\in \Th$ in \eqref{lemma3.18.bound_1} and make use of the fact that $\spt(\tuh) \subset 8Q$, to prove the Corollary. 
\end{proof}

We shall now prove a crucial estimate on each time slice. 
\begin{lemma}
 \label{pre_crucial_lemma}
 For any $i \in \Th_1$ and any $0 < \ve \leq 1$, there exists a positive constant $C{(n,p,q, \lamot,b_0)}$ such that for  almost every $t \in 2\tm$, there holds
 \begin{equation}
 \label{3.120}
  \left| \int_{8B} (\tu(x,t) - \tu^i) \vl(x,t) \om_i(x,t) \ dx \right| \leq C \lbr  \frac{\la^p}{\ve} |4Q_i| + \ve |4B_i| |\tu^i|^2\rbr. 
 \end{equation}
 
 In the case $i \in \Th_2$, there exists a positive constant $C{(n,p,q, \lamot,b_0)}$ such that for  almost every $t \in 2\tm$, there holds
 \begin{equation}
 \label{3.121}
  \left| \int_{8B} \tu(x,t) \vl(x,t) \om_i(x,t) \ dx \right| \leq C\lbr  {\la^p} |4Q_i| + \frac{1}{s}\iint_{\frac34Q_i} |u|^2 \lsb{\chi}{8Q}\ dz\rbr. 
 \end{equation}

\end{lemma}
\begin{proof}
Let us fix any $t \in 2\tm$,  $i \in \Th$  and take $\eta(y) \zeta(\tau) \om_i(y,\tau) \vlh(y,\tau)$ as a test function in \eqref{very_weak_solution}. Further integrating the resulting expression over $ \left(t_i - \ga \left(\frac34 r_i\right)^2 , t\right)$ along with making use of  the fact that $\om_i(y,t_i - \ga (3r_i/4)^2) = 0$, we get for  any $a\in \RR$, the equality
%
\begin{equation}
 \label{3.123}
 \begin{array}{ll}
  \int_{\Om} \lbr[(]  (\tuh - a)  \om_i \vlh \rbr (y,t) \ dy & = \int_{t_i - \ga \left(\frac34 r_i\right)^2}^t \int_{\Om} \pa_t \left(  (\tuh - a) \om_i \vlh \right) (y,\tau) \ dy \ d\tau \\
  & = \int_{t_i - \ga \left(\frac34 r_i\right)^2}^t \int_{\Om} \iprod{[\aa(y,\tau,\nabla u)]_h}{\nabla (\eta \zeta \om_i \vlh)} \ dy \ d\tau  \\
  & \qquad +\  \int_{t_i - \ga \left(\frac34 r_i\right)^2}^t \int_{\Om}   [u]_h  \ \pa_t\lbr \eta \zeta \om_i \vlh \right) (y,\tau) \ dy \ d\tau  \\
  & \qquad -\  \int_{t_i - \ga \left(\frac34 r_i\right)^2}^t \int_{\Om} a \pa_t \lbr \om_i \vlh\rbr \ dy \ d\tau.
 \end{array}
\end{equation}

We can estimate $|\nabla (\eta \zeta \om_i \vl)|$ using the chain rule, \eqref{cut_off_function} and \descref{W9}, to get
\begin{equation}
 \label{3.126}
 \begin{array}{ll}
 |\nabla (\eta \zeta \om_i \vl)| 
 & \apprle \frac{1}{\rho} |\vl| + \frac{1}{r_i} |\vl| + |\nabla \vl|.
 \end{array}
\end{equation}
Similarly, we can estimate $\left|\pa_t\lbr \eta \zeta \om_i \vl \right)\right|$ using the chain rule, \eqref{cut_off_function} and \descref{W9}, to get
\begin{align}
 \left|\pa_t\lbr \eta \zeta \om_i \vl \right)\right|  & \apprle \frac{\lsb{\chi}{i \in \Th_2}}{s} |\vl| + \frac{1}{\ga r_i^2} |\vl| + |\pa_t \vl|, \label{3.127} \\
 \left| \pa_t \lbr \om_i \vl\rbr\right| & \apprle  \frac{1}{\ga r_i^2} |\vl| + |\pa_t \vl|,\label{3.128}
\end{align}
where we have set $\lsb{\chi}{i \in \Th_2} = 1$ if  $i \in \Th_2$ and $\lsb{\chi}{i \in \Th_2}=0$ otherwise.

Let us now prove each of the assertions of the Lemma.
\begin{description}[leftmargin=*]
 \item[Proof of \eqref{3.120}:] Note that $i \in \Th_1$, which implies $\zeta(t) \equiv 1$ on $\frac34I_i$, thus taking $a=\tuh^i$ in the \eqref{3.123} followed by letting $h \searrow 0$ and making use of \eqref{3.126}, \eqref{abounded} and \eqref{bound_b},  we get
 \begin{equation}
  \label{first_1}
  \begin{array}{ll}
   \left| \int_{8B} \lbr[(] (\tu - \tu^i) \om_i \vl \rbr (y,t) \ dy \right| & \apprle J_1 + J_2 + J_3,
  \end{array}
 \end{equation}
 where we have set 
 \begin{align}
  J_1& := \frac{1}{\min\{\rho, r_i\}} \iint_{8Q} \lbr |\nabla u|^{p-1}+|\th|^{p-1} \rbr  |\vl|   \lsb{\chi}{\frac34Q_i\cap 8Q} \ dy \ d\tau, \nonumber \\
  J_2& :=  \iint_{8Q} \lbr |\nabla u|^{p-1}+|\th|^{p-1} \rbr  |\nabla \vl|   \lsb{\chi}{\frac34Q_i\cap 8Q} \ dy \ d\tau,\nonumber \\
  J_3&:= \iint_{8Q} |\tu-\tu^i| | \pa_t (\om_i \vl)|  \lsb{\chi}{\frac34Q_i\cap 8Q} \ dy \ d\tau.\label{bound_J_3_1}
 \end{align}

 Let us now estimate each of the terms as follows: 
 \begin{description}
  \item[Bound for $J_1$:] If $\rho \leq r_i$,  we can directly use H\"older's inequality, Lemma \ref{lemma3.7} and \descref{W4}, to find that for any $\ve \in (0,1]$, there holds 
  \begin{equation}
   \label{bound_I_1_rho_leq_r_i}
   \begin{array}{ll}
    J_1 & \apprle \la |Q_i| \lbr \fiint_{16Q_i} |\nabla u|^{q}+|\th|^q   \ dy \ d\tau\rbr^{\frac{p-1}{q}} \apprle \la^p |4Q_i|\leq \frac{\la^p}{\ve} |4Q_i|.
   \end{array}
  \end{equation}
  In the case $r_i \leq \rho$, we make use of \eqref{lemma3.10_bound3}, \descref{W4} along with the fact $|Q_i| = |B_i| \times 2\la^{2-p} r_i^2$, to get
  \begin{equation}
  \label{bound_I_1_rho_geq_r_i}
   \begin{array}{ll}
    J_1
    & {\apprle} \frac{1}{r_i} \lbr \frac{r_i \la}{\ve} + \frac{\ve}{\la r_i} |\tu^i|^2  \rbr \lbr \iint_{4Q_i} |\nabla u|^{p-1}  +|\th|^{p-1} \ dy \ d\tau \rbr \\
    & {\apprle} \frac{1}{r_i} \lbr \frac{r_i \la}{\ve} + \frac{\ve}{\la r_i}  |\tu^i|^2  \rbr |4Q_i|\la^{p-1} 
     {\apprle} \frac{\la^p}{\ve} |4Q_i| +  \ve |4B_i| |\tu^i|^2.	
   \end{array}
  \end{equation}

Thus combining \eqref{bound_I_1_rho_geq_r_i} and \eqref{bound_I_1_rho_leq_r_i}, we get
\begin{equation}
 \label{bound_I_1}
 J_1 \apprle \frac{\la^p}{\ve} |4Q_i| +  \lsb{\chi}{r_i \leq \rho}\ve |4B_i| |\tuh^i|^2,
\end{equation}
where we have set $\lsb{\chi}{r_i \leq \rho} = 1$ if  $r_i \leq \rho$ and $\lsb{\chi}{r_i \leq \rho}=0$ else.

  \item[Bound for $J_2$:] In this case, we can directly use Lemma \ref{lemma3.9} and \descref{W4} to get for any $\ve \in (0,1]$, the bound
  \begin{equation}
   \label{bound_I_22}
    J_2 
     \apprle \frac{\la}{\ve} \iint_{4Q_i}|\nabla u|^{p-1}+ |\th|^{p-1} \ dy \ d\tau 
   \apprle \frac{\la^p}{\ve} |4Q_i|.
  \end{equation}
  \item[Bound for $J_3$:] Substituting  \eqref{lemma3.10_bound4}, \eqref{lemma3.11.bound2} and \descref{W9} into \eqref{3.128},  for any $\ve \in (0,1]$, there holds
  \begin{equation}
   \label{bound_3_1}
   \begin{array}{ll}
    |\pa_t(\om_i \vl)(z)| 
    & \apprle \frac{1}{\ga r_i^2} \lbr \frac{r_i \la}{\ve}  + \frac{\ve}{r_i \la} |\tuh^i|^2 \rbr. 
   \end{array}
  \end{equation}
Making use of  \eqref{bound_3_1} in the expression for $J_3$ in \eqref{bound_J_3_1}, we get
  \begin{equation*}
   \label{bound_I_3}
   \begin{array}{ll}
    J_3 & \apprle \frac{1}{\ga r_i^2} \lbr \frac{r_i \la}{\ve}  + \frac{\ve}{r_i \la} |\tuh^i|^2 \rbr\iint_{\frac34Q_i} |\tu-\tu^i|    \ dy \ d\tau.
   \end{array}
  \end{equation*}
  We can use \eqref{bound_when_i_1} to get 
%
%
%
%
\begin{equation}
 \label{bound_I33}
 J_3 \apprle \frac{\la^2}{r_i^2 \ve} |Q_i| + \frac{\ve}{\ga r_i^2} |\tu^i|^2 |Q_i| \apprle \frac{\la^p }{\ve} |4Q_i| + \ve |4B_i| |\tu^i|^2.
\end{equation}

 \end{description}

 Substituting the estimates \eqref{bound_I_1}, \eqref{bound_I_22} and \eqref{bound_I33} into \eqref{first_1}, we obtain the proof of  \eqref{3.120}. 
 
 \item[Proof of \eqref{3.121}:] Note that in this case, we have $\ga r_i^2 \apprge s$.  Setting $a =0$ in \eqref{3.123} along with making use of the the bounds  \eqref{3.126} and \eqref{3.127}, we get the following estimate:
 \begin{equation}
  \label{second_1}
  \begin{array}{ll}
   \left| \int_{8B} \lbr[(] \tu  \om_i \vl \rbr (y,t) \ dy \right| & \apprle II_1 + II_2+II_3+II_4,
  \end{array}
 \end{equation}
where we have set 
\begin{equation*}
 \begin{array}{ll}
  II_1 & :=  \frac{1}{\min\{\rho, r_i\}} \iint_{8Q}\lbr  |\nabla u|^{p-1}+  |\th|^{p-1}\rbr |\vl|   \lsb{\chi}{\frac34Q_i\cap 8Q} \ dy \ d\tau, \\
  II_2 & :=  \iint_{8Q} \lbr |\nabla u|^{p-1} +|\th|^{p-1}\rbr  |\nabla \vl|   \lsb{\chi}{\frac34Q_i\cap 8Q} \ dy \ d\tau, \\
  II_3 & :=  \frac1{\min\{s,\ga r_i^2\}} \iint_{8Q} |u| |\vl|  \lsb{\chi}{\frac34Q_i\cap 8Q} \ dy \ d\tau,\\
  II_4 & :=   \iint_{8Q} |u| |\pa_t\vl|  \lsb{\chi}{\frac34Q_i\cap 8Q} \ dy \ d\tau.
 \end{array}
\end{equation*}
We shall now estimate each of the terms as follows:
\begin{description}[leftmargin=*]
 \item[Bound for $II_1$:] 
 In the case $\rho \leq r_i$, we can directly use Lemma \ref{lemma3.7} along with \descref{W4}, to get the estimate
 \begin{equation}
  \label{3.83_1}
    II_1 \apprle \la \iint_{4Q_i} |\nabla u|^{p-1} +|\th|^{p-1} \ dy \ d\tau \apprle  |4Q_i| \la^p.
  \end{equation}
  
  In the case $r_i \leq \rho$, we make use of \eqref{lemma3.10_bound1}, \descref{W4} along with the bounds $|\tu| \leq |u|\lsb{\chi}{8Q}$ and $\ga r_i^2 \apprge s$, to get 
  \begin{equation}
  \label{3.84_1}
   \begin{array}{ll}
    II_1 & \apprle \frac{1}{r_i} \lbr r_i \la + \frac{1}{\la r_i} \fiint_{4Q_i} |\tu|^2 \ dy \ d\tau \rbr \lbr \iint_{4Q_i} |\nabla u|^{p-1}+|\th|^{p-1} \ dy \ d\tau \rbr \\
    & \apprle \la^p |4Q_i| + \frac{1}{s} \iint_{4Q_i} |u|^2 \lsb{\chi}{8Q} \ dy \ d\tau.
   \end{array}
  \end{equation}
Thus combining \eqref{3.83_1} and \eqref{3.84_1}, we get 
\begin{equation}
 \label{II_1_1}
 II_1 \apprle \la^p |4Q_i| + \frac{1}{s} \iint_{4Q_i} |u|^2\lsb{\chi}{8Q} \ dy \ d\tau.
\end{equation}

 \item[Bound for $II_2$:] To bound this term, we make use of \eqref{lemma3.10_bound6}, \descref{W4} along with the observation $|\tu| \leq |u|\lsb{\chi}{8Q}$ and $\ga r_i^2 \apprge s$, to get
 \begin{equation}
  \label{II_2_1}
  \begin{array}{ll}
   II_2 & \apprle \lbr \iint_{4Q_i} |\nabla u|^{p-1} +|\th|^{p-1}\ dy \ d\tau \rbr \lbr \la + \frac{\la^{1-p}}{s} \fiint_{4Q_i} |\tu|^2 \ dy \ d\tau \rbr \\
           & \apprle \la^p |4Q_i| + \frac{1}{s} \iint_{4Q_i} |u|^2 \lsb{\chi}{8Q}\ dy \ d\tau.
  \end{array}
 \end{equation}

 \item[Bound for $II_3$:] For bounding this term, we  use  \eqref{lemma3.10_bound1} along with $\ga r_i^2 \apprge s$ and $|\tu|\leq|u| \lsb{\chi}{8Q}$,  to  get
 \begin{equation}
  \label{II_3_1}
  \begin{array}{ll}
   II_3 & \apprle \frac{1}{s} |4Q_i|\lbr \fiint_{4Q_i} |u| \lsb{\chi}{8Q} \ dy \ d\tau\rbr\lbr  \fiint_{4Q_i} |\tu| \ dy \ d\tau \rbr \\
   & \apprle \frac{1}{s} \iint_{4Q_i} |u|^2 \lsb{\chi}{8Q} \ dy \ d\tau.
  \end{array}
 \end{equation}

 \item[Bound for $II_4$:] In this case, we make use of \eqref{lemma3.11.bound1} along with  $\ga r_i^2 \apprge s$ and $|\tu|\leq|u| \lsb{\chi}{8Q}$, to  get
 \begin{equation}
  \label{II_4_1}
  \begin{array}{ll}
   II_4 & \apprle  |4Q_i|\lbr \fiint_{4Q_i} |u| \lsb{\chi}{8Q} \ dy \ d\tau\rbr\lbr \frac{1}{\ga r_i^2} \fiint_{4Q_i} |\tu| \ dy \ d\tau \rbr \\
   & \apprle \frac{1}{s} \iint_{4Q_i} |u|^2 \lsb{\chi}{8Q} \ dy \ d\tau.
  \end{array}
 \end{equation}
\end{description}
We finally combine the estimates \eqref{II_1_1}, \eqref{II_2_1}, \eqref{II_3_1} and \eqref{II_4_1} and recalling \eqref{second_1},  we obtain \eqref{3.121}.
\end{description}
This completes the proof of the lemma. 
\end{proof}

We now come to  essentially the most important estimate which will be used to obtain the Caccioppoli inequality in Section \ref{Caccioppoli_inequality}.
\begin{lemma}
 \label{crucial_lemma}
 There exists a positive constant $C{(n,p,q, \lamot,b_0)}$ such that the following estimate holds for every $t \in 2\tm$:
 \begin{equation}
 \label{3122}
  \int_{8B \setminus \elam^t} (|\tu|^2 - |\tu - \vl|^2)(x,t) \ dx \geq C \lbr - \la^p |\RR^{n+1} \setminus \elam|  - \frac{1}{s} \iint_{8Q} |\tu|^2 \ dz \rbr.
 \end{equation}

\end{lemma}
\begin{proof}
 Let us fix any $t\in 2\tm$ and any point $x \in 8B \setminus \elam^t$.  Now define
 \begin{equation}
  \tTh := \left\{  i \in \Th: \spt(\om_i) \cap 8B \times \{t\} \neq \emptyset \quad \text{and} \quad |\tu| + |\vl| \neq 0 \quad  \text{on}\ \spt(\om_i)\cap (8B \times \{t\}) \right\}.
 \end{equation}

 From \eqref{theta}, we see that if $\spt(\om_i) \cap 8B \times \{t\} \neq \emptyset$, then  $i \in \Th$.  If $i \neq \tTh$, then $\tu = \vl = 0$ on $\spt(\om_i) \cap 8B \times \{t\}$, which  implies
 \[
  \int_{\spt(\om_i) \cap 8B \times\{t\}} |\tu|^2 - |\tu - \vl|^2 \ dx = 0.
 \]
 Hence we only need to consider $i \in \tTh$.  We now decompose  $\tTh = \tTh_1 \cup \tTh_2$, where $\tTh_1 := \tTh \cap \Th_1$ and $\tTh_2 := \tTh \cap \Th_2$. Noting that $\sum_{i \in \tTh} \om_i(\cdot, t) \equiv 1$ on $\RR^n \cap \elam^t$, we can rewrite the left-hand side of \eqref{3122} as 
 \begin{equation}
  \label{3124}
  \int_{8B \setminus \elam^t} (|\tu|^2 - |\tu - \vl|^2)(x,t) \ dx = J_1 + J_2,
 \end{equation}
where we have set
\begin{gather*}
 J_1:= \sum_{i \in \tTh_1} \int_{8B} \om_i (|\tu|^2 - |\tu - \vl|^2) \ dx, \qquad 
 J_2:= \sum_{i \in \tTh_2} \int_{8B} \om_i (|\tu|^2 - |\tu - \vl|^2) \ dx 
\end{gather*}
We shall now estimate each of the terms as follows:
\begin{description}[leftmargin=*]
 \item[Estimate of $J_1$:] We can further rewrite this term as follows:
 \begin{equation}
  \label{I_1}
  \begin{array}{ll}
  J_1 & = \sum_{i \in \tTh_1} \int_{8B} \om_i(z) \lbr |\tu^i|^2  + 2 \vl (\tu - \tu^i) \rbr \ dx - \sum_{i \in \tTh_1} \int_{8B} \om_i(z) |\vl - \tu^i|^2 \ dx\\
  & := J_1^1 + J_1^2.
  \end{array}
 \end{equation}
%
 \begin{description}
  \item[Estimate of $J_1^1$:] Using \eqref{3.120}, we get
  \begin{equation}
   \label{I_1_1}
   J_1^1 \apprge \sum_{i \in \tTh_1} \int_{8B} \om_i(z)  |\tu^i|^2 \ dz - \ve \sum_{i \in \tTh_1} |4B_i| |\tu^i|^2 - \sum_{i \in \tTh_1} \frac{\la^p}{\ve} |4Q_i|.
  \end{equation}
  From \eqref{def_tuh_i}, we have $\tu^i = 0$ whenever $\spt(\om_i) \cap \Om_{8\rho}^c \neq \emptyset$. Hence we only have to sum over all those $i \in \tTh_1$ for which $\spt(\om_i) \subset \mcc_{8\rho}$.  In this case, we make use of a suitable choice for $\ve \in (0,1]$, and use \descref{W7} along with \descref{W8}, to estimate \eqref{I_1_1} from below. We have
  \begin{equation}
 \label{bound_I1}
 J_1^1 \apprge -\la^p |\RR^{n+1} \setminus \elam|. 
\end{equation}

%
%
%
%
%

  \item[Estimate of $J_1^2$:] For any $x \in 8B \setminus \elam^t$, we have from \descref{W10} that $\sum_{j} \om_j(x,t) = 1$, which gives
  \begin{equation}
   \label{I_2_1}
   \begin{array}{ll}
    \om_i(z) |\vl(z) - \tu^i|^2  
    & \apprle \om_i(z)  \sum_{j \in A_i} |\om_j(z)|^2  \lbr \tu^j - \tu^i\rbr^2 \\
    & \overset{\redlabel{3.104.a}{a}}{\apprle} \min\{ \rho, r_i\}^2 \la^2.
   \end{array}
  \end{equation}
To obtain \redref{3.104.a}{a} above, we made use of Lemma \ref{lemma3.8} (recall $i \in \tTh_1 \subset \Th_1$) along with \descref{W13}.  Substituting \eqref{I_2_1} into the expression for $J_1^2$ and using $|Q_i| = |B_i| \times 2\ga r_i^2$, we get
\begin{equation}
 \label{bound_I_2}
J_1^2 
\apprle 
 \sum_{i \in \tTh_1} |8B \cap \frac34B_i| \frac{\ga r_i^2}{\ga} \la^2 
 \apprle \la^p |\RR^{n+1} \setminus \elam|.
\end{equation}
 \end{description}
Substituting \eqref{bound_I1} and \eqref{bound_I_2} into \eqref{I_1}, we get
\begin{equation}
 \label{bound_I}
 J_1 \apprge - \la^p |\RR^{n+1} \setminus \elam|.
\end{equation}

 \item[Estimate of $J_2$:] From the identity $|\tu|^2 - |\tu-\vl|^2 = 2\tu \vl - |\vl|^2$, we get
 \begin{equation}
  \label{II1}
  \begin{array}{ll}
  \sum_{i \in \tTh_2} \left| \int_{8B} \om_i (|\tu|^2 - |\tu - \vl|^2) \ dx\right| & \apprle \sum_{i \in \tTh_2} \left| \int_{8B} \om_i \tu \vl \ dx\right| + \sum_{i \in \tTh_2} \left| \int_{8B} \om_i |\vl|^2 \ dx\right| \\
  & := J_2^1 + J_2^2.
  \end{array}
 \end{equation}
\begin{description}
 \item[Estimate for $J_2^1$:]  This term can be easily estimated using \eqref{3.121} followed by summing over $i \in \tTh_2$ and using \descref{W13}, we get
 \begin{equation}
  \label{II1_1}
  \begin{array}{ll}
   \sum_{i \in \tTh_2} \left| \int_{8B} \om_i \tu \vl \ dx\right| & \apprle \sum_{i \in \tTh_2}\lbr  \la^p |Q_i| + \frac{1}{s} \iint_{4Q_i} |u|^2 \lsb{\chi}{8Q}\ dz\rbr \\
   & \apprle \la^p |\RR^{n+1} \setminus \elam| + \frac{1}{s} \iint_{8Q} |u|^2 \ dz.
  \end{array}
 \end{equation}

 \item[Estimate for $J_2^2$:] To estimate this term, we make use of \eqref{lemma3.10_bound1} along with $|Q_i| = |B_i| \times 2\ga r_i^2$, $\ga r_i^2 \apprge s$, the bound $|\tu| \leq |u| \lsb{\chi}{8Q}$ and \descref{W13}, to get
 \begin{equation}
  \label{II2_1}
  \begin{array}{ll}
   \sum_{i \in \tTh_2} \left| \int_{8B} \om_i |\vl|^2 \ dx\right| & \apprle \sum_{i \in \tTh_2}  \lbr \fiint_{4Q_i} |\tu| \ dz \rbr^2 |8B \cap \frac34B_i| 
    \apprle \frac{1}{s} \iint_{8Q} |u|^2 \ dz.
  \end{array}
 \end{equation}
\end{description}
 We combine \eqref{II1_1}, \eqref{II2_1} and \eqref{II1} to obtain
 \begin{equation}
  \label{bound_II}
  J_2\apprle \la^p |\RR^{n+1} \setminus \elam| + \frac{1}{s} \iint_{8Q} |u|^2 \ dz.
 \end{equation}
\end{description}
Thus, from \eqref{bound_I}, \eqref{bound_II} and\eqref{3124}, the proof of the Lemma follows.
%
%
\end{proof}

\section{Caccioppoli inequality and Reverse H\"older inequality}\label{Caccioppoli_inequality}
We now have all the estimates need to prove a Caccioppoli type inequality near the boundary. \emph{In subsequent calculations, we shall alternate between $\mcc_{8\rho} =\Om_{8\rho} \times 8\tm$ and $8Q$ depending on convenience, as $\vlh(z) = \tu(z) =u(z)= 0$ for all $z \in \mcc_{8\rho}^c$.}

Before we state Lemma \ref{Caccioppoli}, let us make the following remark:
\begin{remark}
 \label{def_al_0_1}
 Under the hypothesis of Lemma \ref{Caccioppoli} (more specifically suppose \eqref{hypothesis_1} holds), then we see that there exists a universal constant $c_e = c_e(n)$ such that for any $\la \geq c_e \al_0$, there holds $\elam \neq \emptyset$. This follows from Lemma \ref{bound_g_x_t} and \eqref{hypothesis_1} below. 
\end{remark}

\subsection{Caccioppoli Inequality}
We shall prove the Caccioppoli inequality in this subsection.
\begin{lemma}
 \label{Caccioppoli}
 Let $\ka \geq 1$ and $\Om$ satisfy Definition \ref{p_thick_domain} with constants $r_0,b_0$. Then there exists  constants $\be_0 = \be_0(n,p,q,\lamot,b_0,\ka) \in (0,1)$ small and $C = C(n,p,q,\La_0,\La_1,\ka,b_0)$ such that the following holds.
Let $Q := B \times \tm = Q_{\rho,s}(z)$ be a parabolic cylinder  with $s = \al_0^{2-p} \rho^2$   for some $\al_0>0$ as given in \eqref{def_s} and let $u \in L^2(-T,T;L^2(\Om)) \cap L^{p-\be}(-T,T; W_0^{1,p-\be}(\Om))$ for some $\be \in (0,\be_0)$ be any very weak solution of \eqref{main} in the sense of \eqref{very_weak_solution}.  Suppose the following bound holds:
 \begin{gather}
  \frac{\al_0^{p-\be}}{\ka} \leq \fiint_Q (|\nabla u|+|\th|)^{p-\be} \ dz \quad \text{and} \quad \fiint_{8Q} (|\nabla u|+|\th|)^{p-\be} \ dz \leq \ka \al^{p-\be}_0 \label{hypothesis_1},
 \end{gather}
 where $h$ is from \eqref{bound_b}. Then we have
 \begin{equation*}
  \label{conclusion_1}
  \al_0^{p-\be} + \al_0^{p-2} \sup_{t \in\tm} \hint_B \left| \frac{u(x,t)}{\rho} \right|^2 \mathcal{M}(x,t)^{-\be} \ dx \leq {\cac}\fiint_{8Q}\lbr  \al_0^{p-2-\be} \left| \frac{u}{\rho} \right|^2 + \left| \frac{u}{\rho} \right|^{p-\be} + |\th|^{p-\be} \rbr dz
 \end{equation*}
 where we have set $\mathcal{M}(x,t):=\max\{g(x,t), \al_0\}$.
\end{lemma}

\begin{proof}[Proof of Lemma \ref{Caccioppoli}]
Let $(t_1 ,t) \subset 2\tm$ be a time interval and recall the definition of $\vlh$ from \eqref{lipschitz_extension}.  We  shall make use of  $\eta(x)\vlh(x,\tau)$ as a test function in \eqref{very_weak_solution} (this is possible since $\spt(\vlh) \subset \Om_{8\rho} \times 8\tm$ and $\vlh C^{0,1}(2\mch)$ from Lemma \ref{lemma3.15}). Thus, after integrating the resulting expression of \eqref{very_weak_solution} over $(t_1,t)$, we get
%
 \begin{equation}
 \label{cac1}
  L_1 + L_2:=\int_{t_1}^t \left[ \int_{\Om_{8\rho}} \frac{d{[u]_h}}{d\tau} \eta(y) \vlh(y,\tau) + \iprod{[\aa(y,\tau,\nabla u)]_h}{\nabla (\eta \vlh)} \ dy\right]\ d\tau = 0.
 \end{equation}

 \begin{description}[leftmargin=*]
  \item[Estimate of $L_1$:] Note that $\zeta(\tau) = 1$ for all $\tau \in (t_1,t)$. Thus making use of the standard hole filling technique and  \eqref{def_u_tilde}, we get
  \begin{equation}
   \label{6.38}
   \begin{array}{ll}
    L_1 
    & = \int_{t_1}^t \int_{\Om_{8\rho}} \dds{\tu_h(y,\tau)}  \vlh(y,s) \ dy\ d\tau \\
    & = \int_{t_1}^t \int_{\Om_{8\rho}\setminus \elam^{\tau}}  \dds{\vlh}  (\vlh-\tu_h) \ dy \ d\tau +    \int_{t_1}^t \int_{\Om_{8\rho}}  \frac{d{\lbr (\tu_h)^2 - (\vlh - \tu_h)^2 \rbr }}{d\tau}  \ dy \ d\tau \\
    & := J_2 + J_1(t) - J_1(t_1),
   \end{array}
  \end{equation}
  where we have set 
  \begin{equation*}
  \label{def_i_1}J_1(\tau) := \frac12 \int_{\Om_{8\rho}} ( (\tu_h)^2 - (\vlh - \tu_h)^2 ) (y,\tau) \ dy.
  \end{equation*}

  \begin{description}
   \item[Estimate for $J_2$:] Taking absolute values and making use of Lemma \ref{lemma3.14} with $\vartheta =1$, we get
   \begin{equation}
    \label{6.39}
    \begin{array}{ll}
     |J_2| 
     & \apprle \iint_{8Q\setminus \elam}   \left| \dds{\vlh}  (\vlh-\tu_h)\right| \ dy \ d\tau 
       \apprle\la^p |\RR^{n+1} \setminus \elam| + \frac{1}{s} \iint_{8Q} |[u]_h|^2 \ dy \ d\tau. 
    \end{array}
   \end{equation}

   \item[Estimate for $J_1(t_1)$:] We first claim that we can choose $t_1\in (-2^2s,-s)$ such that
   \begin{equation}
    \label{6.40}
    |J_1(t_1)|  
    \leq \frac{1}{s} \int_{-2^2s}^{-s} \left| \int_{\Om_{8\rho}} \lbr |\tu_h|^2 - |\vlh - \tu_h|^2 (x,\tau)\rbr \ dy\right| d\tau.
   \end{equation}
Suppose not, then for any $t_1 \in (-2^2s,-s)$, we would have 
\begin{equation}
\label{J_1_t_11}
 \left| \int_{\Om_{8\rho}} \lbr |\tu_h|^2 - |\vlh - \tu_h|^2\rbr (y,t_1) \ dy \right| > \frac{1}{s} \int_{-2^2s}^{-s} \left| \int_{\Om_{8\rho}} \lbr |\tu_h|^2 - |\vlh - \tu_h|^2 \rbr(y,\tau) \ dy\right| d\tau.
\end{equation}
But we have
\begin{equation}
\label{J_1_t_22}
 3 \hint_{-2^2s}^{-s} \left| \int_{\Om_{8\rho}} \lbr |\tu_h|^2 - |\vlh - \tu_h|^2 (y,\tau)\rbr \ dy\right| d\tau \geq 3 \min_{\tau \in (-4s,-s)} |J_1(\tau)|.
\end{equation}
Combining \eqref{J_1_t_11} and \eqref{J_1_t_22}, for any $t_1 \in (-4s,-s)$, we get $|J_1(t_1)| > 3 \min_{\tau \in (-4s,-s)} |J_1(\tau)|$,
which is absurd. Hence \eqref{6.40}  must be true for some $t_1 \in (-4s,-s)$.

From the construction \eqref{lipschitz_extension},  we have $\vlh = \tu_h$ on $\elam$. Furthermore,  $\spt(\tu_h) \subset 8Q$ and $|\tu_h| \leq |[u]_h|\lsb{\chi}{8Q}$ holds. For $t_1 \in (-2^2s,-s)$ satisfying \eqref{6.40}, we then get
\begin{equation}
 \label{6.44}
 \begin{array}{ll}
  |J_1(t_1)| 
& \leq   \frac{1}{s} \iint_{8Q\cap \elam } \left| |\tu_h|^2 - |\vlh - \tu_h|^2 \right| dz + \frac{1}{s} \iint_{8Q\setminus \elam} \left| |\tu_h|^2 - |\vlh - \tu_h|^2 \right| dz \\
& \apprle   \frac{1}{s} \iint_{8Q\cap \elam }  |\tu_h|^2  dz + \frac{1}{s} \iint_{8Q\setminus \elam}  |\tu_h|^2 + |\vlh |^2  dz \\
& \overset{\redlabel{4.8.a}{a}}{\apprle} \frac{1}{s} \iint_{8Q} |[u]_h|^2 \ dz.
 \end{array}
\end{equation}
To obtain \redref{4.8.a}{a}, we used Lemma \ref{lemma3.12} (applied with $\vartheta = 2$).

  \end{description}

  \item[Estimate for $L_2$:] 
   We decompose the expression as 
   \begin{equation}\label{4.9}
    \begin{array}{ll}
     L_2 & = \int_{t_1}^t \int_{\elam^{\tau}} \iprod{[\aa(y,\tau,\nabla u)]_h}{\nabla (\eta \vlh)} \ dy\ d\tau  + \\
     & \qquad \quad + \int_{t_1}^t \int_{\Om_{8\rho} \setminus \elam^{\tau}} \iprod{[\aa(y,\tau,\nabla u)]_h}{\nabla (\eta \vlh)} \ dy\ d\tau \\
     & := L_2^1 + L_2^2.
    \end{array}
   \end{equation}

   \begin{description}
    \item[Estimate for $L_2^2$:] Using the chain rule, \eqref{abounded}, \eqref{cut_off_function} along with Corollary \ref{corollary3.20}, we get
    \begin{equation}
    \label{4.10}
     \begin{array}{ll}
      L^2_2 
      & \leq \int_{t_1}^t \int_{\Om_{8\rho} \setminus E_{\tau}(\la)} {[|\nabla u|^{p-1}+|\th|^{p-1}]_h}{|\nabla (\eta \vlh)|} \ dy\ d\tau  \\
      & \apprle \sum_{k=0}^1  \rho^{-k+1}\iint_{(8B\times 2\tm) \setminus \elam} {[|\nabla u|^{p-1}+|\th|^{p-1}]_h}{|\nabla^k  \vlh|} \ dy\ d\tau  \\
&\apprle \la^p |\RR^{n+1} \setminus \elam| + \frac{1}{s} \iint_{8Q} |\tuh|^2 \ dy\ d\tau.
     \end{array}
    \end{equation}
\end{description} 
\end{description}

Substituting \eqref{4.10} into \eqref{4.9} and \eqref{6.39}, \eqref{6.44} into \eqref{6.38},  and finally making use of \eqref{cac1} along with the bound $|\tuh| \leq |[u]_h|\lsb{\chi}{8Q}$, we get
 \begin{equation}
 \label{6.45}
  \begin{array}{l}
    \frac12 \int_{\Om_{8\rho}} | (\tu_h)^2 - (\vlh - \tu_h)^2 | (y,t) \ dy +  \int_{t_1}^t \int_{E_{\tau}(\la)} \iprod{[\aa(y,\tau,\nabla u)]_h}{\nabla (\eta \vlh)} \ dy \ d\tau   \\
   \hspace*{6cm} \apprle \la^p |\RR^{n+1} \setminus \elam| + \frac{1}{s} \iint_{8Q} |[u]_h|^2 \ dy \ d\tau.
  \end{array}
 \end{equation}

 Since the estimate in \eqref{6.45} is independent of $h$, we can let $h \searrow 0$ to get
 \begin{equation}
 \label{6.46.1}
  \begin{array}{l}
    \frac12 \int_{\Om_{8\rho}} | (\tu)^2 - (\vl - \tu)^2 | (y,t) \ dy +  \int_{t_1}^t \int_{E_{\tau}(\la)} \iprod{\aa(y,\tau,\nabla u)}{\nabla (\eta \vl)} \ dy \ d\tau   \\
   \hspace*{6cm} \apprle \la^p |\RR^{n+1} \setminus \elam| + \frac{1}{s} \iint_{8Q} |u|^2 \ dy \ d\tau.
  \end{array}
 \end{equation}
%
%
%
  Using the fact that $\tu = \vl$ on $\elam$ and Lemma \ref{crucial_lemma},  we get
 \begin{equation}
 \label{4.13}
  \begin{array}{ll}
    \int_{\Om_{8\rho}} | (\tu)^2 - (\vl - \tu)^2 | (y,t) \ dy   
   & \apprge \int_{\elam^t} | \tu (x,t)|^2 \ dx   - \la^p |\RR^{n+1} \setminus \elam| - \frac{1}{s} \iint_{8Q} |u|^2 \ dz. 
  \end{array}
 \end{equation}
%
%
Making use of the bound \eqref{4.13} into \eqref{6.46.1}, followed by multiplying the resulting expression with $\la^{-1-\be}$ and integrating from $(c_e\al_0,\infty)$ (recall that  $c_e$ is as in Remark \ref{def_al_0_1}), for almost every $t \in (-s,s)$, we get 
 \begin{equation}
 \label{K_expression}
K_1 + K_2 \apprle K_3 + K_4,
 \end{equation}
 where we have set
 \begin{equation*}
  \begin{array}{@{}r@{}c@{}l@{}}
   K_1 \ &:=&\  \frac12\int_{c_e \al_0}^{\infty} \la^{-1-\be} \int_{\elam^t} | \tu(y,t)|^2 \ dy \ d\la, \\
  K_2 \ &:=& \ \int_{c_e \al_0}^{\infty} \la^{-1-\be}\int_{t_1}^t \int_{\elam^{\tau}} \iprod{\aa(y,\tau,\nabla u)}{\nabla (\eta u)} \ dy\ d\tau \ d\la, \\
  K_3\  &:=& \ \int_{c_e \al_0}^{\infty} \la^{-1-\be}  \la^p |\RR^{n+1} \setminus \elam| \  d\la, \\
  K_4 \ &:=& \ \frac{1}{s} \int_{c_e \al_0}^{\infty} \la^{-1-\be}   \iint_{8Q} |u|^2 \ dy  \ d\tau \ d\la .
  \end{array}
 \end{equation*}

 We now define the truncated Maximal function $\mathcal{M}(z) := \max \{ g(z), \al_0\}$ and then estimate each of the $K_i$ for $i \in \{1,2,3,4\}$ as follows:
 \begin{description}
  \item[Estimate for $K_1$:]  By applying Fubini, we get
  \begin{equation}
  \label{4.19}
   \begin{array}{ll}
    K_1 & \geq \int_{8B} |\tu(y,t)|^2 \int_{c_e g(y,t)}^{\infty} \la^{-1-\be} \ d\la \ dy 
     \apprge \frac{1}{\be c_e^{\be}}  \int_{8B} \mathcal{M}(y,t)^{-\be} | \tu(y,t)|^2 \ dy .
   \end{array}
  \end{equation}

  \item[Estimate for $K_2$:] Again applying Fubini, we get
  \begin{equation}
  \label{4.20}
   \begin{array}{ll}
    K_2 & = \frac{1}{\be c_e^{\be}} \int_{t_1}^s \int_{\Om_{8\rho}} \mathcal{M}(y,\tau)^{-\be} \iprod{\aa(y,\tau,\nabla u)}{\nabla (\eta u)} \ dy \ d\tau. 
   \end{array}
  \end{equation}
  Applying the chain rule and using \eqref{def_u_tilde} along with \eqref{abounded} and \eqref{bound_b}, we  get
  \begin{equation}
\label{4.25}
 \begin{array}{ll}
  K_2 &= \int_{t_1}^s \int_{\Om_{8\rho}} \mathcal{M}(y,\tau)^{-\be} \iprod{\aa(y,\tau,\nabla u)}{\nabla u} \eta^2  \ dy \ d\tau   \\
  & \qquad + \int_{t_1}^s \int_{\Om_{8\rho}} \mathcal{M}(y,\tau)^{-\be} \iprod{\aa(y,\tau,\nabla u)}{\nabla \eta^2 } u \ dy \ d\tau  \\
  &\apprge \int_{t_1}^s \int_{\Om_{8\rho}} \mathcal{M}(y,\tau)^{-\be} |\nabla u|^p \eta^2  \ dy \ d\tau  - \iint_{8Q} \mathcal{M}(y,\tau)^{-\be} |\th|^p   \ dy \ d\tau  \\
  & \qquad - \int_{t_1}^s \int_{\Om_{8\rho}} \mathcal{M}(y,\tau)^{-\be} \lbr |\nabla u|^{p-1} + |\th|^{p-1} \rbr \frac{|u|}{\rho} \ dy \ d\tau  \\
  & \apprge  \iint_{Q} \mathcal{M}(y,\tau)^{-\be} |\nabla u|^p \eta^2  \ dy \ d\tau  - \iint_{8Q} \mathcal{M}(y,\tau)^{-\be} |\th|^p   \ dy \ d\tau \\
  & \qquad - \iint_{8Q} \mathcal{M}(y,\tau)^{-\be} \lbr |\nabla u|^{p-1} + |\th|^{p-1} \rbr \frac{|u|}{\rho} \ dy \ d\tau  \\
  & := \aa_1 + \aa_2 + \aa_3.
 \end{array}
\end{equation}

\begin{description}
 \item[Estimate for $\aa_1$:] Note that $\eta \equiv 1$ on $B$.  Let $S :=\{ z \in Q: |\nabla u(z)| \geq \be g(z)\},$ then we get
 \begin{equation}
 \label{4.26}
  \begin{array}{@{}l@{}c@{}l@{}}
   \iint_Q |\nabla u|^{p-\be} \ dz  & =& \iint_S |\nabla u|^{p-\be} \ dz + \iint_{Q\setminus S} |\nabla u|^{p-\be} \ dz \\
   & \leq &\be^{-\be} \iint_Q \mathcal{M}(z)^{-\be} |\nabla u|^p \ dz + \be^{p-\be} \iint_{Q\setminus S} \mathcal{M}(z)^{p-\be} \ dz \\
            & \overset{\text{Lemma \ref{bound_g_x_t}}}{\apprle} & \iint_Q \mathcal{M}(z)^{-\be} |\nabla u|^p \ dz + \be^{p-\be} \iint_{8Q} |\nabla u|^{p-\be} \ dz  + \be^{p-\be} |Q| \al_0^{p-\be}.
      \end{array}
 \end{equation}
  \item[Estimate for $\aa_2$:]  From \eqref{def_g}, we obtain the bound  $\lsb{\chi}{8Q}( |\nabla u(z)| + |\th(z)|)  \leq \mathcal{M}(z)$ for a.e $z \in \RR^n$. This gives
  \begin{equation}
   \label{4.26.1}
 \aa_2 =  \iint_{8Q} \mathcal{M}(z)^{-\be} |\th|^p   \ dz \apprle  \iint_{8Q}  |\th|^{p-\be}   \ dz.
  \end{equation}

 \item[Estimate for $\aa_3$:]  We use the bound  $\lsb{\chi}{8Q}( |\nabla u(z)| + |\th(z)|)  \leq \mathcal{M}(z)$ for a.e $z \in \RR^n$, along with Young's inequality, Theorem \ref{sobolev-poincare} and  \eqref{hypothesis_1}, to get
 \begin{equation}
 \label{4.27}
  \begin{array}{rcl}
   \aa_3 & \apprle& \iint_{8Q}  (|\nabla u|+ |\th|) ^{p-1-\be} \frac{|u|}{\rho} \ dz
   \apprle  \varepsilon |Q| \ka \al_0^{p-\be} + C(\varepsilon) \iint_{8Q} \left| \frac{u}{\rho}\right|^{p-\be}\ dz.
  \end{array}
 \end{equation}

\end{description}

  \item[Estimate for $K_3$:] Applying the layer-cake representation (see for example \cite[Chapter 1]{Grafakos}), we get
  \begin{equation}
  \label{4.21}
   \begin{array}{rcl}
    K_3  & =& \frac{1}{p-\be} \iint_{\RR^{n+1}} \mathcal{M}(z)^{p-\be} \ dz
     \overset{\text{Lemma \ref{bound_g_x_t}}}{\apprle} \frac{1}{p-\be} \iint_{8Q}  \lbr |\nabla u|^{p-\be}+ |\th|^{p-\be}+\al_0^{p-\be}\rbr  \ dz \\
    & \overset{\text{\eqref{hypothesis_1}}}{\apprle} &\al_0^{p-\be} |Q|.  
   \end{array}
  \end{equation}
  \item[Estimate for $K_4$] Again applying Fubini, we get 
  \begin{equation}
  \label{4.22}
   \begin{array}{ll}
    K_4 & =\frac{1}{s} \int_{{c_e \al_0}}^{\infty} \la^{-1-\be}   \iint_{8Q} |u|^2 \ dz \ d\la  = \frac{1}{\be} \iint_{8Q}  (\al_0)^{-\be}  \frac{|u|^2}{s} \ dz.
   \end{array}
  \end{equation}
 \end{description}
%
%
Substituting \eqref{4.26}, \eqref{4.26.1} and \eqref{4.27} into \eqref{4.25} followed by combining  \eqref{4.19}, \eqref{4.21} and \eqref{4.22} into \eqref{K_expression}, we get
\begin{equation*}
\label{4.28}
 \begin{array}{ll}
 \frac{1}{2\be}  \int_{B} \mathcal{M}(y,t)^{-\be} | \tu(y,t)|^2 \ dy + \frac{1}{\be} \iint_Q |\nabla u|^{p-\be} \ dz \apprle \frac{1}{\be} \be^{p-\be} \iint_{8Q} |\nabla u|^{p-\be} \ dz +\iint_{8Q} |\th|^{p-\be} \ dz\\
    \hspace*{2cm}+ \frac{1}{\be} \varepsilon |Q| \al_0^{p-\be} + \frac{1}{\be} C(\varepsilon) \iint_{8Q} \left| \frac{u}{\rho}\right|^{p-\be}\ dz+  \al_0^{p-\be} |Q| + \frac{1}{\be} \iint_{8Q} \al_0^{-\be}  \frac{|u|^2}{s} \ dz.
 \end{array}
\end{equation*}
Multiplying the above expression  by $\be$ and then using the intrinsic scaling $s = \rho^2 \al_0^{2-p}$ from \eqref{def_s} along with \eqref{hypothesis_1}, we get 
\begin{equation*}
 \begin{array}{ll}
        \int_{B} \mathcal{M}(y,t)^{-\be} | \tu(y,t)|^2 \ dy + |Q| \al_0^{p-\be}  \apprle \be^{p-\be} |Q| \al_0^{p-\be} + \varepsilon |Q| \al_0^{p-\be} +\iint_{8Q} |\th|^{p-\be} \ dz \\
    \hspace*{2cm}  +  C(\varepsilon) \iint_{8Q} \lbr \frac{|u|}{\rho} \rbr^{p-\be} \ dz + \be \al_0^{p-\be} |Q| +  \iint_{8Q} \al_0^{p-2-\be}   \lbr \frac{|u|}{\rho}\rbr^2 \ dz.
 \end{array}
\end{equation*}

Choosing $\be \in (0,\be_0)$  and $\varepsilon \in (0,1)$ small, we then make use of  \eqref{hypothesis_1} and the observation $|Q| \approx |B| s^2 = \rho^2 \al_0^{2-p}|B|$, to get
\begin{equation*}
\label{4.31}
        \sup_{t \in \tm} \al_0^{p-2} \hint_{B} \mathcal{M}(y,t)^{-\be} \left| \frac{ \tu(y,t)}{\rho}\right|^2 \ dy +  \al_0^{p-\be}   \ \apprle \   \fiint_{8Q}\left[ \al_0^{p-2-\be}   \lbr \frac{|u|}{\rho}\rbr^2 +  \lbr \frac{|u|}{\rho} \rbr^{p-\be} + |\th|^{p-\be} \right]\ dz.
\end{equation*} 
This completes the proof of the Lemma. 
\end{proof}

\subsection{Estimates needed to prove Reverse H\"older inequality}

\begin{lemma}
\label{lemma7.4}
Let $\ka \geq 1$, then there exists $\be_0(n,p,q,\lamot,b_0,\ka)$. Let  $\be \in (0,\be_0)$ be given and consider any very weak solution $u \in L^2(-T,T;L^2(\Om)) \cap L^{p-\be}(-T,T; W_0^{1,p-\be}(\Om))$. Let $Q_{\rho,s}(0,0) = B_{\rho}(0) \times \tm_s(0)$ be the parabolic cylinder with $(0,0) \in \pa \Om \times (-T,T)$ and $s = \rho^2 \al_0^{2-p}$ for some $\al_0 >0$ as in \eqref{def_s}. Let $\al Q$ be a rescaled parabolic cylinder for some $\al \in (1,8]$ and let $p-\varepsilon_0$  be as given in \eqref{def_q} and also suppose that 
 \begin{equation}\label{hypothesis_2}
  \fiint_{ \al Q} \lbr |\nabla u|+|\th| \rbr^{p-\be} \ dz \leq \ka \al_0^{p-\be}.
 \end{equation}
 Let us define
\begin{equation}
 \label{def_J}
 J:= \sup_{t\in \tm} \hint_{ B} \lbr \frac{|u|}{\rho}\rbr^2 \tilde{g}(x,t)^{-\be} \ dx,
\end{equation}
where $\tilde{g}(z) := \max\{ \mm( |\nabla \tu|^q \lsb{\chi}{\al Q})^{\frac{1}{q}}(z), \mm(( |\nabla u|^q+|\th|^q )\lsb{\chi}{\al Q})^{\frac{1}{q}}(z), \al_0\}$.
 Then there holds:
\begin{equation}
\label{lemma4.3.1}
 \fiint_{ Q} \lbr \frac{|u|}{\rho} \rbr^{p-\be} \ dz \leq C_{(n,p,q,\lamot,b_0,\ka)} (\al_0^{\be} J)^{\frac{\varepsilon_0-\be}{2}}  \lbr \fiint_{ Q} |\nabla u|^q \ dz\rbr^{\frac{p-\varepsilon_0}{q}}  .
\end{equation}
 If $\frac{p-\varepsilon_0}{2} \geq 1$, then there exists an $\tga = \tga(n,p,\varepsilon_0) \in (0,1)$ such that the following holds:
 \begin{equation}
  \label{lemma4.3.2}
 \begin{array}{ll}
\fiint_{ Q} \lbr \frac{|u|}{\rho} \rbr^{2} \ dz & \leq C_{(n,p,q,\lamot,b_0,\ka)} (\al_0^{\be} J)^{1-\tga}\lbr  \fiint_{ Q} |\nabla u|^{q} \ dz\rbr^{\frac{2\tga}{q}} .
\end{array}
\end{equation}
 If $\frac{p-\varepsilon_0}{2} \in (0,1)$, then there holds
\begin{equation}
\label{lemma4.3.3}
 \begin{array}{ll}
\fiint_{ Q} \lbr \frac{|u|}{\rho} \rbr^{2} \ dz & \leq C_{(n,p,q,\lamot,b_0,\ka)} (\la_0^{\be} J)^{\frac{2-p+\varepsilon_0}{2}}\lbr  \fiint_{ Q} |\nabla u|^{q} \ dz\rbr^{\frac{p-\varepsilon_0}{q}} .
\end{array}
\end{equation}
\end{lemma}

\begin{proof}
 In order to prove the lemma, we want to make use of Lemma \ref{interpolation_capacity_version_main}. Let us prove each of the assertions as follows:
 
 

\begin{description}[leftmargin=*]
 \item[Proof of \eqref{lemma4.3.1}:] In this case, in  order to apply Lemma \ref{interpolation_capacity_version_main}, we make the following choice of exponents:
 \begin{equation}
 \label{exp_ch}
\sig =  p-\be, \quad  \tr = \frac{2(p-\be)}{p}, \quad \vartheta = p-\varepsilon_0 \qquad \text{and} \qquad \tal = \frac{p-\varepsilon_0}{p-\be} < 1. 
 \end{equation}
 These choice of exponents are admissible under the additional restriction that $\varepsilon_0 \leq \max \left\{\frac{4}{n+2}, \frac14\right\}$ and  $\be_0$ is small such that $\frac{(p-\varepsilon_0)^2}{p} \geq \frac{n}{2} (\varepsilon_0-\be_0)$ holds. Thus applying Lemma \ref{interpolation_capacity_version_main}, we get
 \begin{equation}
 \label{7.14}
\fiint_Q \lbr \frac{|u|}{\rho} \rbr^{p-\be} \ dz \apprle \hint_{\tm} \lbr \hint_B |\nabla u|^{p-\varepsilon_0} \ dx \rbr \lbr \hint_B  \lbr \frac{|u|}{\rho} \rbr^{\tr} \ dx \rbr^{\frac{\varepsilon_0 - \be}{2}\frac{p}{p-\be}} dt.
\end{equation}
 Making use Lemma \ref{bound_g_x_t} along with \eqref{hypothesis_2}, we get
 \begin{equation}
\label{7.15}
 \fiint_Q \tilde{g}(z)^{p-\be} \ dz \apprle  \fiint_{\al Q} \lbr |\nabla u|+|\th| +\al_0 \rbr ^{{{p-\be}}} \ dz \apprle \al_0^{p-\be}.
\end{equation}
For almost every $t\in \al I$, using H\"older's inequality and  \eqref{def_J}, we also get 
\begin{equation}\label{7.16}
\begin{array}{ll}
 \hint_B  \lbr \frac{|u(x,t)|}{\rho} \rbr^{\tr} \ dx  & \apprle J^{\frac{\tr}{2}} \lbr \hint_B \tilde{g}(x,t)^{p-\be} \ dx \rbr^{\frac{2-\tr}{2}}.
 \end{array}
\end{equation}
Substituting \eqref{7.16} into \eqref{7.14}, followed by an application of H\"older's inequality with exponents $\frac{q}{p-\varepsilon_0}$ and $\tilde{q_0}:=\frac{q}{q-p+\varepsilon_0}$, we get
\begin{equation*}
\label{4.44}
 \fiint_Q \lbr \frac{|u|}{\rho} \rbr^{p-\be} \ dz \apprle J^{\frac{\varepsilon_0-\be}{2}}  \lbr \fiint_Q |\nabla u|^q \ dz \rbr^{\frac{p-\varepsilon_0}{q}}  \lbr \hint_{\tm} \lbr \hint_B  \tilde{g}(x,t)^{p-\be} \ dx \rbr^{\tilde{q_0} \frac{(2-\tr)(\varepsilon_0 - \be)p}{2(p-\be)}} dt\rbr^{\frac{1}{\tilde{q_0}}}.
\end{equation*}
We further restrict $\be_0 \leq \frac{p-\ve_0}{p}$, which ensures $\tilde{q_0} \frac{(2-\tr)(\varepsilon_0 - \be)p}{2(p-\be)} = q_0\frac{\be(\varepsilon_0-\be)}{2(p-\be)}<1$ holds.
Thus Jensen's inequality for concave functions becomes applicable and hence we get
\begin{equation}
\label{4.45}
 \fiint_Q \lbr \frac{|u|}{\rho} \rbr^{p-\be} \ dz \apprle J^{\frac{\varepsilon_0-\be}{2}}  \lbr \fiint_Q |\nabla u|^q \ dz \rbr^{\frac{p-\varepsilon_0}{q}}  \lbr  \fiint_Q  \tilde{g}(z)^{p-\be} \ dz \rbr^{\frac{\be(\varepsilon_0-\be)}{2(p-\be)}}.
\end{equation}
Using \eqref{7.15} in \eqref{4.45}, we get 
\begin{equation*}
\label{4.46}
 \fiint_Q \lbr \frac{|u|}{\rho} \rbr^{p-\be} \ dz \apprle (\al_0^{\be} J)^{\frac{\varepsilon_0-\be}{2}}  \lbr \fiint_Q |\nabla u|^q \ dz \rbr^{\frac{p-\varepsilon_0}{q}}  .
\end{equation*}

 \item[Proof of \eqref{lemma4.3.2}:] Note that in this case, we have $p-\ve_0 \geq 2$, hence we make the following choice of exponents to apply Lemma \ref{interpolation_capacity_version_main}:
 \begin{equation}
 \label{exp_ch_3}
\sig = 2, \quad  \tr = \frac{2(p-\be)}{p}, \quad \vartheta = p-\varepsilon_0.
 \end{equation}
 Let us now choose $\tga \in(0,1)$ so as to satisfy the following restrictions:
\begin{equation}
- \frac{n}{2} \leq \tga\lbr 1 - \frac{n}{p-\varepsilon_0} \rbr  -  \frac{(1-\tga)np}{2(p-\be)} \quad \text{and} \quad 
\frac{q (1-\tga) \be}{(q-2\tga)(p-\be)} < 1. \label{7.24}
\end{equation}
This choice of $\tga$ in \eqref{7.24} is possible, due to the following reasons.
\begin{itemize}
 \item \eqref{7.24} holds  for $\tga=1$, since $\frac{n}{p-\varepsilon_0} < 1 + \frac{n}{2}$, as we are in the case $p-\varepsilon_0 \geq 2$.
 \item Since we assumed $p-\varepsilon_0 \geq 2$, this means that $q-2\tga >0$ for $\tga < 1$. This follows from the fact that  $ q \geq p-\varepsilon_0 \geq 2 \geq 2\tga \in (0,2).$ 
\end{itemize}
Thus we can apply Lemma \ref{interpolation_capacity_version_main} with \eqref{exp_ch_3} and $\tal=\tga$ as the choice of exponents, to get
\begin{equation}
 \label{7.20}
\fiint_Q \lbr \frac{|u|}{\rho} \rbr^{2} \ dz \apprle \hint_{\tm} \lbr \hint_B |\nabla u|^{p-\varepsilon_0} \ dx \rbr^{\frac{2\tga}{p-\varepsilon_0}} \lbr \hint_B  \lbr \frac{|u|}{\rho} \rbr^{\tr} \ dx \rbr^{\frac{2(1-\tga)}{r}} dt.
\end{equation}
Since $\tr<2$, we can substitute \eqref{7.16} into \eqref{7.20} to obtain
\begin{equation}
 \label{7.22}
\fiint_Q \lbr \frac{|u|}{\rho} \rbr^{2} \ dz \apprle J^{1-\tga} \hint_{\tm} \lbr \hint_B |\nabla u|^{p-\varepsilon_0} \ dx \rbr^{\frac{2\tga}{p-\varepsilon_0}} \lbr \hint_B \tilde{g}(x,t)^{p-\be} \ dx \rbr^{\frac{(2-\tr)(1-\tga)}{\tr}} dt.
\end{equation}
We will now apply H\"older inequality with exponents $\frac{q}{2\tga}$ and $\frac{q}{q-2\tga}$, and then make use of \eqref{7.24}, which will enable us to apply Jensen's inequality. Thus we get
\begin{equation}
 \label{7.23}
 \begin{array}{ll}
\fiint_Q \lbr \frac{|u|}{\rho} \rbr^{2} \ dz 
& \apprle J^{1-\tga}\lbr  \fiint_Q |\nabla u|^{q} \ dz\rbr^{\frac{2\tga}{q}} \lbr  \fiint_Q \tilde{g}(x,t)^{p-\be} \ dz \rbr^{\frac{(2-\tr)(1-\tga)}{\tr}}.
\end{array}
\end{equation}
Further making use of \eqref{7.15} and \eqref{7.23}, we get
\begin{equation*}
 \label{7.25}
 \begin{array}{ll}
\fiint_Q \lbr \frac{|u|}{\rho} \rbr^{2} \ dz & \apprle (\al_0^{\be} J)^{1-\tga}\lbr  \fiint_Q |\nabla u|^{q} \ dz \rbr^{\frac{2\tga}{q}} .
\end{array}
\end{equation*}

 \item[Proof of \eqref{lemma4.3.3}:]
 In this case, we  take
\begin{equation*}
 \label{exp_ch_4}
\sig = 2, \quad  \tr = \frac{2(p-\be)}{p}, \quad \vartheta = p-\varepsilon_0 \qquad \text{and} \qquad \tal = \frac{p-\varepsilon_0}{2} < 1. 
 \end{equation*}
Proceeding as in the proof of \eqref{lemma4.3.1}, we get
\begin{equation*}
 \label{7.26}
 \begin{array}{ll}
\fiint_Q \lbr \frac{|u|}{\rho} \rbr^{2} \ dz & \apprle (\al_0^{\be} J)^{\frac{2-p+\varepsilon_0}{2}}\lbr  \fiint_Q |\nabla u|^{q} \ dz \rbr^{\frac{p-\varepsilon_0}{q}} .
\end{array}
\end{equation*}
\end{description}

\end{proof}

\begin{lemma}
\label{lemma6.2}
 Let $\frac{2n}{n+2} < p < 2+\be$.  Then under the assumptions of Lemma \ref{Caccioppoli}, there holds
 \[
  \fiint_Q \lbr \frac{|u|}{\rho} \rbr^2 \ dz \leq  C_{(n,p,q,\lamot,b_0,\ka)}\al_0^2. 
 \]
\end{lemma}
\begin{proof}
  Let $1 \leq \tal_1 < \tal_2 \leq 8$ be any two fixed numbers. We can then apply Lemma \ref{lemma7.4} to get $\tga \in (0,1)$ such that the following estimate holds:
  \begin{equation}
  \label{eq7.31}
\fiint_{ \tal_1 Q} \lbr \frac{|u|}{\rho} \rbr^{2} \ dz  \apprle (\al_0^{\be} J)^{1-\tga}\lbr  \fiint_{ \tal_2Q} |\nabla u|^{q} \ dz\rbr^{\frac{2\tga}{q}}, 
\end{equation}
where $J:= \sup_{t \in \tal_1\tm} \hint_{\tal_1 B} \lbr \frac{|u|}{\rho} \rbr^2 \tilde{g}(x,t)^{-\be} \ dx $ with 
$$\tilde{g}(z) := \max\left\{ \mm( |\nabla \tu|^q )^{\frac{1}{q}}(z), \mm( \lbr |\nabla u|^q+|\th|^q\rbr \lsb{\chi}{\tal_2 Q})^{\frac{1}{q}}(z), \al_0\right\}$$ 
where we have set $\tu(x,t) = u(x,t) \eta(x)\zeta(t)$ for  $\eta(x) \zeta(t) \in C_c^{\infty}(\tal_2 Q)$ with $\eta(x) \zeta(t) \equiv 1$ for  $(x,t) \in \tal_1Q$ instead of \eqref{cut_off_function}.

We can now make  use of Lemma \ref{Caccioppoli} and  obtain the following estimate for some $ \be \in (0,\be_0]$:
 \begin{equation}
  \label{conclusion_1_rescaled}
  \begin{array}{ll}
  \al_0^{p-\be} + \al_0^{p-2} \sup_{t \in\tal_1\tm} \hint_{\tal_1B} \left| \frac{u(x,t)}{\rho} \right|^2 \tilde{g}(x,t)^{-\be} \ dx &\leq \cac \fiint_{\tal_2Q}  \al_0^{p-2-\be} \left| \frac{u}{(\tal_2-\tal_1)\rho} \right|^2\ dz \\
  &\quad +\fiint_{\tal_2Q}  \left[\left| \frac{u}{(\tal_2-\tal_1)\rho} \right|^{p-\be} + |\th|^{p-\be}\right] \ dz.
  \end{array}
 \end{equation}

In particular, \eqref{conclusion_1_rescaled} implies
\begin{equation}
 \label{eq7.32}
\al_0^{\be} J \apprle  \fiint_{\tal_2Q}\left[   \left| \frac{u}{(\tal_2-\tal_1)\rho} \right|^2  + \al_0^{2-p+\be} \left| \frac{u}{(\tal_2-\tal_1)\rho} \right|^{p-\be} + |\th|^{p-\be} \right]dz.
\end{equation}

We now apply Young's inequality to the middle term on the right hand side of \eqref{eq7.32} (possible since $p-\be < 2$) and make use of \eqref{hypothesis_1}, to get 
\begin{equation}
 \label{eq7.34}
 \al_0^{\be} J \apprle  \fiint_{\tal_2Q}   \left| \frac{u}{(\tal_2-\tal_1)\rho} \right|^2 \ dz+  \al_0^2.
\end{equation}

 Substituting \eqref{eq7.34} into \eqref{eq7.31} followed by 
applying Young's inequality with exponent $\frac{1}{1-\tga} $ and $\frac{1}{\tga}$, along with  using  \eqref{hypothesis_1} and the restriction $\tal_1,\tal_2 \in [1,8]$, we get
\begin{equation}
  \label{eq7.36}
  \begin{array}{ll}
   \fiint_{ \tal_1 Q} \lbr \frac{u}{\rho} \rbr^{2} \ dz  
   & \leq \frac{1}{2} \fiint_{\tal_2 Q}   \left( \frac{u}{\rho} \right)^2   dz + C\al_0^2 (\tal_2-\tal_1)^{-\frac{2(1-\tal)}{\tal}} + C \al_0^2 (\tal_2 - \tal_1)^2 \\
   &\leq  \frac{1}{2} \fiint_{\tal_2 Q}   \left( \frac{u}{\rho} \right)^2   dz + C \al_0^2 (\tal_2-\tal_1)^{-\frac{2(1-\tal)}{\tal}} + C \al_0^2.
  \end{array}
 \end{equation}

 If we set $\varpi(a) := \fiint_{a Q}   \left| u \right|^2   dz$, then we see that for $\tal_1=:a < b:=\tal_2$,  \eqref{eq7.36} can be rewritten as 
 \[
  \varpi(a) \leq \frac12 \varpi(b) + C\al_0^2 \rho^2 (b-a)^{-\frac{2(1-\tal)}{\tal}} + C\al_0^2 \rho^2.
 \]
We can now apply Lemma \ref{iteration_lemma} along with the hypothesis $ \tal_1, \tal_2 \in [1,8]$ to get 
\[
 \fiint_{ Q}   \left| u \right|^2   dz \apprle  \al_0^2 \rho^2.
\]

This completes the proof of the lemma. 
\end{proof}

\subsection{Reverse H\"older inequality}

\begin{lemma}
 \label{reverse_holder_inequality}
 Let $\ka \geq 1$.  Then there exists  $\be_0(n,p,q,\lamot,b_0,\ka)$ and $C{(n,p,q,\lamot,b_0,\ka)}$  such that the following holds: Let  $u \in L^2(-T,T;L^2(\Om))\cap L^{p-\be}(-T,T;W_0^{1,p-\be}(\Om))$ for some $\be \in (0,\be_0)$ be a very weak solution of \eqref{main}. Furthermore, let  $Q=Q_{\rho,s}$ be a parabolic cylinder with $s = \al_0^{2-p} \rho^2$ as in \eqref{def_s} be given. Suppose the following bounds hold:
 \begin{gather}
  \frac{\al_0^{p-\be}}{\ka} \leq \fiint_Q (|\nabla u|+|\th|)^{p-\be} \ dz \qquad \text{and} \qquad   \fiint_{8^3Q} (|\nabla u|+|\th|)^{p-\be} \ dz \leq \ka \al_0^{p-\be}. \label{hyp_2}
 \end{gather}
Then the following conclusion holds:
\begin{equation*}
 \label{conclusion_2}
 \al_0^{p-\be} \leq C \lbr \fiint_{8Q} |\nabla u|^q \ dz \rbr^{\frac{p-\be}{q}} + \fiint_{8Q} |\th|^{p-\be} \ dz.
\end{equation*}
\end{lemma}

\begin{proof}
 Since all the hypothesis of  Lemma \ref{Caccioppoli} (note the scaling factor is $8^2$ instead of $8$ as in Lemma \ref{cac1}) is satisfied, we  have
 \begin{equation}
\label{eq7.41}
  \al_0^{p-\be}  \apprle \fiint_{8^2Q}\lbr  \al_0^{p-2-\be} \left| \frac{u}{\rho} \right|^2 + \left| \frac{u}{\rho} \right|^{p-\be} + |\th|^{p-\be}\rbr dz.
 \end{equation}

 Let $\tilde{g}(z) := \max\left\{ \mm( |\nabla \tu|^q )^{\frac{1}{q}}(z), \mm(( |\nabla u|^q |\th|^q)\lsb{\chi}{8^2 Q})^{\frac{1}{q}}(z),\al_0\right\}$, where $\tu(x,t):= u(x,t) \eta(x) \zeta(t)$ with $\eta(x) \zeta(t) \in C_c^{\infty}(8^2Q)$ satisfying $\eta(x) \zeta(t) \equiv 1$ for $(x,t) \in 8Q$.  Let us set 
 \begin{equation*}
  \label{def_J_new}
  J= \sup_{t \in8\tm} \hint_{8B} \left| \frac{u(x,t)}{\rho} \right|^2 \tilde{g}(x,t)^{-\be} \ dx.
 \end{equation*}

Applying a rescaled version of Lemma \ref{Caccioppoli}, we have
\begin{equation}
 \label{J_bound}
 \al_0^{p-2} J \apprle \fiint_{8^2Q}\lbr  \al_0^{p-2-\be} \left| \frac{u}{\rho} \right|^2 + \left| \frac{u}{\rho} \right|^{p-\be}+|\th|^{p-\be}\rbr dz.
\end{equation}

 We now have the following three bounds:
 \begin{itemize}
  \item In the case $2 \geq p-\be$, we can apply Lemma \ref{lemma6.2} (this is possible since \eqref{hyp_2} is on $8^3Q$) to get 
  \begin{equation}
   \label{eest_1}
   \fiint_{8^2Q} \lbr \frac{|u|}{\rho} \rbr^2 \ dz \apprle \al_0^2.
  \end{equation}
  \item In the case $2 \leq p-\be$, for  $\sig \in \{2,p-\be\}$,  we can apply H\"older's inequality and Theorem \ref{sobolev-poincare}  along with \eqref{hyp_2}, to get
\begin{equation}
 \label{eest_2}
 \fiint_{8Q} \lbr \frac{|u|}{\rho} \rbr^{\sig} \ dz  \leq \lbr \fiint_{8Q} \lbr \frac{|u|}{\rho} \rbr^{p-\be} \ dz\rbr^{\frac{\sig}{p-\be}} \apprle  \lbr \fiint_{8Q} |\nabla u|^{p-\be} \ dz\rbr^{\frac{\sig}{p-\be}} {\apprle} \al_0^{\sig}.
\end{equation}
%

 \end{itemize}
 
 Substituting \eqref{eest_1} and \eqref{eest_2}  into \eqref{J_bound} gives
  \begin{equation}
\label{eq7.42}
 \al^{\be} J \apprle \al_0^2. 
 \end{equation}

Substituting \eqref{eq7.42} into a rescaled version of Lemma \ref{lemma7.4} and applying Young's inequality followed by substituting the resulting estimate into \eqref{eq7.41}, we get
%
%
%
%
%
%
%
%
%
%
\begin{equation*}
 \begin{array}{@{}r@{}c@{}l@{}}
 \al_0^{p-\be} & \apprle &\al_0^{p-2-\be}\al_0^{2(1-\tga)}\lbr  \fiint_{ 8^2Q} |\nabla u|^{q} \ dz\rbr^{\frac{2\tga}{q}} + \al_0^{\frac{\varepsilon_0-\be}{2}}  \lbr \fiint_{8^2 Q} |\nabla u|^q \ dz \rbr^{\frac{p-\varepsilon_0}{q}}  + \fiint_{8^2Q} |\th|^{p-\be} \ dz \\
 &\apprle & \varepsilon_1 \al_0^{p-\be} + C(\varepsilon_1) \lbr  \fiint_{ 8^2Q} |\nabla u|^{q} \ dz\rbr^{\frac{p-\be}{q}} +  C(\varepsilon_1) \lbr  \fiint_{ 8^2Q} |\nabla u|^{q} \ dz \rbr^{\frac{p-\be}{q}} +  \fiint_{8^2Q} |\th|^{p-\be} \ dz .
 \end{array}
\end{equation*}
Choosing $\varepsilon_1$ small enough, we get the desired conclusion. This completes the proof of the lemma. 
\end{proof}

\section{Proof of the main result}\label{main_estimates}

The proof of Theorem \ref{main_theorem} now follows exactly as in proof of \cite[Theorem 2.1]{Bog}. For the sake of completeness, we present the details below. Without loss of generality, we can assume $\rho = 1$ and $z_0 = (0,0) \in \pa \Om \times (-T,T)$. We take $Q_1 := Q_{(0,0)}(1,1) $ and  $Q_2 := Q_{(0,0)}(2,2^2)$
and for any  $z \in Q_2$, define the parabolic distance of $z$ to $\pa Q_2$ by 
\begin{gather*}
d_p(z):= \inf_{ \tilde{z} \in \RR^{n+1} \setminus Q_2} \min \{ |x-\tilde{x}| , \sqrt{|t-\tilde{t}|} \} \label{par_dist_boundary}. 
\end{gather*}

Furthermore, let $\be_0$ be the constant from Lemma \ref{reverse_holder_inequality} and $\be \in (0,\be_0]$ be such that $p-\be > \frac{2n}{n+2}$. For $z \in Q_2$, let us define the following function:
\begin{gather}
  \psi(z) := |\nabla u(z)|+ |\th(z)| \quad \text{and} \quad f(z) := d_p^{\al} (z) \psi(z) \quad \text{with} \ \al:= \frac{n+2}{d}, \label{def_g_f} 
\end{gather}
where $d$ is as defined in \eqref{de_d}. 
Finally  we define $\al_0$ to be
\begin{gather}
 \al_0^d := \fiint_{Q_2} \psi(z)^{p-\be} \ dz + 1. \label{def_al_0}
\end{gather}
Let $\la$ be any number such that 
\begin{gather}
 \la \geq b^{\frac{1}{d}} \al_0 \qquad \text{where} \ b:= 2^{10(n+2)}.\label{5.5}
\end{gather}
Now suppose that $\mfz \in Q_2$ with $f(\mfz)>0$, then let us denote the parabolic distance of $\mfz$ to $\pa Q_2$ by $r_{\mfz} := d_p(\mfz) $
and define the intrinsic scaling factor as
\begin{gather}
 \ga = \ga(\mfz) := (r_{\mfz}^{-\al} \la )^{2-p} = (d_p(\mfz)^{-\al} \la)^{2-p} \label{intrinsic_scaling}.
\end{gather}

In order to prove higher integrability, we want to apply Lemma \ref{gehring_lemma}. So the rest of the proof is devoted to ensuring that all the hypotheses of Lemma \ref{gehring_lemma} are satisfied.
\begin{description}[leftmargin=*]
 \item[Case $p \geq 2$:]  Let us note that $r_{\mfz}^{\al} \leq 2^{\al} \leq b^{\frac1{d}} \al_0 \leq \la$, which implies $\ga = (r_{\mfz}^{-\al} \la)^{2-p} \leq 1$. Hence we shall consider intrinsic cylinders of the type $Q_{\mfz}(R, \ga R^2)$ with $0 < R \leq r_{\mfz}$.

 In order to apply Lemma \ref{gehring_lemma}, we need to find an appropriate intrinsic parabolic cylinder around $\mfz$ on which all the hypotheses of Lemma \ref{reverse_holder_inequality} are satisfied. 
 In order to do this, let us first take $R$ such that $r_{\mfz} \leq 2^9 R < 2^9 r_{\mfz}$. In this case, there holds:
 \begin{equation}
  \label{5.8}
  \begin{array}{ll}
   \fiint_{Q_{\mfz}(R,\ga R^2)} \psi(z)^{p-\be} \ dz & \leq \frac{|Q_2|}{|{Q_{\mfz}(R,\ga R^2)}|} \fiint_{Q_{2}} \psi(z)^{p-\be} \ dz \\
   &\overset{\text{\eqref{def_al_0}}}{=} \frac{2^{n+2}}{R^{n+2} \ga}\al_0^d
   \overset{\text{\eqref{5.5}}}{\leq} \frac{2^{10(n+2)}}{r_{\mfz}^{n+2}} \frac{\la^d}{b} 
    \overset{\text{\eqref{intrinsic_scaling}}}{=} (r_{\mfz}^{-\al} \la)^{p-\be}.
  \end{array}
 \end{equation}

 Furthermore, by Lebesgue differentiation theorem, for every $\mfz \in Q_2$ with $f(\mfz) > \la$, there holds
 \begin{equation}
  \label{5.9}
  \lim_{r\searrow 0} \fiint_{Q_{\mfz}(r,\ga r^2)} \psi(z)^{p-\be} \ dz = \psi(\mfz)^{p-\be} \overset{\text{\eqref{def_g_f}\eqref{intrinsic_scaling}}}{=} (r_{\mfz}^{-\al} f(\mfz))^{p-\be} >  (r_{\mfz}^{-\al} \la)^{p-\be} .
 \end{equation}

 Thus from \eqref{5.8} and \eqref{5.9}, we observe that there should exist  $\rho \in \lbr0,\frac{r_{\mfz}}{2^9}\rbr$, such that 
 \begin{equation*}
  \label{def_rho}
  \begin{array}{c}
  \fiint_{Q_{\mfz}(\rho, \ga \rho^2)} \psi(z)^{p-\be} \ dz = (r_{\mfz}^{-\al} \la)^{p-\be}, \\
  \fiint_{Q_{\mfz}(R, \ga R^2)} \psi(z)^{p-\be} \ dz \leq (r_{\mfz}^{-\al} \la)^{p-\be},  \quad \ \forall\   R \in [\rho, r_{\mfz}].
  \end{array}
 \end{equation*}

  We now set $Q:= Q_{\mfz}(\rho, \ga \rho^2)$, then $2^9Q \subset Q_2$, thus all hypotheses of Lemma \ref{reverse_holder_inequality} are satisfied with $(r_{\mfz}^{-\al} \la, 1)$ instead of $(\al_0,\kappa)$, i.e., the following holds:
  \begin{equation}
  \label{5.11}
   (r_{\mfz}^{-\al} \la)^{p-\be}  = \fiint_Q \psi(z)^{p-\be} \ dz \qquad \text{and} \qquad \fiint_{2^8Q} \psi(z)^{p-\be} \ dz \apprle (r_{\mfz}^{-\al} \la)^{p-\be}.
  \end{equation}

Thus we can apply  Lemma \ref{reverse_holder_inequality}, which gives
\begin{equation}
 \label{5.12}
 (r_{\mfz}^{-\al} \la)^{p-\be} \apprle \lbr \fiint_{2^8Q} |\nabla u|^q\ \ dz \rbr^{\frac{p-\be}{q}} + \fiint_{2^8Q} |\th|^{p-\be} \ dz.
\end{equation}

 Since $2^9 \rho \leq r_{\mfz}$ and $\ga \leq 1$, we also have for all $z \in 2^8Q$ that 
 \begin{gather}
  d_p(z) \leq \min\{ r_{\mfz} + 2^8 \rho, \sqrt{r_{\mfz}^2 + \ga (2^8 \rho)^2} \} \leq \frac32 r_{\mfz}, \label{up_bound_dpz} \\
  d_p(z) \geq \min\{ r_{\mfz} - 2^8 \rho, \sqrt{r_{\mfz}^2 - \ga (2^8 \rho)^2} \} \geq \frac12 r_{\mfz}. \label{low_bound_dpz}
 \end{gather}

 Now substituting \eqref{up_bound_dpz} and \eqref{low_bound_dpz} into \eqref{def_g_f}, we find
 \begin{equation}
  \label{5.15}
  c^{-1} f(z) \leq r_{\mfz}^{\al} \psi(z) \leq c f(z), \quad \forall \ z \in 2^8Q \ \text{ with  } \ c = c(n)>1.
 \end{equation}

 We now claim the following estimate holds:
 \begin{equation}
  \label{5.16}
  \la^{p-\be} \overset{\text{\redlabel{5.16.a}{a}}}{\apprle} \fiint_{2^8Q} f(z)^{p-\be} \ dz \overset{\text{\redlabel{5.16.b}{b}}}{\apprle} \lbr \fiint_{2^8Q} f(z)^q \ dz \rbr^{\frac{p-\be}{q}} + \fiint_{2^8Q} (r_{\mfz}^{\al} |\th(z)|)^{p-\be} \ dz \overset{\text{\redlabel{5.16.c}{c}}}{\apprle} \la^{p-\be}.
 \end{equation}
 \begin{description}[leftmargin=*]
  \item[Estimate \redref{5.16.a}{a}:] This follows easily by making use of \eqref{5.11} and \eqref{5.15} and subsequently enlarging the parabolic cylinder $Q$. 
  \item[Estimate \redref{5.16.b}{b}:] This is obtained by the following chain of estimates:
  \begin{equation*}
   \begin{array}{lll}
    \fiint_{2^8Q} f(z)^{p-\be} \ dz & \overset{\text{\eqref{5.15}}}{\apprle} & r_{\mfz}^{\al(p-\be)} \fiint_{2^8Q} \psi(z)^{p-\be} \ dz \overset{\text{\eqref{5.11}}}{\apprle} \la^{p-\be} \\
    & \overset{\text{\eqref{5.12}}}{\apprle} &  r_{\mfz}^{\al(p-\be)} \lbr \fiint_{2^8Q} |\nabla u|^q \ dz \rbr^{\frac{p-\be}{q}}  + r_{\mfz}^{\al(p-\be)} \fiint_{2^8Q} |\th(z)|^{p-\be} \ dz \\
    & \overset{\text{\eqref{5.15}}}{\apprle} & \lbr \fiint_{2^8Q} f(z)^q \ dz \rbr^{\frac{p-\be}{q}} + \fiint_{2^8Q} (r_{\mfz}^{\al} |\th(z)|)^{p-\be} \ dz.
   \end{array}
  \end{equation*}

  \item[Estimate \redref{5.16.c}{c}:] This follows by applying Jensen's inequality (since $q < p-\be$) along with the bound \eqref{5.11} to get:
  \begin{equation*}
   \lbr \fiint_{2^8Q} f(z)^q \ dz \rbr^{\frac{p-\be}{q}} + \fiint_{2^8Q} (r_{\mfz}^{\al} |\th(z)|)^{p-\be} \ dz \overset{\text{\eqref{5.15}}}{\apprle}  \fiint_{2^8Q} f(z)^{p-\be} \ dz \overset{\text{\eqref{5.11}}}{\apprle} \la^{p-\be}.
  \end{equation*}
 \end{description}

 Thus  \eqref{5.16} holds and  as a consequence, we can apply Lemma \ref{gehring_lemma} over $2^8Q$ to see that  for any $\be \in (0,\be_0]$, there exists  $\de_0 = \de_0(n,p,b_0,\varepsilon_0,\lamot)>0$ such that $f \in L_{loc}^{p-\be+\de_1}(Q_2)$ with $\de_1 = \min\{ \de_0, \tilde{p}-p+\be\}$. This is quantified by the estimate:
\begin{equation*}
 \label{5.19}
 \iint_{Q_2} f(z)^{p-\be+\de} \ dz \apprle \al_0^{\de} \iint_{Q_2} f(z)^{p-\be} \ dz + \iint_{Q_2} (r_{\mfz}^{\al} |\th(z)|)^{p-\be+\de} \ dz\qquad \forall \de \in (0,\de_1].
\end{equation*}

By iterating the previous arguments, for any $\be \in (0,\varepsilon_{\text{geh}}]$ where $\varepsilon_{\text{geh}}>0$ is the gain in higher integrability coming from Lemma \ref{gehring_lemma}, we obtain the bound
\begin{equation}
 \label{5.20}
 \iint_{Q_2} f(z)^{p} \ dz \apprle \al_0^{\be} \iint_{Q_2} f(z)^{p-\be} \ dz + \fiint_{Q_2} |\th(z)|^p \ dz.
\end{equation}

For any $z \in Q_1$, we have $d_p(z) \geq \min \{ 1 , \sqrt{3}\} \geq 1$, $\frac{|Q_2|}{|Q_1|} = C(n)$, which implies the following bounds  hold:
\begin{gather}
 |\nabla u (z)| \leq \psi(z) \leq f(z) \qquad \forall\  z \in Q_1, \label{5.21}\\
 \fiint_{Q_1} |\nabla u|^{p+\be} \ dz \apprle_n \fiint_{Q_2} f(z)^p \ dz ,\label{5.22}\\
 f(z) \leq 2^{\al} \psi(z) \qquad \forall z \in Q_2,  \quad \text{since} \ d_p(z) \leq 2.\label{5.23}
\end{gather}

Using \eqref{5.21}, \eqref{5.22}, \eqref{5.23} along with \eqref{5.20} and making use of \eqref{def_g_f} and \eqref{def_al_0}, we get
\begin{equation*}
 \fiint_{Q_1} |\nabla u|^p \ dz \apprle \al_0^{\be} \fiint_{Q_2} \psi(z)^{p-\be} \ dz  + \fiint_{Q_2} \th^p \ dz \apprle \lbr \fiint_{Q_2} \lbr |\nabla u| + |\th| \rbr^{p-\be} \ dz \rbr^{1+\frac{\be}{d}} + \fiint_{Q_2} (1+\th^p) \ dz.
\end{equation*}
This proves the asserted estimate.

 \item[Case $\frac{2n}{n+2} <p < 2$:] The basic change with respect to the case $p\geq 2$ is that, we now switch to the sub-quadratic scaling, i.e., we consider intrinsic cylinders of the type $Q_{\mfz}(\ga^{-\frac12}R,R^2)$. 
 
 The parameter $\al_0$ is still given by \eqref{def_al_0} and $\la$ is chosen as in \eqref{5.5} and $\ga$ is again given as in \eqref{intrinsic_scaling}, where $\mfz \in Q_2$ with $f(z) > \la$. But in contrast to $p\geq 2$ case, we have $\ga = (r_{\mfz}^{-\al} \la )^{2-p} \geq 1$. Hence for $R \in (0,r_{\mfz})$, we have $Q_{\mfz}(\ga^{-\frac12}R,R^2)\subset Q_2$. Now once again, in order to apply Lemma \ref{gehring_lemma}, we need to find a suitable intrinsic parabolic cylinder around $\mfz$, which enables us to apply Lemma \ref{reverse_holder_inequality}.  We observe, from the definition \eqref{def_g_f} that $n+2 = \al d$ (recall $d$ is defined as in \eqref{de_d}) and $(2-p)\frac{n}{2} + d = p-\be$, which gives
 \begin{equation*}
  \begin{array}{ll}
   \fiint_{Q_{\mfz}(\ga^{-\frac12}R, R^2)} \psi(z)^{p-\be} \ dz & \apprle \frac{|Q_2|}{|Q_{\mfz}(\ga^{-\frac12}R, R^2)|}\fiint_{Q_2} \psi(z)^{p-\be} \ dz  = \frac{2^{n+2}}{R^{n+2} \ga^{-\frac{n}{2}}} \al_0^d  
   \apprle (r_{\mfz}^{-\al} \la )^{p-\be}. 
  \end{array}
 \end{equation*}
 Now we can continue as in the $p \geq 2$ case to obtain the desired conclusion. 
\end{description}
This completes the  proof. \hfill \qed

\section*{References}


\end{document}